\documentclass[11pt]{amsart}

\usepackage[T1]{fontenc}
\usepackage{lmodern}
\usepackage{microtype}
\usepackage{amsmath,amssymb,amsthm,mathtools,mathrsfs}
\usepackage{enumitem}
\usepackage{aliascnt}
\usepackage{xcolor}
\usepackage{geometry}
\usepackage{hyperref}
\usepackage[nameinlink,capitalize,noabbrev]{cleveref}
\allowdisplaybreaks

\hypersetup{
  colorlinks=true,
  linkcolor=blue!55!black,
  citecolor=blue!55!black,
  urlcolor=blue!55!black,
  pdftitle={Cyclotomic Newton Expansions and a Rank-Uniform Integer-Valued Newton Completion},
  pdfauthor={Honghuai Fang and Tian Zhou},
  pdfsubject={Cyclotomic expansions and a rank-uniform Newton completion for symmetric-power knot invariants},
  pdfkeywords={SU(n) invariant, cyclotomic expansion, Newton interpolation, HOMFLY-PT invariant, quantum integer-valued polynomial, Habiro ring, Newton inverse limit}
}

\newtheorem{theorem}{Theorem}[section]
\newaliascnt{proposition}{theorem}
\newtheorem{proposition}[proposition]{Proposition}
\aliascntresetthe{proposition}
\newaliascnt{lemma}{theorem}
\newtheorem{lemma}[lemma]{Lemma}
\aliascntresetthe{lemma}
\newaliascnt{corollary}{theorem}
\newtheorem{corollary}[corollary]{Corollary}
\aliascntresetthe{corollary}
\theoremstyle{definition}
\newaliascnt{definition}{theorem}
\newtheorem{definition}[definition]{Definition}
\aliascntresetthe{definition}
\newaliascnt{example}{theorem}
\newtheorem{example}[example]{Example}
\aliascntresetthe{example}
\newaliascnt{remark}{theorem}
\newtheorem{remark}[remark]{Remark}
\aliascntresetthe{remark}

\newcommand{\Z}{\mathbb Z}
\newcommand{\Q}{\mathbb Q}
\newcommand{\ord}{\operatorname{ord}}
\newcommand{\hc}{\operatorname{hc}}
\newcommand{\qbinom}[2]{\begin{bmatrix}#1\\#2\end{bmatrix}_{q}}

\newcommand{\val}{\operatorname{val}}
\newcommand{\id}{\operatorname{id}}
\newcommand{\Hab}{\widehat{\Z[q]}}
\newcommand{\HH}{\mathcal H}
\newcommand{\mfrakq}{\mathfrak q}
\newcommand{\cO}{\mathcal O}
\newcommand{\cZ}{\mathcal Z}

\title[Cyclotomic Newton expansions and rank-uniform completion]
{Cyclotomic Newton Expansions and a Rank-Uniform Integer-Valued Newton Completion}
\author{Honghuai Fang}
\author{Tian Zhou}

\date{}


\begin{document}

\begin{abstract}
Let $J_r^{SU(n)}(K;q)$ denote the reduced $SU(n)$ quantum invariant
of a zero-framed knot $K$, colored by the $r$th symmetric power of the
defining representation and normalized to be $1$ for the unknot. For
every fixed $n\ge2$ we prove the Chen--Liu--Zhu cyclotomic expansion
conjecture: there are unique coefficients
$H_k^{(n)}(K;q)\in\Z[q^{\pm1}]$ such that
\[
 J_r^{SU(n)}(K;q)=\sum_{k=0}^{r}
 \left(\prod_{i=0}^{k-1}\{r-i\}\{r+n+i\}\right)
 H_k^{(n)}(K;q),
\]
where $\{m\}=q^m-q^{-m}$. The finite dual interpolation formula of Beliakova--Gorsky gives an
integral one-sided factorial expansion.  After identifying their reduced
scalar with the Habiro--L\^e convention, we restrict the completed center to
one-row colors.  Completed Harish--Chandra reflection then yields inversion
symmetry in the variable $z$, and integral descent through $X=z+z^{-1}$
converts the one-sided expansion into the two-sided Newton basis.
Cyclotomic-local interpolation and a UFD denominator-removal argument prove
Laurent integrality of the Newton coefficients.

We also determine a natural coefficient ring for a rank-uniform expansion.
For every zero-framed knot there are unique Laurent differential coefficients
$G_k(K;A,q)\in\Z[A^{\pm1},q^{\pm1}]$.  The associated Newton coefficients
are Laurent polynomials in $A$ over $\Q(q)$ whose values at every geometric
node $A=q^n$, $n\ge2$, lie in $\Z[q^{\pm1}]$.  They define a two-variable
Newton inverse-limit element whose positive-rank specializations recover all
symmetric-color HOMFLY--PT polynomials.  The completion is taken in the
Newton kernels rather than coefficientwise at roots of unity.  After a
positive rank and a color have been fixed, the series is finite and may be
evaluated at a root of unity. 
\end{abstract}

\maketitle
\tableofcontents
\section{Introduction}\label{sec:introduction}

Habiro's cyclotomic expansion of the colored Jones polynomial gives a
hierarchy of integral divisibilities whose multiplicities grow with the
color.  This is stronger than the existence of a finite-order
$q$-difference equation: the colored Jones function is $q$-holonomic
\cite{GL05}, as are fixed-row colored HOMFLY--PT functions \cite{GLL18},
whereas the expansion proved here additionally requires integrality of
the corresponding Newton divided differences.  For $\mathfrak{sl}_2$,
Habiro obtained the cyclotomic
divisibilities from the integral universal invariant and the completed
center \cite{Habiro02,Habiro06}.

Chen--Liu--Zhu formulated the corresponding expansion for
symmetric-power-colored $SU(n)$ invariants \cite[Conjecture~1.3]{CLZ15}.
They verified it for the figure-eight knot and the trefoil
\cite[Examples~2.5 and~2.6]{CLZ15}; subsequent work proved the
colored HOMFLY--PT expansion for double-twist knots
\cite[Theorem~1.2]{CLZ21}.  Beliakova--Gorsky later constructed a
cyclotomic interpolation basis for the center of
$U_{q^2}(\mathfrak{gl}_n)$ and proved Laurent integrality of the universal
knot coefficients \cite[Theorems~8.2--8.3 and Proposition~8.7]{BG24}.
Their finite dual coefficient formula and rank-reduction identity provide
the central input used here.  Restriction to symmetric powers initially
produces a one-sided factorial basis.  The main step is to recover the
two-sided Chen--Liu--Zhu kernel by a
reflection of the completed Harish--Chandra character and integral
descent through $z\mapsto z+z^{-1}$.

Fix $n\ge2$.  The notation $SU(n)$ is the conventional knot-theoretic
label; algebraically the construction uses
$U_{q^2}(\mathfrak{sl}_n)$.  Let $J_r^{SU(n)}(K;q)$ be the reduced,
zero-framed quantum invariant of a knot $K$ colored by the $r$th symmetric
power of the defining representation, normalized by
$J_r^{SU(n)}(U;q)=1$ for the unknot.  We use
\[
 \{m\}=q^m-q^{-m},\qquad [m]_q=\frac{\{m\}}{\{1\}},\qquad
 X_r^{(n)}=q^{2r+n}+q^{-2r-n}.
\]
The identity
\begin{equation}\label{eq:intro-node-factorization}
 X_r^{(n)}-X_i^{(n)}=\{r-i\}\{r+n+i\}
\end{equation}
shows that the Chen--Liu--Zhu kernel is exactly the Newton kernel associated
with the nodes $X_r^{(n)}$ \cite{CLZ15}.  The main result is the following.

\begin{theorem}[Cyclotomic expansion for symmetric-power-colored $SU(n)$ invariants]\label{thm:main-intro}
For every zero-framed knot $K$ and every fixed $n\ge2$, there is a unique
sequence of Laurent polynomials
\[
 \bigl(H_k^{(n)}(K;q)\bigr)_{k\ge0},
 \qquad H_k^{(n)}(K;q)\in\Z[q^{\pm1}],
\]
such that, for all $r\ge0$,
\begin{equation}\label{eq:main-expansion-intro}
 J_r^{SU(n)}(K;q)=\sum_{k=0}^{r}
 \left(\prod_{i=0}^{k-1}\{r-i\}\{r+n+i\}\right)H_k^{(n)}(K;q).
\end{equation}
Moreover, $H_0^{(n)}(K;q)=1$.
\end{theorem}
\medskip
\noindent For the rank-uniform statement, put
\[
 \{m;A\}=Aq^m-A^{-1}q^{-m},\qquad
 \{m;A\}_k=\prod_{s=0}^{k-1}\{m-s;A\},\qquad
 \{k\}!=\prod_{j=1}^k\{j\}.
\]
Write $\HH_r(K;A,q)$ for the reduced, zero-framed HOMFLY--PT invariant
colored by the $r$th symmetric power, in the skein normalization fixed in
\eqref{eq:HOMFLY-skein}; it satisfies
$\HH_r(K;q^n,q)=J_r^{SU(n)}(K;q)$ for $n\ge2$.
The rank-uniform coefficient ring will be
\[
 \mathscr I_q^{\ge2}
 =\left\{f(A)\in\Q(q)[A^{\pm1}]:
 f(q^n)\in\Z[q^{\pm1}]\ \text{for every }n\ge2\right\}.
\]

\begin{theorem}[Rank-uniform integer-valued coefficients]
\label{thm:rank-uniform-intro}
For every zero-framed knot $K$, there is a unique sequence
$G_k(K;A,q)\in\Z[A^{\pm1},q^{\pm1}]$, $k\ge1$, such that
\begin{equation}\label{eq:general-DE-intro}
 \HH_r(K;A,q)
 =1+\sum_{k=1}^{r}
 \qbinom{r}{k}\{r+k-1;A\}_k\{-1;A\}G_k(K;A,q)
 \qquad(r\ge0).
\end{equation}
Set
\begin{equation}\label{eq:rank-uniform-H-intro}
 \mathsf H_0(K;A,q)=1,\qquad
 \mathsf H_k(K;A,q)=\frac{\{-1;A\}G_k(K;A,q)}{\{k\}!}
 \quad(k\ge1).
\end{equation}
Then, for every $k\ge0$,
\begin{equation}\label{eq:rank-specialization-intro}
 \mathsf H_k(K;A,q)\in\mathscr I_q^{\ge2},\qquad
 \mathsf H_k(K;q^n,q)=H_k^{(n)}(K;q)
 \quad(n\ge2).
\end{equation}
\end{theorem}

\subsection*{Contributions and proof strategy}
Beliakova--Gorsky provide the interpolation basis, the finite dual formula,
the rank-reduction identity, and Laurent integrality of the dual
coefficients.  After matching normalizations, restriction to one-row colors
gives the one-sided basis $U_k(z)$.  We construct the corresponding completed
central sector and a faithful $h$-adic realization on which the
Harish--Chandra symmetries act continuously.  These symmetries yield
$z\leftrightarrow z^{-1}$; descent to $X=z+z^{-1}$, local interpolation at
cyclotomic primes, and global denominator removal then give the
Chen--Liu--Zhu expansion.

For the rank-uniform result, Morton's integrality and the sign, transpose, and
negative-rank symmetries produce Laurent coefficients $G_k(K;A,q)$.  Their
fixed-rank specializations yield integer-valued Newton coefficients in
$\mathscr I_q^{\ge2}$, which assemble into the two-variable Newton completion.

\section{Conventions and quantum-group input}
\label{sec:conventions-input}

We fix the coefficient rings, quantum-group conventions, and external input
used in the proof.

\subsection{Coefficient rings and elementary factorizations}

Throughout, put
\begin{equation}\label{eq:base-rings}
 Q=q^2,\qquad
 R=\Z[q^{\pm1}],\qquad
 R_{\mathrm{ev}}=\Z[Q^{\pm1}]
 =\Z[q^{\pm2}]\subset R,\qquad
 \Bbbk=\Q(q),
\end{equation}
where $q$ is an indeterminate.  For $m\in\Z$ put
\begin{equation}\label{eq:brace}
 \{m\}=q^m-q^{-m},\qquad [m]_q=\frac{\{m\}}{\{1\}}.
\end{equation}
For $m\ge0$, set
$\{m\}!=\prod_{a=1}^{m}\{a\}$ and
$[m]_q!=\prod_{a=1}^{m}[a]_q$, with the empty product equal to $1$.
Chen--Liu--Zhu write $[m]$ for the unnormalized expression
$q^m-q^{-m}$ \cite{CLZ15}; their $[m]$ is therefore our $\{m\}$, not
our $[m]_q$.  In particular, the kernel in \cref{thm:main-intro} is
their conjectured kernel without an additional power of $q-q^{-1}$.
The symmetric Gaussian coefficient is
\begin{equation}\label{eq:qbinomial}
 \qbinom{a}{b}=\frac{[a]_q!}{[b]_q![a-b]_q!}
 \qquad(0\le b\le a).
\end{equation}
Let
\[
 \begin{bmatrix}a\\ b\end{bmatrix}_{Q}^{\mathrm{ord}}
 :=
 \frac{(Q;Q)_a}{(Q;Q)_b(Q;Q)_{a-b}}
 \in\Z[Q]
\]
denote the ordinary Gaussian coefficient.  Then
\begin{equation}\label{eq:symmetric-ordinary-Gaussian}
 \begin{bmatrix}a\\ b\end{bmatrix}_{q}
 =
 q^{-b(a-b)}
 \begin{bmatrix}a\\ b\end{bmatrix}_{Q}^{\mathrm{ord}}\in R.
\end{equation}

For $m>0$,
\begin{equation}\label{eq:brace-cyclotomic-factorization}
 \{m\}=q^{-m}(q^{2m}-1)
 =q^{-m}\prod_{d\mid 2m}\Phi_d(q),
 \qquad \{-m\}=-\{m\}.
\end{equation}
Thus every nonunit irreducible factor of a quantum integer is cyclotomic,
and each occurs with multiplicity one in that individual quantum integer.
The ring $R$ is a UFD because it is the localization of the UFD $\Z[q]$ at
the powers of $q$.

\subsection{Translation of the Beliakova--Gorsky parameter}

Beliakova--Gorsky use a quantum-group parameter $\mfrakq$ and a square root
$v$ satisfying $v^2=\mfrakq$.  Our convention is
\begin{equation}\label{eq:parameter-translation}
 v=q,\qquad \mfrakq=Q.
\end{equation}
Consequently, their coefficient ring $\Z[v^{\pm1}]$ is $R$, their Laurent
coefficients in $\Z[\mfrakq^{\pm1}]$ lie in
$R_{\mathrm{ev}}$, and their interpolation parameter
$\mfrakq^{-1}$ is $q^{-2}$.  In particular, the quantum integer
$v^m-v^{-m}$ in their notation is exactly $\{m\}$ here.
The ambient integral quantum group is transported along
$\Z[v^{\pm1}]\xrightarrow{\sim}R$, $v\mapsto q$.  The even center and
its interpolation basis are instead defined over
$\Z[\mfrakq^{\pm1}]\xrightarrow{\sim}R_{\mathrm{ev}}$,
$\mfrakq\mapsto q^2$.  We retain this distinction in the coefficient
completion below.  The inclusion $R_{\mathrm{ev}}\subset R$ is finite
free, since $R=R_{\mathrm{ev}}\oplus qR_{\mathrm{ev}}$.

\subsection{The quantum groups and the integral completion}

Let $U_{q^2}(\mathfrak{gl}_n)$ be generated over $\Bbbk$ by
$E_i,F_i$ $(1\le i<n)$ and $K_j^{\pm1}$ $(1\le j\le n)$.  Put
$L_i=K_iK_{i+1}^{-1}$.  The relations needed later are
\begin{align}
 K_iE_iK_i^{-1}&=qE_i,&
 K_{i+1}E_iK_{i+1}^{-1}&=q^{-1}E_i,\notag\\
 K_iF_iK_i^{-1}&=q^{-1}F_i,&
 K_{i+1}F_iK_{i+1}^{-1}&=qF_i,\label{eq:gl-relations}\\
 [E_i,F_j]&=\delta_{ij}\frac{L_i-L_i^{-1}}{q-q^{-1}},
\end{align}
together with the commuting Cartan relations and the type-$A$ quantum
Serre relations.  The Hopf structure is
\begin{align}
 \Delta(E_i)&=E_i\otimes1+L_i\otimes E_i,&
 \Delta(F_i)&=F_i\otimes L_i^{-1}+1\otimes F_i,\notag\\
 \Delta(K_j)&=K_j\otimes K_j,&
 S(E_i)&=-L_i^{-1}E_i,\label{eq:Hopf}\\
 S(F_i)&=-F_iL_i,& S(K_j)&=K_j^{-1}.
\end{align}
This is the standard coproduct and antipode convention used in the
Habiro--L\^e construction underlying the completed universal invariant
\cite[Section~3A4]{HL16}.  For the displayed coproduct, the Hopf
identities
$m(S\otimes\id)\Delta(E_i)=0$ and
$m(\id\otimes S)\Delta(F_i)=0$ require the factor order in
\eqref{eq:Hopf}.  The antipode displayed in \cite[Section~3.1]{BG24}
uses the reverse order; throughout this paper we use
\eqref{eq:Hopf}.

For the universal matrix, let
$\Theta_{\mathrm{HL}}$ be the quasi-$R$ matrix defined in
\cite[equations~(65)--(68)]{HL16}; thus
$\Theta_{\mathrm{HL}}=\sum_{\mathbf n}F_{\mathbf n}\otimes
E_{\mathbf n}$ and its inverse is the corresponding sum of the bar
transforms.  With $D$ denoting the full $\mathfrak{gl}_n$ Cartan factor,
we fix the normalized positive universal matrix by
\begin{equation}\label{eq:HL-R-convention}
 \mathcal R=D\Theta_{\mathrm{HL}}^{-1}=D\Psi,
 \qquad
 \mathcal R^{-1}=\Theta_{\mathrm{HL}}D^{-1},
 \qquad \Psi:=\Theta_{\mathrm{HL}}^{-1}.
\end{equation}
This is the convention of \cite[equation~(70)]{HL16}.  The symbol
$\Theta$ in \cite[Section~3.2]{BG24} is used with a different convention for
the quasi-$R$ factor.  We therefore reserve $\Theta_{\mathrm{HL}}$ for the
matrix in \cite[equations~(65)--(68)]{HL16} and use
\eqref{eq:HL-R-convention} throughout.

The subalgebra generated by $E_i,F_i,L_i^{\pm1}$ is
$U_{q^2}(\mathfrak{sl}_n)$.

Let
\[
 \rho=\frac12\sum_{\alpha>0}\alpha
 =\sum_{j=1}^{n}\frac{n+1-2j}{2}\epsilon_j,
 \qquad
 K_{\pm2\rho}
 =\prod_{j=1}^{n}K_j^{\pm(n+1-2j)}.
\]
With the antipode convention in \eqref{eq:Hopf}, a direct calculation
gives
\[
 S^2(E_i)=L_i^{-1}E_iL_i=q^{-2}E_i,\qquad
 S^2(F_i)=L_i^{-1}F_iL_i=q^2F_i,\qquad
 S^2(K_j)=K_j.
\]
Consequently
\begin{equation}\label{eq:S2-balancing-convention}
 S^2=\operatorname{Ad}(K_{-2\rho}).
\end{equation}
Thus $g=K_{-2\rho}$ is the pivotal element in the convention used here,
in agreement with the terminology of
\cite[Section~3.2, equation~(11)]{BG24}.  We use the trace convention
$\operatorname{Tr}_{V}(gx)$ below.  Both $K_{2\rho}$ and
$K_{-2\rho}$ are fixed by the diagram automorphism used below.

Put
\begin{equation}\label{eq:BG-integral-generators}
 e_i=(q-q^{-1})E_i,
 \qquad F_i^{(a)}=\frac{F_i^a}{[a]_q!}.
\end{equation}
The integral form $U_{\Z}$ of \cite[Section~3.1]{BG24} is generated over
$R$ by the elements in \eqref{eq:BG-integral-generators} and by
$K_j^{\pm1}$ $(1\le j\le n)$.  Let
$\Gamma=\{(\zeta_1,\ldots,\zeta_n):\zeta_j=\pm1\}$ act by
\begin{equation}\label{eq:Gamma-action}
 K_j\longmapsto\zeta_jK_j,
 \qquad E_i\longmapsto E_i,
 \qquad F_i\longmapsto\zeta_i\zeta_{i+1}F_i.
\end{equation}
Write
\[
 U_{\Z}^{\Gamma}
 =\{u\in U_{\Z}:\zeta(u)=u\text{ for every }\zeta\in\Gamma\}.
\]
An element of this fixed subalgebra will be called $\Gamma$-invariant.
The elements $\sigma_\lambda$ are $\Gamma$-fixed central elements of the
ambient quantum group \cite[Theorem~8.2]{BG24}.

\begin{definition}
\label{def:filtered-continuity}
Let $B$ be equipped with a descending filtration
$B=B^{(0)}\supseteq B^{(1)}\supseteq\cdots$, and set
\[
 \widehat B=\varprojlim_N B/B^{(N)}.
\]
We use the induced linear topology: $b_\nu\to0$ if, for every $N$, all
sufficiently large $b_\nu$ lie in $B^{(N)}$.  A homomorphism
$\varphi:B\to C$ is continuous if, for every $N$, there is $M$ such that
\begin{equation}\label{eq:filtered-continuity-criterion}
 \varphi(B^{(M)})\subseteq C^{(N)}.
\end{equation}
\end{definition}

\begin{lemma}
\label{lem:filtered-extension}
Every continuous homomorphism $\varphi:B\to C$ extends uniquely to a
continuous homomorphism $\widehat\varphi:\widehat B\to\widehat C$.
If $\varphi$ is an isomorphism and $\varphi^{-1}$ is continuous,
then the extension is an isomorphism.
In the ring case, it preserves multiplication whenever multiplication is
continuous.
\end{lemma}

\begin{proof}
The extension is defined quotientwise.  Compatibility gives a map of inverse
limits, while density and separatedness give uniqueness.  The remaining
assertions follow by applying the same argument to the inverse and to
multiplication.
\end{proof}

\subsection{Standard structural facts}

We use the following standard Drinfeld--Jimbo facts for
$\mathfrak{gl}_n$ and its $\mathfrak{sl}_n$ subalgebra in the
$h$-adic/root-height completion.  The underlying construction of quantum
link invariants from ribbon Hopf algebras goes back to
Reshetikhin--Turaev \cite{RT90}; below we use the bottom-tangle
formulation and the conventions specified here.
\begin{enumerate}[label=\textup{(DJ\arabic*)},leftmargin=3.2em]
\item multiplication gives the triangular decomposition
$U^-\widehat\otimes U^0\widehat\otimes U^+\cong U$;
projection of the weight-zero part to $U^0$, followed by the $\rho$-shift,
defines the Harish--Chandra map on the center;
\item there is a unique normalized universal $R$-matrix
$\mathcal R=D\Psi$ in the convention \eqref{eq:HL-R-convention}, where
$D$ is determined by the invariant bilinear form and
$\Psi=\Theta_{\mathrm{HL}}^{-1}$ is $1$ plus
negative/positive root-height terms;
\item the Drinfeld element
$u=m(S\otimes\id)(\mathcal R_{21})$, ribbon element
$\mathbf r$, and pivotal element
$g=u\mathbf r^{-1}=K_{-2\rho}$ satisfy
$S^2=\operatorname{Ad}(g)$; the quantum trace used below employs $g$;
\item the universal invariant of a bottom tangle is constructed functorially
from $\mathcal R^{\pm1}$, the antipode, multiplication, and the
pivotal/ribbon data.
A continuous ribbon Hopf automorphism fixing these data acts componentwise
on the universal invariant.
\end{enumerate}
Triangular decomposition and the algebraic Harish--Chandra map are treated
in \cite{Jantzen96}; the convention used by Beliakova--Gorsky is recorded in
\cite[Section~5.1]{BG24}.  The root-height expansion and normalization of
the universal $R$-matrix used here are given in
\cite[Sections~3G1--3G2, equations~(64)--(70)]{HL16}, and the ribbon
data are treated in
\cite{Kassel95}.  The braided bottom-tangle functor is constructed in
\cite[Section~8.2 and Proposition~8.1]{Habiro06}; the convention used in
\cite{BG24} is summarized in its Section~3.5.

\subsection{Normalization of the knot invariant}

A zero-framed knot is represented by a one-component bottom tangle.  Let
$J_K(\mathfrak g;q^2)$ denote its universal invariant.  The universal
invariant of a one-component bottom tangle is adjoint invariant
\cite[Proposition~8.2]{Habiro06}; in the present generic type-$A$ setting it
is therefore central, as also used in \cite[Section~3.5]{BG24}.  Hence on an
irreducible highest-weight module $V(\lambda)$ it acts by a scalar
$c_K(\lambda)$.  We take this scalar as the reduced invariant.  Equivalently,
for any quantum-trace convention with nonzero generic quantum dimension it
is the normalized trace quotient.  In our convention,
\[
 \operatorname{Tr}_{q,V(\lambda)}(x)
 =\operatorname{Tr}_{V(\lambda)}(K_{-2\rho}x),
 \qquad
 \dim_qV(\lambda)=\operatorname{Tr}_{q,V(\lambda)}(1).
\]
The reduced scalar invariant is
\begin{equation}\label{eq:reduced-normalization}
 J_K(V(\lambda);q^2)
 =\frac{\operatorname{Tr}_{q,V(\lambda)}
       (J_K(\mathfrak g;q^2)|_{V(\lambda)})}
      {\dim_qV(\lambda)}.
\end{equation}
Both sides of \eqref{eq:reduced-normalization} equal $c_K(\lambda)$.
The invariant is normalized to $1$ for the zero-framed unknot.  For the one-row
partition $(r)=(r,0,\ldots,0)$ we write
\begin{equation}\label{eq:SU-normalization}
 J_r^{SU(n)}(K;q)=J_K(V((r));q^2).
\end{equation}
The label $SU(n)$ is conventional; all algebra is over the generic quantum
group.  The zero-framing assumption is essential because a framing change
introduces a color-dependent twist eigenvalue.

\begin{proposition}[Chen--Liu--Zhu normalization]
\label{prop:CLZ-normalization}
Let $K$ be an oriented zero-framed knot.  For every $n\ge2$ and
$r\ge0$, the invariant in \eqref{eq:SU-normalization} is the invariant
denoted by $J_r^{SU(n)}(K;q)$ in \cite{CLZ15}.
\end{proposition}

\begin{proof}
Under $t_{\mathrm{LZ}}^{1/2}=q^{-1}$ and
$\nu_{\mathrm{LZ}}^{1/2}=A^{-1}$, the two constructions have the same
braiding, framing correction, and unknot normalization; see
\cref{app:CLZ-normalization} for the detailed comparison.
\end{proof}

\subsection{HOMFLY--PT normalization and rank specialization}

We use the oriented HOMFLY--PT normalization
\cite{FYHLMO85,PT87}
\begin{equation}\label{eq:HOMFLY-skein}
 A P_{L_+}-A^{-1}P_{L_-}=(q-q^{-1})P_{L_0},
 \qquad P_U=1,
\end{equation}
and write $\HH_r(K;A,q)$ for the reduced invariant colored by the
$r$th symmetric power.  With the zero-framing and quantum-group
conventions fixed above, its positive-rank specialization is
\begin{equation}\label{eq:HOMFLY-SU-specialization}
 \HH_r(K;q^N,q)=J_r^{SU(N)}(K;q)\qquad(N\ge2).
\end{equation}
More generally, for a partition $\lambda$ let
$\HH_\lambda(K;A,q)$ denote the reduced zero-framed HOMFLY--PT invariant
colored by $\lambda$, normalized by
\(\HH_\lambda(U;A,q)=1\).  Thus
\(\HH_{(r)}=\HH_{[r]}=\HH_r\).

\begin{proposition}[All-partition normalization]
\label{prop:all-partition-normalization}
Let $K$ be an oriented zero-framed knot.  For every $N\ge2$ and every
partition $\lambda$ with at most $N$ parts,
\begin{equation}\label{eq:all-partition-normalization}
 \HH_\lambda(K;q^N,q)=J_K(V(\lambda);q^2).
\end{equation}
Here the right side is computed with the Hopf and ribbon convention
\eqref{eq:Hopf}--\eqref{eq:HL-R-convention}.  Thus the scalar family used
below in the algebraic dual interpolation formula is the standard reduced
colored HOMFLY--PT family.
\end{proposition}

\begin{proof}
For every partition $\lambda$, the two Hecke representations are conjugate,
their closure traces agree, and their framing corrections coincide after
$A=q^N$.  The detailed comparison is given in
\cref{app:CLZ-normalization}.
\end{proof}

\subsection{The Beliakova--Gorsky results}

Write $x_i=K_i^2$ for the Cartan coordinates in the target of the
Harish--Chandra map.  Let
$F_{\lambda;n}(x_1,\ldots,x_n;q^{-2})$ denote the interpolation
Macdonald polynomial in the normalization of \cite{BG24} after the
translation \eqref{eq:parameter-translation}.  For foundational
constructions and the Newton-interpolation characterization of these
polynomials, see \cite{Sahi96,Okounkov98}.  Trailing zero parts are allowed,
and $\ell(\lambda)$ denotes the number of nonzero parts.
The $\rho$-shift means concretely that, if
$f(x_1,\ldots,x_n)=\hc(c)$, then the scalar by which $c$ acts on
$V(\lambda)$ is obtained by substituting
\[
 x_i=q^{2(\lambda_i+\rho_i)}
 =q^{2\lambda_i+n+1-2i}.
\]
Consequently, for the shifted coordinates $y_i=q^{n-1}x_i$,
\[
 y_i(\lambda)=q^{2(\lambda_i+n-i)}.
\]
This explains both the factor $q^{n-1}$ in
\eqref{eq:BG-HC-exact} and the evaluation point in
\eqref{eq:BG-eigenvalue-exact}.

We use the following results from Beliakova--Gorsky; the scalar
normalization and Laurent-integrality interfaces are verified below.
\begin{enumerate}[label=\textup{(BG\arabic*)},leftmargin=3.2em]
\item there are central $\Gamma$-invariant elements $\sigma_\lambda$ whose
shifted Harish--Chandra images are
\begin{equation}\label{eq:BG-HC-exact}
 \hc(\sigma_\lambda)
 =F_{\lambda;n}(q^{n-1}x_1,\ldots,q^{n-1}x_n;q^{-2}),
\end{equation}
and whose scalar on $V(\mu)$ is
\begin{equation}\label{eq:BG-eigenvalue-exact}
 F_{\lambda;n}
 (q^{2(\mu_1+n-1)},q^{2(\mu_2+n-2)},\ldots,q^{2\mu_n};q^{-2});
\end{equation}
\cite[Theorem~8.2]{BG24}.
\item the dual interpolation basis associated with the
$\sigma_\lambda$ is given in \cite[Theorem~8.3]{BG24}.  In particular,
the interpolation matrix is triangular, its diagonal entries are the
nonzero hook products in \cite[Theorem~8.2(c)]{BG24}, and the scalar
values on finite-dimensional highest-weight modules determine every
finite linear combination of the $\sigma_\lambda$.  The completed topology
used here is constructed in \cref{sec:BG-restriction}.
\item for every zero-framed knot, the dual interpolation formula of
\cite[Theorem~8.3]{BG24} applied to the scalar family
$J_K(V(\mu);q^2)$ defines
\begin{equation}\label{eq:BG-finite-dual-coefficient}
 a_\lambda(K;q^2)
 =\sum_{\mu\subseteq\lambda}
 d_{\lambda,\mu}(Q^{-1})J_K(V(\mu);q^2)\in\Bbbk,
\end{equation}
where $d_{\lambda,\mu}$ is the interpolation coefficient used in that
theorem.  Duality and triangularity give, for every partition $\mu$, the
finite representationwise identity
\begin{equation}\label{eq:BG-knot-expansion}
 J_K(V(\mu);q^2)\id_{V(\mu)}
 =\left(\sum_{\lambda\subseteq\mu}
 a_\lambda(K;q^2)\sigma_\lambda\right)\Big|_{V(\mu)}.
\end{equation}
Their Laurent integrality is proved in
\cref{prop:BG-coefficient-integrality}.
\item the rank-reduction formula is
\begin{equation}\label{eq:BG-rank-reduction-exact}
 F_{\lambda;m}(x_1,\ldots,x_{m-1},1;q^{-2})
 =\begin{cases}
 (-1)^{m-1}q^{-2\binom{m-1}{2}}
 F_{\lambda;m-1}(q^{-2}x_1,\ldots,q^{-2}x_{m-1};q^{-2}),
 &\lambda_m=0,\\
 0,&\lambda_m>0,
 \end{cases}
\end{equation}
and in rank one
$F_{(k);1}(x;q^{-2})=(x;q^{-2})_k$
\cite[Lemma~11.1 and equation~(34)]{BG24}.
\item in the convention \eqref{eq:HL-R-convention}, for every
zero-framed algebraically split link
\begin{equation}\label{eq:BG-gl-sl-exact}
 J_L(\mathfrak{gl}_n;q^2)=J_L(\mathfrak{sl}_n;q^2)
\end{equation}
after embedding both universal elements in $U_h(\mathfrak{gl}_n)$.  This is proved in
\cref{lem:BG-gl-sl-interface}; compare
\cite[Corollary~6.5 and Remark~6.4]{BG24}.
\end{enumerate}
\begin{lemma}[Scalar comparison]
\label{lem:BG-scalar-interface}
Let $J_K^{\mathrm{BG}}(V(\mu);\mfrakq)$ denote the reduced
Reshetikhin--Turaev knot scalar of
\cite[Section~3.5 and equations~(12)--(13)]{BG24}, normalized to be $1$
for the zero-framed unknot.  Under the parameter translation
\eqref{eq:parameter-translation},
\begin{equation}\label{eq:BG-scalar-interface}
 J_K^{\mathrm{BG}}(V(\mu);Q)
 =J_K(V(\mu);q^2)
 =\HH_\mu(K;q^n,q)
\end{equation}
for every zero-framed oriented knot $K$, every integer $n\ge2$, and every
partition $\mu$ with at most $n$ parts.
\end{lemma}

\begin{proof}
Section~3.5 of \cite{BG24} explicitly adopts the convention of
\cite[Section~2G]{HL16}, writes the universal contributions from left to
right along each oriented component, and defines the colored
Reshetikhin--Turaev scalar by equations~(12)--(13).  For a knot, division by
$\dim_qV(\mu)$ gives the reduced scalar normalized to be $1$ on the
zero-framed unknot.  After $v=q$ and $\mfrakq=Q$, this uses the same
module $V(\mu)$, the same positive braiding
$\mathcal R=D\Theta_{\mathrm{HL}}^{-1}$, the same pivotal trace
$\operatorname{Tr}_{V(\mu)}(K_{-2\rho}\,\cdot)$, and the same zero-framing
correction as \eqref{eq:reduced-normalization}.  Hence the first equality in
\eqref{eq:BG-scalar-interface} is an equality of finite reduced scalar
invariants.  The second equality is
\cref{prop:all-partition-normalization}, whose full Hecke, crossing, twist,
and trace comparison is given in \cref{app:CLZ-normalization}.
\end{proof}

\begin{lemma}[The abelian Cartan factor]
\label{lem:BG-gl-sl-interface}
For every integer $n\ge2$ and every zero-framed algebraically split link $L$, the
$\mathfrak{gl}_n$ and $\mathfrak{sl}_n$ universal elements constructed in
the convention \eqref{eq:Hopf}--\eqref{eq:HL-R-convention} agree after
the standard embedding of $U_h(\mathfrak{sl}_n)$ in
$U_h(\mathfrak{gl}_n)$.
\end{lemma}

\begin{proof}
Work in $U_h(\mathfrak{gl}_n)$ and set
$C=H_1+\cdots+H_n$.  Orthogonal decomposition of the Cartan space into
$\mathfrak{sl}_n$ and its scalar summand gives
\begin{equation}\label{eq:gl-sl-Cartan-factor}
 D_{\mathfrak{gl}_n}
 =D_{\mathfrak{sl}_n}
   \exp\!\left(\frac{h}{2n}C\otimes C\right).
\end{equation}
The two universal matrices have the same normalized root factor $\Psi$;
the last factor in \eqref{eq:gl-sl-Cartan-factor} is the universal matrix
of the one-dimensional abelian summand.  In the state sum of a framed link, its
contribution is determined entirely by the linking matrix: off-diagonal
entries record crossings between distinct components, and diagonal entries
record framings.  This follows by commuting the central primitive element
$C$ along each component; it is also the calculation underlying
\cite[Remark~6.4 and Corollary~6.5]{BG24}.  The pivotal element
$K_{-2\rho}$ lies in the $\mathfrak{sl}_n$ factor because
$\sum_i\rho_i=0$.  Thus the abelian contribution is $1$ when every
framing and every pairwise linking number is zero, proving
\eqref{eq:BG-gl-sl-exact}.
\end{proof}

\begin{proposition}[Laurent integrality of the dual coefficients]
\label{prop:BG-coefficient-integrality}
For every zero-framed knot $K$, every $n\ge2$, and every partition
$\lambda$ with at most $n$ parts, the coefficient defined in
\textup{(BG3)} satisfies
\begin{equation}\label{eq:BG-coefficient-integrality}
 a_\lambda(K;q^2)\in R_{\mathrm{ev}}=\Z[q^{\pm2}].
\end{equation}
\end{proposition}

\begin{proof}
A zero framing is even, so
\cite[Theorem~1.3 and Proposition~8.7]{BG24} apply.  Written with the
Beliakova--Gorsky parameter $\mfrakq$, their finite scalar conclusion is
\[
 \sum_{\mu\subseteq\lambda}
 d_{\lambda,\mu}(\mfrakq^{-1})
 J_K^{\mathrm{BG}}(V(\mu);\mfrakq)
 \in\Z[\mfrakq^{\pm1}].
\]
By \cref{lem:BG-scalar-interface} and $\mfrakq=Q$, the left side is
exactly the coefficient \eqref{eq:BG-finite-dual-coefficient}.  Therefore
\[
 a_\lambda(K;q^2)\in\Z[Q^{\pm1}]=R_{\mathrm{ev}},
\]
as claimed.
\end{proof}

\section{Newton interpolation and local algebra}\label{sec:newton-local}

\subsection{Newton coefficients, divided differences, and determinants}

Let $\mathscr D$ be an integral domain with fraction field $L$, and let
$x_0,x_1,\ldots\in L$ be pairwise distinct.  Define
\begin{equation}\label{eq:newton-polynomials}
 p_0(T)=1,
 \qquad
 p_k(T)=\prod_{i=0}^{k-1}(T-x_i)\quad(k\ge1).
\end{equation}
For a sequence $f=(f_r)_{r\ge0}$ in $L$, a Newton expansion is
\begin{equation}\label{eq:abstract-newton-expansion}
 f_r=\sum_{k=0}^{r}p_k(x_r)h_k.
\end{equation}

\begin{proposition}[Newton inversion]\label{prop:newton-inversion}
Every sequence $f_r\in L$ has a unique expansion
\eqref{eq:abstract-newton-expansion}.  Its coefficients satisfy
\begin{equation}\label{eq:newton-recursion}
 h_k=
 \frac{f_k-\sum_{j=0}^{k-1}p_j(x_k)h_j}{p_k(x_k)}
\end{equation}
and
\begin{equation}\label{eq:newton-barycentric}
 h_k=\sum_{j=0}^{k}\frac{f_j}{D_{k,j}},
 \qquad
 D_{k,j}=\prod_{\substack{0\le i\le k\\i\ne j}}(x_j-x_i).
\end{equation}
For $k\ge1$, put
\begin{equation}\label{eq:Delta}
 \Delta_k(f)=\det
 \begin{pmatrix}
 1&x_0&\cdots&x_0^{k-1}&f_0\\
 1&x_1&\cdots&x_1^{k-1}&f_1\\
 \vdots&\vdots&&\vdots&\vdots\\
 1&x_k&\cdots&x_k^{k-1}&f_k
 \end{pmatrix},
 \qquad \Delta_0(f)=f_0,
\end{equation}
and
\begin{equation}\label{eq:Vandermonde}
 V_k=\prod_{0\le i<j\le k}(x_j-x_i),
 \qquad V_0=1.
\end{equation}
Then
\begin{equation}\label{eq:determinant-divided-difference}
 h_k=\frac{\Delta_k(f)}{V_k}.
\end{equation}
\end{proposition}

\begin{proof}
For fixed $k$, the matrix
$M_k=(p_j(x_i))_{0\le i,j\le k}$ is lower triangular because
$p_j(x_i)=0$ when $j>i$.  Its diagonal entries are
$p_i(x_i)=\prod_{a<i}(x_i-x_a)$, which are nonzero.  Therefore the system
\eqref{eq:abstract-newton-expansion} for $0\le r\le k$ has a unique
solution.  Solving its last row gives
\eqref{eq:newton-recursion}.
Applying the last-row equation successively defines the coefficients; the
finite systems are compatible by uniqueness.

Let $P_k(T)$ be the unique polynomial of degree at most $k$ satisfying
$P_k(x_j)=f_j$ for $0\le j\le k$.  Lagrange interpolation gives
\begin{equation}\label{eq:Lagrange-polynomial}
 P_k(T)=\sum_{j=0}^{k}f_j
 \prod_{\substack{0\le i\le k\\i\ne j}}
 \frac{T-x_i}{x_j-x_i}.
\end{equation}
The coefficient of $T^k$ is the right-hand side of
\eqref{eq:newton-barycentric}.  On the other hand, the Newton form
of $P_k$ is
$\sum_{j=0}^{k}h_jp_j(T)$.  Indeed, for $0\le i\le k$,
\[
 \sum_{j=0}^{k}h_jp_j(x_i)
 =\sum_{j=0}^{i}h_jp_j(x_i)=f_i,
\]
because $p_j(x_i)=0$ for $j>i$; uniqueness of interpolation identifies
this polynomial with $P_k$.  Only the monic polynomial $p_k$ contributes
to its leading coefficient.  Hence that coefficient is $h_k$, proving
\eqref{eq:newton-barycentric}.

Finally, solve the interpolation system in the monomial basis.  Its
coefficient matrix is the Vandermonde matrix, whose determinant is $V_k$.
By Cramer's rule, the leading coefficient is
$\Delta_k(f)/V_k$.  Since the same leading coefficient is $h_k$,
\eqref{eq:determinant-divided-difference} follows.  This argument also
fixes the sign convention in \eqref{eq:Delta}.
\end{proof}

\begin{corollary}
\label{cor:integral-newton-criterion}
Assume $x_r,f_r\in\mathscr D$.  The following are equivalent:
\begin{enumerate}[label=\textup{(\roman*)}]
\item $h_k\in\mathscr D$ for every $k$;
\item $\Delta_k(f)\in V_k\mathscr D$ for every $k$;
\item for every $k\ge1$, the numerator of
\eqref{eq:newton-recursion} belongs to $p_k(x_k)\mathscr D$.
\end{enumerate}
\end{corollary}

\begin{proof}
The equivalence of \textup{(i)} and \textup{(iii)} follows inductively from
\eqref{eq:newton-recursion}; the equivalence of \textup{(i)} and
\textup{(ii)} is \eqref{eq:determinant-divided-difference}.
\end{proof}

\begin{corollary}
\label{cor:newton-filtration}
If $h_k\in\mathscr D$ for all $k$, then, for $0\le m<r$,
\begin{equation}\label{eq:newton-remainder-congruence}
 f_r\equiv\sum_{k=0}^{m}p_k(x_r)h_k
 \pmod{p_{m+1}(x_r)\mathscr D}.
\end{equation}
Moreover, for $r>s$,
\begin{equation}\label{eq:abstract-pairwise-congruence}
 f_r-f_s\in(x_r-x_s)\mathscr D.
\end{equation}
\end{corollary}

\begin{proof}
For $k\ge m+1$, the polynomial $p_k$ is divisible by $p_{m+1}$, proving
\eqref{eq:newton-remainder-congruence}.  The finite interpolation polynomial
$P_r(T)=\sum_{k=0}^{r}h_kp_k(T)$ lies in $\mathscr D[T]$.  Therefore
$P_r(x_r)-P_r(x_s)$ is divisible by $x_r-x_s$ in $\mathscr D$.  Moreover,
$P_r(x_r)=f_r$, while for $s<r$ all terms with $k>s$ vanish at $x_s$, so
\[
 P_r(x_s)=\sum_{k=0}^{s}h_kp_k(x_s)=f_s.
\]
Thus the polynomial divisibility gives
\eqref{eq:abstract-pairwise-congruence}.
\end{proof}

\begin{example}
\label{ex:pairwise-insufficient}
Take $\mathscr D=\Z$, nodes $(x_0,x_1,x_2)=(0,2,4)$, and values
$(f_0,f_1,f_2)=(0,0,4)$.  For every $i<j$ one has
$f_j-f_i\in(x_j-x_i)\Z$: the only nonzero differences are $4$, which are
divisible by both $4$ and $2$.  Nevertheless,
\[
 h_0=0,\qquad h_1=0,\qquad
 h_2=\frac{4}{(4-0)(4-2)}=\frac12\notin\Z.
\]
Modulo $2$, all three nodes collide.  Pairwise congruences see that
collision only one pair at a time, whereas the Vandermonde denominator
retains its total multiplicity.
\end{example}

\subsection{Nodes for symmetric powers}

For fixed $n\ge2$, put
\begin{equation}\label{eq:SU-nodes}
 X_r^{(n)}=q^{2r+n}+q^{-2r-n}\qquad(r\ge0).
\end{equation}

\begin{lemma}
\label{lem:node-factorization}
For all integers $r,i$,
\begin{equation}\label{eq:node-factorization}
 X_r^{(n)}-X_i^{(n)}=\{r-i\}\{r+n+i\}.
\end{equation}
For $r,i\ge0$ and $n\ge2$, the nodes are pairwise distinct.
\end{lemma}

\begin{proof}
Expanding the product gives
\begin{align*}
 \{r-i\}\{r+n+i\}
 &=(q^{r-i}-q^{-r+i})(q^{r+n+i}-q^{-r-n-i})\\
 &=q^{2r+n}+q^{-2r-n}-q^{2i+n}-q^{-2i-n},
\end{align*}
which is \eqref{eq:node-factorization}.  If $r,i\ge0$ and the
right-hand side vanishes in $\Bbbk$, then one quantum integer is zero.
Since
$q$ is an indeterminate, $\{m\}=0$ only for $m=0$.  The factor
$r+n+i$ is positive, so $r=i$.
\end{proof}

It follows that
\begin{equation}\label{eq:SU-newton-kernel}
 p_k(X_r^{(n)})=
 C_{r+1,k}^{(n)}
 :=\prod_{i=0}^{k-1}\{r-i\}\{r+n+i\}.
\end{equation}
Thus every sequence in $\Bbbk$ has a unique formal expansion
\begin{equation}\label{eq:formal-SU}
 J_r=\sum_{k=0}^{r}C_{r+1,k}^{(n)}H_k^{(n)}.
\end{equation}
For a sequence $J=(J_r)_{r\ge0}$, let
$\Delta_k^{(n)}(J)$ and $V_k^{(n)}$ denote the determinant
\eqref{eq:Delta} and the Vandermonde product \eqref{eq:Vandermonde}
formed with the nodes $x_i=X_i^{(n)}$.  Then
\begin{equation}\label{eq:SU-determinant}
 H_k^{(n)}=\frac{\Delta_k^{(n)}(J)}{V_k^{(n)}},
 \qquad
 V_k^{(n)}=
 \prod_{0\le i<j\le k}\{j-i\}\{n+i+j\}.
\end{equation}

\begin{lemma}
\label{lem:barycentric-denominator}
Fix $n\ge2$.  For $0\le j\le k$,
\begin{equation}\label{eq:Dkj}
 D_{k,j}^{(n)}
 =(-1)^{k-j}\{j\}!\{k-j\}!
 \frac{\{n+j+k\}!}
      {\{n+j-1\}!\{n+2j\}}.
\end{equation}
The displayed quotient is a Laurent polynomial.
\end{lemma}

\begin{proof}
Using \eqref{eq:node-factorization}, split
\[
 D_{k,j}^{(n)}=
 \prod_{i\ne j}\{j-i\}\cdot
 \prod_{i\ne j}\{n+i+j\}.
\]
For the first product, the indices below $j$ give
$\{j\}!$ while those above $j$ give
$(-1)^{k-j}\{k-j\}!$.  For the second,
\[
 \prod_{\substack{0\le i\le k\\i\ne j}}\{n+i+j\}
 =\frac{\prod_{b=n+j}^{n+j+k}\{b\}}{\{n+2j\}}
 =\frac{\{n+j+k\}!}
        {\{n+j-1\}!\{n+2j\}}.
\]
Because $0\le j\le k$, the index $n+2j$ lies in
$[n+j,n+j+k]$, and hence
\[
 \frac{\{n+j+k\}!}{\{n+j-1\}!\{n+2j\}}
 =\prod_{\substack{n+j\le b\le n+j+k\\b\ne n+2j}}\{b\}\in R.
\]
\end{proof}

\subsection{Cyclotomic local rings and exact valuations}

For $d\ge1$, define
\begin{equation}\label{eq:cyclotomic-local-rings}
 \cO_d=R_{(\Phi_d(q))},
 \qquad
 \widehat\cO_d=\varprojlim_N\cO_d/(\Phi_d(q)^N),
 \qquad
 \pi_d=\Phi_d(q).
\end{equation}
Set
\begin{equation}\label{eq:dsharp}
 d^\sharp=\frac{d}{\gcd(d,2)}.
\end{equation}

\begin{lemma}
\label{lem:cyclotomic-DVR}
The ring $\cO_d$ is a DVR with uniformizer $\pi_d$.  Its completion is a
complete DVR with residue field $\Q(\zeta_d)$.  The maps
$R\hookrightarrow\cO_d\hookrightarrow\widehat\cO_d$ are injective, and
$\Bbbk$ embeds in $\operatorname{Frac}(\widehat\cO_d)$.  Every nonzero
integer,
in particular $2$, is a unit in both local rings.  For $m\ne0$,
\begin{equation}\label{eq:quantum-integer-valuation}
 \val_{\Phi_d}(\{m\})=
 \begin{cases}
 1,&d^\sharp\mid m,\\
 0,&d^\sharp\nmid m.
 \end{cases}
\end{equation}
For every $\xi\in\Bbbk^\times$, the valuation of its image in
$\operatorname{Frac}(\widehat\cO_d)$ equals $\val_{\Phi_d}(\xi)$.
\end{lemma}

We use the convention $\val_{\Phi_d}(0)=+\infty$ whenever the valuation is
applied to an expression that may vanish.

\begin{proof}
The Laurent polynomial ring $R$ is a UFD, and $\Phi_d(q)$ is irreducible.
Localization at the height-one prime it generates is therefore a DVR\@.
The quotient of the localization by its maximal ideal is the fraction
field of $R/(\Phi_d)$, namely $\Q(\zeta_d)$.  Completion of a DVR is again a
DVR with the same uniformizer and residue field.

The map from a DVR to its completion is injective because
$\bigcap_{N\ge0}(\pi_d^N)=0$: a nonzero element has finite valuation.
Localization then embeds the fraction field.  No nonzero integer belongs to
$(\Phi_d)$, since $\Phi_d$ is a nonconstant primitive polynomial; therefore
all nonzero integers are units.

By \eqref{eq:brace-cyclotomic-factorization}, $\Phi_d$ divides
$\{m\}$ exactly when $d\mid2m$, equivalently when $d^\sharp\mid m$.  The
roots of $q^{2|m|}-1$ are simple in characteristic zero, so the multiplicity
is one.

For valuation compatibility, write $\xi=\pi_d^e u$ with
$u\in\cO_d^\times$.  The element $u$ remains a unit after completion, so
the valuation is unchanged.
\end{proof}

\subsection{Integral interpolation over a complete DVR}

The next lemma preserves the full multiplicity of a residue cluster by
working in a finite-free quotient.

\begin{lemma}[Finite-free quotient]
\label{lem:finite-quotient-germ}
Let $\cO$ be a complete DVR with uniformizer $\pi$, let $m\ge1$, let
$u_1,\ldots,u_m\in\pi\cO$, and put
\[
 g(U)=\prod_{i=1}^{m}(U-u_i)\in\cO[U].
\]
For every $F(U)\in\cO[[U]]$, there is a unique polynomial
$P(U)\in\cO[U]$ of degree $<m$ whose class in
$B=\cO[U]/(g)$ is the convergent value of the series $F(U)$.  Moreover,
\begin{equation}\label{eq:germ-evaluation}
 P(u_i)=F(u_i)\qquad(1\le i\le m).
\end{equation}
\end{lemma}

\begin{proof}
Because $g$ is monic, $B$ is a free $\cO$-module with basis
$1,U,\ldots,U^{m-1}$; hence it is $\pi$-adically complete and separated.
All elementary symmetric functions of $u_1,\ldots,u_m$ of positive degree
lie in $\pi\cO$.  The relation $g(U)=0$ therefore implies
$U^m\in\pi B$, and consequently
$U^{mN}\in\pi^NB$ for every $N$.  If $a\ge mN$, then
$U^a=U^{mN}U^{a-mN}\in\pi^NB$.  Hence the tail
$\sum_{a\ge mN}c_aU^a$ is zero modulo $\pi^N B$, and the partial sums of
$\sum_{a\ge0}c_aU^a$ form a Cauchy sequence for arbitrary coefficients
$c_a\in\cO$.  Completeness gives a limit.  Its unique representative of
degree $<m$ is obtained by expressing that limit in the free basis
$1,U,\ldots,U^{m-1}$.

For each $i$, the equality $g(u_i)=0$ shows that evaluation
$U\mapsto u_i$ factors through $B$.  It is continuous because
$u_i\in\pi\cO$, so the image of $U^a$ tends to zero $\pi$-adically.
Applying evaluation to the convergent series gives $F(u_i)$, while
applying it to the polynomial representative gives $P(u_i)$.  This proves
\eqref{eq:germ-evaluation}.
\end{proof}

\begin{lemma}
\label{lem:resultant-comaximal}
Let $g,h\in\cO[T]$ be monic.  If their reductions in the residue field have
no common root in an algebraic closure, then
$\operatorname{Res}(g,h)\in\cO^\times$ and the ideals $(g)$ and $(h)$ are
comaximal in $\cO[T]$.
\end{lemma}

\begin{proof}
The reduction of the resultant is the resultant of the reductions.  It is
nonzero precisely when the two reduced polynomials are coprime, so the
original resultant is a unit.  The standard Sylvester-matrix identity
expresses the resultant as $A(T)g(T)+B(T)h(T)$ with
$A,B\in\cO[T]$.  Dividing by the unit resultant yields
$1\in(g)+(h)$.
\end{proof}

\begin{lemma}[Clusterwise integral interpolation]
\label{lem:cluster-interpolation}
Let $\cO$ be a complete DVR with fraction field $L$, and let
$x_0,\ldots,x_k\in\cO$ be pairwise distinct in $L$.  Partition the indices
into residue clusters $C$ according to equality of $\bar x_i$.  Suppose
that, for each cluster $C$, there are $c_C\in\cO$ and
$F_C(U)\in\cO[[U]]$ such that
\begin{equation}\label{eq:cluster-germ-values}
 x_i-c_C\in\pi\cO,
 \qquad
 y_i=F_C(x_i-c_C)
 \quad(i\in C).
\end{equation}
Then the unique polynomial $P(T)\in L[T]$ of degree at most $k$ satisfying
$P(x_i)=y_i$ belongs to $\cO[T]$.
\end{lemma}

\begin{proof}
For a cluster $C$, put
$u_i=x_i-c_C$ and
$h_C(U)=\prod_{i\in C}(U-u_i)$.  By
\cref{lem:finite-quotient-germ}, the class of $F_C$ modulo $h_C$ has a
representative $P_C(U)\in\cO[U]$ of degree $<|C|$ satisfying
$P_C(u_i)=y_i$.  In the original coordinate define
$Q_C(T)=P_C(T-c_C)$ and
$g_C(T)=\prod_{i\in C}(T-x_i)$.  Then
\begin{equation}\label{eq:PC-cluster}
 Q_C(x_i)=y_i\quad(i\in C).
\end{equation}

If $C\ne C'$, then
$\bar g_C=(T-\bar c_C)^{|C|}$ and
$\bar g_{C'}=(T-\bar c_{C'})^{|C'|}$ have distinct roots.  By
\cref{lem:resultant-comaximal}, their ideals are comaximal.
Indeed, for $i\in C$, the condition
$x_i-c_C\in\pi\cO$ gives $\bar c_C=\bar x_i$.  Distinct residue clusters
therefore satisfy $\bar c_C\ne\bar c_{C'}$, which proves that the two
displayed powers have no common root.
The Chinese remainder theorem gives
\[
 \cO[T]/\left(\prod_Cg_C\right)
 \cong\prod_C\cO[T]/(g_C).
\]
Since
$G(T)=\prod_Cg_C(T)$ is monic of degree $k+1$, division by $G$ gives a
unique representative $P(T)\in\cO[T]$ of degree at most $k$.  Equation
\eqref{eq:PC-cluster} shows that it interpolates all values.
Over $L$, two degree-$\le k$ polynomials agreeing at the $k+1$ distinct
nodes are equal, so this is the unique interpolation polynomial.
\end{proof}

\begin{lemma}[UFD denominator removal]
\label{lem:UFD-denominator}
Let $N,V\in R$ with $V\ne0$.  Assume that $V$ is primitive as a Laurent
polynomial after multiplication by a power of $q$, every nonconstant
irreducible factor of $V$ is cyclotomic, and
$N/V\in\widehat\cO_d$ for every cyclotomic factor $\Phi_d$ of $V$.
Then $N/V\in R$.
\end{lemma}

\begin{proof}
Write $N/V=A/B$ in lowest terms in the UFD $R$, so $(A,B)=1$.  The
identity $NB=VA$ and Euclid's lemma imply that $B$ divides $V$ up to a
unit: every irreducible factor of $B$ is coprime to $A$ and must therefore
divide $V$.  After clearing a power of $q$, primitivity of $V$ means that
its coefficients have greatest common divisor $1$.  Gauss's lemma then
excludes every nonunit integer prime from the factorization of $V$, and
hence from that of $B$.  By hypothesis, every remaining irreducible
factor of $B$ is a cyclotomic polynomial $\Phi_d(q)$.

Suppose $\Phi_d^e\mid B$ with $e>0$.  Coprimality gives
$\val_{\Phi_d}(A)=0$, while
$\val_{\Phi_d}(B)=e$, so
$\val_{\Phi_d}(A/B)=-e<0$.  This contradicts
$A/B\in\widehat\cO_d$, because the valuation ring of
$\operatorname{Frac}(\widehat\cO_d)$ consists exactly of the elements of
nonnegative $\Phi_d$-valuation.  Thus $B$ has no nonunit irreducible
factor, so it is a unit of $R$ and $A/B\in R$.
\end{proof}

\section{The one-sided expansion from the completed center}\label{sec:BG-restriction}

All invariants below are reduced and zero-framed, with the conventions of
\cref{sec:conventions-input}.

\subsection{Iterated one-row restriction}

For $z$ an invertible variable, define the affine one-row
Harish--Chandra curve
\begin{equation}\label{eq:gamma-affine}
 \gamma_n(z)=
 (q^{n-2}z,q^{2n-4},q^{2n-6},\ldots,q^2,1).
\end{equation}
For every $s\in\Z$, put
\begin{equation}\label{eq:z-s}
 z_s=q^{2s+n}.
\end{equation}
For $k\ge0$, put
\begin{equation}\label{eq:z-U-definitions}
 U_k(z)=(q^{-n}z;q^{-2})_k
 =\prod_{j=0}^{k-1}(1-q^{-n-2j}z),
 \qquad U_0(z)=1,
\end{equation}
and set
\begin{equation}\label{eq:epsilon-n}
 \varepsilon_n=(-1)^{\binom n2}q^{-2\binom n3}.
\end{equation}

\begin{lemma}\label{lem:iterated-rank-reduction}
Fix $n\ge2$, and let $\lambda$ be a partition with at most $n$ parts.  Then
\begin{equation}\label{eq:iterated-rank-reduction}
 F_{\lambda;n}(\gamma_n(z);q^{-2})=
 \begin{cases}
  \varepsilon_n U_k(z),&\lambda=(k),\\
  0,&\ell(\lambda)\ge2.
 \end{cases}
\end{equation}
\end{lemma}

\begin{proof}
Write the rank-reduction identity
\eqref{eq:BG-rank-reduction-exact} at rank $m$ as
\begin{equation}\label{eq:single-rank-reduction}
 F_{\lambda;m}(x_1,\ldots,x_{m-1},1;q^{-2})
 =c_mF_{\lambda;m-1}(q^{-2}x_1,\ldots,q^{-2}x_{m-1};q^{-2})
\end{equation}
when $\lambda_m=0$, where
$c_m=(-1)^{m-1}q^{-2\binom{m-1}{2}}$; it is zero when
$\lambda_m>0$.

Suppose first that $\ell(\lambda)\ge2$.  Repeatedly applying
\eqref{eq:single-rank-reduction} removes trailing zero parts until
rank $m=\ell(\lambda)$.  At that step $\lambda_m>0$, so the restriction
vanishes.

Now let $\lambda=(k)$.  Every reduction step is nonzero.  After the steps
$n,n-1,\ldots,2$, the first coordinate has been multiplied by
$q^{-2(n-1)}$, while all other coordinates have successively become the
last coordinate $1$ and have been removed.  Hence
\begin{align*}
 F_{(k);n}(\gamma_n(z);q^{-2})
 &=\left(\prod_{m=2}^{n}c_m\right)
   F_{(k);1}(q^{n-2}q^{-2(n-1)}z;q^{-2})\\
 &=\left(\prod_{m=2}^{n}c_m\right)
 F_{(k);1}(q^{-n}z;q^{-2}).
\end{align*}
The rank-one interpolation polynomial is
$F_{(k);1}(x;q^{-2})=(x;q^{-2})_k$.  Moreover,
\[
 \sum_{m=2}^{n}(m-1)=\binom n2,
 \qquad
 \sum_{m=2}^{n}\binom{m-1}{2}=\binom n3.
\]
Thus $\prod_{m=2}^{n}c_m=\varepsilon_n$, proving
\eqref{eq:iterated-rank-reduction}.
\end{proof}

\begin{proposition}[Integral one-sided color expansion]
\label{prop:one-row-restriction}
For every fixed $n\ge2$ and every zero-framed knot $K$, there are Laurent
polynomials
\begin{equation}\label{eq:b-coefficients}
 b_k^{(n)}(K;q)=\varepsilon_n a_{(k)}(K;q^2)
 \in\Z[q^{\pm1}]
\end{equation}
such that, for every $r\ge0$,
\begin{equation}\label{eq:one-sided-color}
 J_r^{SU(n)}(K;q)=
 \sum_{k=0}^{r}b_k^{(n)}(K;q)
 \prod_{j=0}^{k-1}(1-q^{2(r-j)}).
\end{equation}
Here $a_{(k)}(K;q^2)$ is the coefficient defined by the finite dual
formula in \textup{(BG3)}.  Explicitly, for each partition $\mu$ one has
the finite representationwise identity
\begin{equation}\label{eq:BG-representationwise-expansion}
 J_K(V(\mu);q^2)\id_{V(\mu)}
 =\left(\sum_{\lambda\subseteq\mu}
 a_\lambda(K;q^2)\sigma_\lambda\right)\Big|_{V(\mu)}.
\end{equation}
\end{proposition}

\begin{proof}
By \eqref{eq:BG-eigenvalue-exact}, the scalar by which
$\sigma_\lambda$ acts on the one-row module $V((r))$ is
\[
 F_{\lambda;n}
 (q^{2r+2n-2},q^{2n-4},\ldots,q^2,1;q^{-2}).
\]
This point is $\gamma_n(z_r)$.  By
\cref{lem:iterated-rank-reduction}, every multirow term in
\eqref{eq:BG-representationwise-expansion} vanishes on this curve and
\[
 \sigma_{(k)}|_{V((r))}
 =\varepsilon_nU_k(z_r)
 =\varepsilon_n\prod_{j=0}^{k-1}(1-q^{2(r-j)}).
\]
For $k>r$, the factor with $j=r$ is zero.  Thus evaluating the finite
representationwise identity \eqref{eq:BG-representationwise-expansion} on
$V((r))$ gives a finite sum and yields
\eqref{eq:one-sided-color}.  The coefficients are Laurent integral by
\cref{prop:BG-coefficient-integrality}.
\end{proof}

The factors in \eqref{eq:one-sided-color} detect only the zeros
$q^{2(r-j)}=1$.  Reflection in the Harish--Chandra variable and descent
through $X=z+z^{-1}$ will produce the second family $\{r+n+j\}$.

\subsection{Coefficient completion and cyclotomic-local disks}

For $m\ge0$, let
\[
 (Q;Q)_m=\prod_{a=1}^{m}(1-Q^a),
 \qquad
 I_m=((Q;Q)_m)\subset
 R_{\mathrm{ev}}=\Z[Q^{\pm1}],
\]
where the empty product is $1$.  Define
\begin{equation}\label{eq:BG-coefficient-completion}
 \widehat R_{\mathrm{ev}}^{\mathrm{BG}}
 =
 \varprojlim_{m\ge1}R_{\mathrm{ev}}/I_m.
\end{equation}
Since $1-(Q;Q)_m\in Q\Z[Q]$, the element $Q$ is invertible
modulo $I_m$.  Hence localization induces canonical isomorphisms
\[
 \Z[Q]/((Q;Q)_m)
 \cong
 \Z[Q^{\pm1}]/((Q;Q)_m),
\]
and therefore
\[
 \widehat R_{\mathrm{ev}}^{\mathrm{BG}}
 \cong
 \varprojlim_{m\ge1}\Z[Q]/((Q;Q)_m)
 =\widehat{\Z[Q]}.
\]
Thus $\widehat R_{\mathrm{ev}}^{\mathrm{BG}}$ is the Habiro
cyclotomic completion in the variable $Q$.  The superscript
$\mathrm{BG}$ records its role as the coefficient ring of the
Beliakova--Gorsky central completion.
For $d\ge1$ and $s\in\Z$, write
\begin{equation}\label{eq:local-disk-ring}
 u_s=z-z_s,
 \qquad
 \mathscr A_{d,s}=\widehat\cO_d[[u_s]],
 \qquad
 \mathfrak m_{d,s}=(\pi_d,u_s).
\end{equation}
The ring $\mathscr A_{d,s}$ is complete and separated for the
$\mathfrak m_{d,s}$-adic topology: it is the inverse limit of the
finite-length quotient rings $\widehat\cO_d[u_s]/(\pi_d,u_s)^N$.

\begin{lemma}
\label{lem:completion-comparison}
For every $d\ge1$, the inclusion
$R_{\mathrm{ev}}\hookrightarrow\cO_d$ extends uniquely to a continuous
homomorphism
\begin{equation}\label{eq:BG-to-cyclotomic-completion}
 \iota_d:\widehat R_{\mathrm{ev}}^{\mathrm{BG}}
 \longrightarrow\widehat\cO_d.
\end{equation}
More precisely,
\begin{equation}\label{eq:BG-ideal-valuation}
 \val_{\Phi_d}((q^2;q^2)_m)
 =\left\lfloor\frac{m}{d^\sharp}\right\rfloor.
\end{equation}
\end{lemma}

\begin{proof}
The factor $1-q^{2a}$ contains $\Phi_d(q)$ exactly when $d\mid2a$, or
$d^\sharp\mid a$, and then with multiplicity one.  Counting such $a$ for
$1\le a\le m$ gives \eqref{eq:BG-ideal-valuation}.  For $N\ge1$, take
the monotone cofinal choice $m(N)=Nd^\sharp$.  Then the image of
$I_{m(N)}$ in $\cO_d$ is contained in $(\pi_d^N)$, and the homomorphisms
\[
 R_{\mathrm{ev}}/I_{m(N)}
 \longrightarrow \cO_d/(\pi_d^N)
\]
are compatible with the transition maps as $N$ varies.  Passing to
inverse limits defines \eqref{eq:BG-to-cyclotomic-completion}.  Uniqueness
follows from density of $R_{\mathrm{ev}}$ in
$\widehat R_{\mathrm{ev}}^{\mathrm{BG}}$ and separatedness of
$\widehat\cO_d$.
\end{proof}

\begin{lemma}
\label{lem:Uk-order}
For every $s\in\Z$ and $k\ge0$,
\begin{equation}\label{eq:Uk-order-bound}
 U_k(z)\in
 \mathfrak m_{d,s}^{\,N_{d,s}(k)},
 \qquad
 N_{d,s}(k)=
 \#\{0\le j<k:j\equiv s\pmod{d^\sharp}\}.
\end{equation}
In particular,
\begin{equation}\label{eq:Uk-floor-bound}
 N_{d,s}(k)\ge\left\lfloor\frac{k}{d^\sharp}\right\rfloor,
 \qquad
 \ord_{\mathfrak m_{d,s}}U_k(z)\longrightarrow\infty.
\end{equation}
\end{lemma}

\begin{proof}
Write $z=z_s+u_s$.  If $j\equiv s\pmod{d^\sharp}$, then
$d\mid2(s-j)$, and therefore
\begin{align*}
 1-q^{-n-2j}z
 &=1-q^{-n-2j}z_s-q^{-n-2j}u_s\\
 &=1-q^{2(s-j)}-q^{-n-2j}u_s
 \in(\pi_d,u_s)=\mathfrak m_{d,s}.
\end{align*}
Multiplying the factors indexed by this congruence class gives
\eqref{eq:Uk-order-bound}.  In an interval of $k$ consecutive integers,
every residue class modulo $d^\sharp$ occurs at least
$\lfloor k/d^\sharp\rfloor$ times, proving
\eqref{eq:Uk-floor-bound}.
\end{proof}

\begin{proposition}[Cyclotomic-local knot germ]
\label{prop:cyclotomic-local-germ}
Fix a zero-framed knot $K$ and an integer $n\ge2$.  For every $d\ge1$
and $s\in\Z$, the series
\begin{equation}\label{eq:local-germ-series}
 \mathscr J_{K,s}(z)=
 \sum_{k\ge0}b_k^{(n)}(K;q)U_k(z)
\end{equation}
converges in $\mathscr A_{d,s}$.  If $i\in\Z_{\ge0}$ and
$z_i\equiv z_s\pmod{\pi_d}$, then continuous evaluation at
$u_s=z_i-z_s$ gives
\begin{equation}\label{eq:local-germ-node-value}
 \mathscr J_{K,s}(z_i)=J_i^{SU(n)}(K;q).
\end{equation}
\end{proposition}

\begin{proof}
By \cref{lem:Uk-order}, the summands in
\eqref{eq:local-germ-series} tend to zero
$\mathfrak m_{d,s}$-adically, so the series converges.  Since
$z_i-z_s\in\pi_d\widehat\cO_d$, evaluation at $u_s=z_i-z_s$ is
continuous.  It may therefore be applied termwise, and $U_k(z_i)=0$ for
$k>i$ reduces the result to \eqref{eq:one-sided-color}.
\end{proof}

\begin{lemma}
\label{lem:disk-compatibility}
Fix a zero-framed knot $K$, an integer $n\ge2$, and $d\ge1$.  If
$z_s\equiv z_t\pmod{\pi_d}$, substitution
\begin{equation}\label{eq:change-of-center}
 u_t=u_s+(z_s-z_t)
\end{equation}
defines a continuous isomorphism
$\mathscr A_{d,t}\cong\mathscr A_{d,s}$, under which
$\mathscr J_{K,t}$ and $\mathscr J_{K,s}$ agree.
\end{lemma}

\begin{proof}
The constant $z_s-z_t$ lies in $\pi_d\widehat\cO_d$, so
\eqref{eq:change-of-center} sends
$(\pi_d,u_t)$ into $(\pi_d,u_s)$ and has the inverse translation.  Hence
it is a continuous isomorphism.  Every finite partial sum
$\sum_{k=0}^{M}b_k^{(n)}U_k(z)$ is the same polynomial in the global
variable $z$ in the two coordinates.  The partial sums converge in both
complete separated disk rings by
\cref{prop:cyclotomic-local-germ}; continuity of the coordinate
change therefore identifies their limits.
\end{proof}

\subsection{Continuous restriction of the completed center}

Put
\begin{align*}
 \Lambda_n^{\mathrm{pol}}
 &=R_{\mathrm{ev}}[y_1,\ldots,y_n]^{\mathfrak S_n},&
 \Lambda_n^{\mathrm{Laur}}
 &=R_{\mathrm{ev}}[y_1^{\pm1},\ldots,y_n^{\pm1}]^{\mathfrak S_n},\\
 \eta_n&=y_1\cdots y_n,&
 G_\lambda(y)
 &=F_{\lambda;n}(y_1,\ldots,y_n;q^{-2}),
\end{align*}
and let
\[
 \mathcal S_m
 =I_m\Lambda_n^{\mathrm{pol}}
  +\sum_{\substack{\ell(\lambda)\le n\\|\lambda|>m}}
    R_{\mathrm{ev}}G_\lambda.
\]
By \cite[Theorem~8.1(a)]{BG24}, after the parameter translation, the
top homogeneous component of $G_\lambda$ is an
$R_{\mathrm{ev}}$-unit times the Schur polynomial $s_\lambda$, and all
remaining terms have lower degree.  Degree induction against the Schur
basis therefore shows that the family $\{G_\lambda:\ell(\lambda)\le n\}$
is an $R_{\mathrm{ev}}$-basis of $\Lambda_n^{\mathrm{pol}}$.  In
particular, every element of $\Lambda_n^{\mathrm{pol}}$ has a unique
finite expansion in the $G_\lambda$.

For $r\ge0$, set
\begin{equation}\label{eq:fN-definition}
 f_r(X)=(X;Q)_r
 =\prod_{a=0}^{r-1}(1-XQ^a),
\end{equation}
with $f_0(X)=1$.

In $R_{\mathrm{ev}}[y_1^{\pm1},\ldots,y_n^{\pm1}]$, let
$\mathcal C_N$ be the ideal generated by the shifted exact factorial
products
\begin{equation}\label{eq:Cartan-exact-products}
 (Q;Q)_a\prod_{j=1}^{n}f_{c_j}(Q^{s_j}y_j),
 \qquad
 a+\sum_{j=1}^{n}c_j=N,\quad
 (s_1,\ldots,s_n)\in\Z^n,
\end{equation}
and put
$\mathcal C_N^{\mathrm{sym}}
=\mathcal C_N\cap\Lambda_n^{\mathrm{pol}}$.
The coefficient and interpolation topologies are compared in
\cref{prop:exact-factorial-interpolation-bridge}.

\begin{lemma}
\label{lem:finite-interpolation-matrix}
For $m\ge0$, let
\[
 \mathscr P_m
 =\{\lambda:\ell(\lambda)\le n,\ |\lambda|\le m\},
 \qquad
 E_m=\bigl(G_\lambda(p_\mu)\bigr)_
       {\mu,\lambda\in\mathscr P_m},
\]
where
\[
 p_\mu=(Q^{\mu_1+n-1},Q^{\mu_2+n-2},\ldots,Q^{\mu_n}),
\]
and the rows and columns are ordered by any linear extension of
inclusion.  Then $E_m$ is lower triangular and
\[
 G_\lambda(p_\lambda)
 =u_\lambda\prod_{\square\in\lambda}
       (1-Q^{h(\square)}),
 \qquad u_\lambda\in R_{\mathrm{ev}}^\times.
\]
In particular,
\[
 D_m:=\det E_m\in R_{\mathrm{ev}}\setminus\{0\}
\]
is a Laurent unit times a product of cyclotomic polynomials, and
\[
 D_mE_m^{-1}=\operatorname{adj}(E_m)
 \in\operatorname{Mat}_{\mathscr P_m}(R_{\mathrm{ev}}).
\]

If $F=\sum_\lambda b_\lambda G_\lambda\in
\Lambda_n^{\mathrm{pol}}$ and
\[
 \mathbf v_m(F)=\bigl(F(p_\mu)\bigr)_{\mu\in\mathscr P_m},
 \qquad
 \mathbf b_m(F)=\bigl(b_\lambda\bigr)_{\lambda\in\mathscr P_m},
\]
then
\begin{equation}\label{eq:finite-evaluation-system}
 \mathbf v_m(F)=E_m\mathbf b_m(F).
\end{equation}
Consequently, for every ideal $J\subseteq R_{\mathrm{ev}}$,
\[
 F(p_\mu)\in J\quad(\mu\in\mathscr P_m)
 \quad\Longrightarrow\quad
 b_\lambda\in D_m^{-1}J
 \quad(\lambda\in\mathscr P_m).
\]
\end{lemma}

\begin{proof}
The interpolation vanishing says
$G_\lambda(p_\mu)=0$ unless $\lambda\subseteq\mu$, while
\cite[Theorem~8.2(c)]{BG24}, after the parameter translation, gives
the displayed nonzero diagonal value.  This proves triangularity and
the determinant formula.  The adjugate identity gives the assertion
about $E_m^{-1}$.

For \eqref{eq:finite-evaluation-system}, terms with
$\lambda\notin\mathscr P_m$ vanish at every $p_\mu$ with
$\mu\in\mathscr P_m$, because such a $\lambda$ cannot be contained in
$\mu$.  The last assertion follows by multiplying
\eqref{eq:finite-evaluation-system} by $D_mE_m^{-1}$.
\end{proof}

\begin{proposition}[Cofinality of the two filtrations]
\label{prop:exact-factorial-interpolation-bridge}
The two filtrations
$\{\mathcal S_m\}_{m\ge0}$ and
$\{\mathcal C_N^{\mathrm{sym}}\}_{N\ge0}$ are cofinal.  More precisely:
\begin{enumerate}[label=\textup{(\roman*)}]
\item for every $N$ there is $m(N)$ such that
$\mathcal S_{m(N)}\subseteq\mathcal C_N^{\mathrm{sym}}$;
\item for every $m$ there is $N(m)$ such that
$\mathcal C_{N(m)}^{\mathrm{sym}}\subseteq\mathcal S_m$.
\end{enumerate}
\end{proposition}

\begin{proof}
The ideals $\mathcal C_N$ are descending: one lowers either $a$ or one
of the $c_j$ in \eqref{eq:Cartan-exact-products} and uses
\[
 (Q;Q)_{a+1}=(1-Q^{a+1})(Q;Q)_a,\qquad
 f_{c+1}(Q^sy)=(1-Q^{s+c}y)f_c(Q^sy).
\]

Fix $N$.  Lemma~10.27 of \cite{BG24}, after replacing its parameter by
$Q^{-1}$ and reversing the factors in each factorial, expands every
$G_\lambda$ as a sum of products of
$d=\binom{n+1}{2}$ shifted factorials whose indices add to
$|\lambda|$.  If $|\lambda|>d(N-1)$, one index is at least $N$.
That factorial is a multiple of its first $N$ factors and hence the
whole term belongs to $\mathcal C_N$.  Thus
$G_\lambda\in\mathcal C_N$ for $|\lambda|>d(N-1)$.  If
$m(N)\ge\max\{N,d(N-1)\}$, then
$I_{m(N)}\subseteq\mathcal C_N$ as well, proving (i).

For the converse, fix $m$ and use
\cref{lem:finite-interpolation-matrix}.  Since $D_m$ is a product of
cyclotomic polynomials up to a Laurent unit, the cyclotomic valuation
of $(Q;Q)_t$ tends to infinity with $t$ at every irreducible factor of
$D_m(Q;Q)_m$.  We may therefore choose $t=t(m)$ such that
\begin{equation}\label{eq:finite-interpolation-denominator-absorption}
 (Q;Q)_t\in D_m (Q;Q)_mR_{\mathrm{ev}}.
\end{equation}

For every integer $r$ and $c\ge0$,
\begin{equation}\label{eq:shifted-factorial-value-divisibility}
 f_c(Q^r)
 =\prod_{b=0}^{c-1}(1-Q^{r+b})
 \in(Q;Q)_cR_{\mathrm{ev}}.
\end{equation}
If the displayed interval contains zero the product vanishes; otherwise
the quotient by $(Q;Q)_c$ is, up to a Laurent unit, a Gaussian
polynomial.  Now take $M=(n+1)t$.  In every generator
\eqref{eq:Cartan-exact-products} of $\mathcal C_M$, at least one of
$a,c_1,\ldots,c_n$ is at least $t$.  Equation
\eqref{eq:shifted-factorial-value-divisibility} shows that its value at
every $p_\mu$ is divisible by $(Q;Q)_t$.  The same is true for every
element of $\mathcal C_M$, since the node coordinates are Laurent
units.

Let $F\in\mathcal C_M^{\mathrm{sym}}$ and write
$F=\sum_\lambda b_\lambda G_\lambda$.  The preceding evaluation
divisibility gives
$\mathbf v_m(F)\in(Q;Q)_t
R_{\mathrm{ev}}^{\mathscr P_m}$.  By
\cref{lem:finite-interpolation-matrix}, for $|\lambda|\le m$,
\[
 b_\lambda\in D_m^{-1}(Q;Q)_tR_{\mathrm{ev}}
 \subseteq (Q;Q)_mR_{\mathrm{ev}}
\]
by \eqref{eq:finite-interpolation-denominator-absorption}.  There is no
condition on the coefficients with $|\lambda|>m$, so
$F\in\mathcal S_m$.  This proves (ii).
\end{proof}

\begin{corollary}
\label{cor:interpolation-completion-structure}
The interpolation topology has continuous multiplication.  Its completion
contains $\Lambda_n^{\mathrm{Laur}}$, and the substitutions
\[
 y_i\longmapsto Q^c y_i,
 \qquad
 y_i\longmapsto y_{n+1-i}^{-1}
 \quad(c\in\Z)
\]
extend to continuous automorphisms.
\end{corollary}

\begin{proof}
Since the ideals $\mathcal C_N$ are multiplicative, cofinality makes
multiplication continuous in the interpolation topology.  The formulas of
\cite[Proposition~11.9 and Corollary~11.10]{BG24} show that multiplication
by $\eta_n^{\pm1}$ preserves the open interpolation submodules.  Thus
$\eta_n$ is a unit in the completion and
$\Lambda_n^{\mathrm{Laur}}$ lies in it.

Scaling changes only the shifts in
\eqref{eq:Cartan-exact-products}, while inversion is governed by
\[
 f_r(Q^sy_i^{-1})=(-1)^rQ^{sr+\binom r2}y_i^{-r}
 f_r(Q^{1-r-s}y_i).
\]
Hence both substitutions preserve the factorial filtration and extend,
together with their inverses, to continuous automorphisms.
\end{proof}

Put
\[
 \cZ_{n,\mathrm{fin}}
 =\bigoplus_{\ell(\lambda)\le n}
   R_{\mathrm{ev}}\sigma_\lambda
 \subset Z(U_{\Z})^\Gamma.
\]
Under the shifted coordinates $y_i=q^{n-1}x_i$, one has
$\hc(\sigma_\lambda)=G_\lambda(y)$.  The basis property established
above implies that every product $G_\lambda G_\mu$ has a finite
expansion in the same basis with coefficients in $R_{\mathrm{ev}}$.
The algebraic Harish--Chandra isomorphism therefore identifies
$\cZ_{n,\mathrm{fin}}$ with $\Lambda_n^{\mathrm{pol}}$ and shows that
$\cZ_{n,\mathrm{fin}}$ is a subalgebra.  For
$m\ge0$, let
\[
 \cZ_{n,\mathrm{fin}}^{(m)}
 =I_m\cZ_{n,\mathrm{fin}}
  +\bigoplus_{\substack{\ell(\lambda)\le n\\|\lambda|>m}}
    R_{\mathrm{ev}}\sigma_\lambda.
\]
These submodules give $\cZ_{n,\mathrm{fin}}$ the interpolation-basis
linear topology.  Define the Beliakova--Gorsky even-coefficient central
sector by
\begin{equation}\label{eq:BG-central-sector-definition}
 \widehat\cZ_n^{\mathrm{BG}}
 =\varprojlim_m
   \cZ_{n,\mathrm{fin}}/\cZ_{n,\mathrm{fin}}^{(m)},
\end{equation}
where the inverse limit is taken as a linear completion.  Under the
algebraic Harish--Chandra isomorphism,
$\cZ_{n,\mathrm{fin}}^{(m)}$ maps to $\mathcal S_m$.
By \cref{cor:interpolation-completion-structure}, multiplication is
continuous and the Harish--Chandra map extends to a topological algebra
isomorphism onto the corresponding completed symmetric polynomial ring.

\begin{lemma}
\label{lem:center-linear-completion}
The interpolation-basis filtration on $\cZ_{n,\mathrm{fin}}$ is
separated, and for every $m\ge0$ there is a canonical isomorphism
\begin{equation}\label{eq:center-finite-quotient}
 \cZ_{n,\mathrm{fin}}/\cZ_{n,\mathrm{fin}}^{(m)}
 \cong
 \bigoplus_{\substack{\ell(\lambda)\le n\\|\lambda|\le m}}
 (R_{\mathrm{ev}}/I_m)\sigma_\lambda.
\end{equation}
Consequently, taking inverse limits identifies
\begin{equation}\label{eq:center-coefficient-product}
 \widehat\cZ_n^{\mathrm{BG}}
 \cong
 \prod_{\ell(\lambda)\le n}
 \widehat R_{\mathrm{ev}}^{\mathrm{BG}}\sigma_\lambda
\end{equation}
as topological modules.
\end{lemma}

\begin{proof}
For a nonzero Laurent polynomial $f\in\Z[Q^{\pm1}]$, let
$\operatorname{wd}_Q(f)$ be the difference between its largest and
smallest $Q$-exponents.  Since the coefficient ring is a domain,
$\operatorname{wd}_Q(fg)=\operatorname{wd}_Q(f)+
\operatorname{wd}_Q(g)$ for nonzero $f,g$, whereas
\[
 \operatorname{wd}_Q((Q;Q)_m)=\sum_{a=1}^{m}a=\frac{m(m+1)}2.
\]
No nonzero Laurent polynomial can therefore belong to every $I_m$; hence
$\bigcap_m I_m=0$.

The $\sigma_\lambda$ form an algebraic $R_{\mathrm{ev}}$-basis.  Modulo
$\cZ_{n,\mathrm{fin}}^{(m)}$, all coordinates with $|\lambda|>m$
vanish and each remaining coordinate is reduced modulo $I_m$, which
proves \eqref{eq:center-finite-quotient}.  There are only finitely many
partitions of size at most $m$.  The preceding intersection result,
applied to the finitely many coefficients of an algebraic element, also
proves separatedness.

For a fixed $\lambda$, the levels $m\ge|\lambda|$ are cofinal and its
coordinate inverse limit is
$\varprojlim_m R_{\mathrm{ev}}/I_m=
\widehat R_{\mathrm{ev}}^{\mathrm{BG}}$.  Conversely, a family of such
completed coefficients defines at level $m$ the finite sum over
$|\lambda|\le m$, with every coefficient reduced modulo $I_m$.
These finite sums are compatible with the transition maps.  This gives
\eqref{eq:center-coefficient-product} in both directions.
\end{proof}

With this identification, every element has a unique expansion
\begin{equation}\label{eq:general-completed-center-series}
 c=\sum_{\ell(\lambda)\le n}c_\lambda\sigma_\lambda,
 \qquad c_\lambda\in\widehat R_{\mathrm{ev}}^{\mathrm{BG}}.
\end{equation}
The series is interpreted quotientwise: for each fixed level $m$, its
truncations by partition size stabilize after size $m$ in the finite
quotient \eqref{eq:center-finite-quotient}.  Applying this description to
the Laurent coefficients in \eqref{eq:BG-finite-dual-coefficient} gives
\begin{equation}\label{eq:BG-knot-element-central-sector}
 J_K^{\mathrm{int}}(\mathfrak{gl}_n;q^2)
 =\sum_\lambda a_\lambda(K;q^2)\sigma_\lambda
 \in\widehat\cZ_n^{\mathrm{BG}}.
\end{equation}
\begin{proposition}[One-row restriction]
\label{prop:continuous-completed-restriction}
Fix an integer $n\ge2$. For every $d\ge1$ and $s\in\Z$, there is a unique continuous
$\widehat R_{\mathrm{ev}}^{\mathrm{BG}}$-algebra homomorphism
\begin{equation}\label{eq:completed-center-restriction-map}
\operatorname{res}_{d,s}:\widehat\cZ_n^{\mathrm{BG}}
 \longrightarrow\mathscr A_{d,s}
\end{equation}
extending algebraic Harish--Chandra restriction along $\gamma_n$, where
$\mathscr A_{d,s}$ is regarded as an
$\widehat R_{\mathrm{ev}}^{\mathrm{BG}}$-algebra through the coefficient map
$\iota_d$.  The restriction is
given by
\begin{equation}\label{eq:completed-restriction-series}
 \operatorname{res}_{d,s}(c)
 =\sum_{k\ge0}\iota_d(c_{(k)})\varepsilon_nU_k(z),
\end{equation}
and hence
\begin{equation}\label{eq:restriction-on-sigma}
 \operatorname{res}_{d,s}(\sigma_\lambda)=
 \begin{cases}
  \varepsilon_nU_k(z),&\lambda=(k),\\
  0,&\ell(\lambda)\ge2.
 \end{cases}
\end{equation}
For every zero-framed knot $K$, the interpolation-sector knot element
satisfies
\begin{equation}\label{eq:restriction-of-knot}
 \operatorname{res}_{d,s}
 (J_K^{\mathrm{int}}(\mathfrak{gl}_n;q^2))=\mathscr J_{K,s}(z).
\end{equation}
\end{proposition}

\begin{proof}
Formula \eqref{eq:restriction-on-sigma} is
\cref{lem:iterated-rank-reduction}.  Modulo $\mathfrak m_{d,s}^N$, only
the terms with $k<Nd^\sharp$ occur in
\eqref{eq:completed-restriction-series}, by \cref{lem:Uk-order}.  Hence
the series defines a continuous map.  On the dense algebraic subalgebra it
is ordinary Harish--Chandra restriction along $\gamma_n$; continuity of
multiplication therefore extends multiplicativity to the completion.

Along the one-row curve, the determinant coordinate restricts to the unit
\begin{equation}\label{eq:gamma-determinant}
 \prod_{i=1}^{n}\gamma_n(z)_i=q^{n(n-2)}z.
\end{equation}
Thus the map also agrees with Laurent-polynomial restriction.  Uniqueness
follows from density.  Finally, substituting
$c_\lambda=a_\lambda(K;q^2)$ and using
\eqref{eq:b-coefficients} gives \eqref{eq:restriction-of-knot}.
\end{proof}

\section{Completed Harish--Chandra reflection and descent}
\label{sec:reflection}

This section proves the reflection symmetry of the completed knot element
and its integral descent to the quotient coordinate $X=z+z^{-1}$.

\subsection{Algebraic automorphisms and their central extensions}

Define maps on the generic quantum group by
\begin{align}
 \vartheta(E_i)&=E_{n-i},&
 \vartheta(F_i)&=F_{n-i},&
 \vartheta(K_j)&=K_{n+1-j}^{-1},
 \label{eq:vartheta-definition}\\
 \delta_c(E_i)&=E_i,&
 \delta_c(F_i)&=F_i,&
 \delta_c(K_j)&=q^cK_j
 \qquad(c\in\Z).
 \label{eq:delta-definition}
\end{align}
Thus
\begin{equation}\label{eq:vartheta-delta-on-L}
 \vartheta(L_i)=L_{n-i},
 \qquad
 \delta_c(L_i)=L_i.
\end{equation}

\begin{proposition}
\label{prop:algebraic-automorphisms}
The map $\vartheta$ is an involutive Hopf algebra automorphism, and each
$\delta_c$ is an algebra automorphism with
$\delta_c\delta_{c'}=\delta_{c+c'}$.
The restriction of $\delta_c$ to the embedded
$U_{q^2}(\mathfrak{sl}_n)$ is the identity.
\end{proposition}

\begin{proof}
It is enough to check the defining relations and coproduct on the
generators.  The weight exponent in the relation
$K_jE_iK_j^{-1}=q^{\delta_{j,i}-\delta_{j,i+1}}E_i$ is preserved because
\[
 -\bigl(\delta_{n+1-j,n-i}-\delta_{n+1-j,n-i+1}\bigr)
 =\delta_{j,i}-\delta_{j,i+1}.
\]
The corresponding weight identity for $F_i$ is obtained by reversing the
exponent.  The commutator relation is preserved by
\eqref{eq:vartheta-delta-on-L}, and the Serre relations are permuted.  Since
$\vartheta(L_i)=L_{n-i}$,
\[
 (\vartheta\otimes\vartheta)\Delta(E_i)
 =E_{n-i}\otimes1+L_{n-i}\otimes E_{n-i}
 =\Delta(\vartheta(E_i)),
\]
and the coproduct formulas for $F_i$ and $K_j$ are preserved in the same
way.  The formulas also give $\vartheta^2=\id$, so $\vartheta$ is an
involutive Hopf algebra automorphism.  The maps $\delta_c$ preserve all
algebra relations, and $\delta_{-c}$ is the inverse of $\delta_c$; they need
not be Hopf because of the scalar on the group-like elements $K_j$.

The subalgebra $U_{q^2}(\mathfrak{sl}_n)$ is generated by the $E_i$,
$F_i$, and $L_i^{\pm1}$, all of which are fixed by $\delta_c$.  Hence
$\delta_c$ restricts to the identity there.
\end{proof}

On the algebraic center the Harish--Chandra actions of
$\delta_c$ and $\vartheta$ are, respectively,
\[
 y_i\longmapsto Q^c y_i,\qquad
 y_i\longmapsto Q^{n-1}y_{n+1-i}^{-1}.
\]
\Cref{cor:interpolation-completion-structure} therefore defines
continuous automorphisms, still denoted $\delta_c$ and $\vartheta$, of
$\widehat\cZ_n^{\mathrm{BG}}$ by transporting these substitutions
through the completed Harish--Chandra map.  They agree with the
algebraic quantum-group automorphisms on
$\cZ_{n,\mathrm{fin}}$.

The substitution $q\mapsto e^{h/2}$ defines an injective field
homomorphism
\begin{equation}\label{eq:q-hadic-coefficient-embedding}
 \Bbbk=\Q(q)\lhook\joinrel\longrightarrow\Q((h)).
\end{equation}
Indeed, if a Laurent polynomial $\sum_{m=a}^{b}c_mq^m$ maps to zero, then
all derivatives at $h=0$ of $\sum_m c_me^{mh/2}$ vanish.  The first
$b-a+1$ derivatives form a Vandermonde system in the distinct exponents
$m$, so every $c_m$ is zero; the assertion for rational functions follows
by applying this to the numerator.

The coefficient embedding in \eqref{eq:q-hadic-coefficient-embedding},
together with $E_i\mapsto E_i$, $F_i\mapsto F_i$, and
$K_j\mapsto e^{hH_j/2}$, defines the generic base-change homomorphism into
$U_h(\mathfrak{gl}_n)[h^{-1}]$.  On the generators of the integral form
listed in \cite[Section~3.1]{BG24}, its values are
\[
 q\longmapsto e^{h/2},\qquad
 K_j^{\pm1}\longmapsto e^{\pm hH_j/2},\qquad
 e_i\longmapsto(e^{h/2}-e^{-h/2})E_i,
 \qquad
 F_i^{(a)}\longmapsto\frac{F_i^a}{[a]_{e^{h/2}}!}.
\]
The last denominator has constant term $a!$ and is therefore a unit in
$\Q[[h]]$.  Thus the generic map carries $U_{\Z}$ into
$U_h(\mathfrak{gl}_n)$; denote its restriction by
\begin{equation}\label{eq:integral-hadic-specialization}
 \operatorname{sp}_h:U_{\Z}\longrightarrow U_h(\mathfrak{gl}_n).
\end{equation}
After inverting $h$, this map is the restriction of the generic base-change
homomorphism to $U_h(\mathfrak{gl}_n)[h^{-1}]$.  Since
$U_h(\mathfrak{gl}_n)$ is $h$-torsion-free, its map to this localization
is injective.  Thus, if $z\in Z(U_{\Z})^\Gamma$, the element
$\operatorname{sp}_h(z)$ commutes with all Drinfeld--Jimbo generators
in the localization and hence already in $U_h(\mathfrak{gl}_n)$.
Therefore $\operatorname{sp}_h(z)$ belongs to
$Z_h=Z(U_h(\mathfrak{gl}_n))$.  Let $\operatorname{sp}_h^0$ be the
induced map on the Cartan part.  It sends
$x_i=K_i^2$ to $e^{hH_i}$ and, in the shifted coordinates used for the
interpolation basis,
\begin{equation}\label{eq:shifted-Cartan-specialization}
 y_i=q^{n-1}K_i^2
 \longmapsto
 \exp\!\left(h\left(H_i+\frac{n-1}{2}\right)\right).
\end{equation}
Let $\hc_h:Z_h\to U_h^0$ be the shifted $h$-adic Harish--Chandra map in the
convention of \textup{(DJ1)}.

\begin{lemma}
\label{lem:Cartan-highest-weight-separation}
Let $A\in U_h^0$.  If the evaluation of $A$ at the shifted highest
weight of $V_h(\mu)$ is zero for every partition $\mu$ with at most
$n$ parts, then $A=0$.
\end{lemma}

\begin{proof}
Write
\[
 A=\sum_{r\ge0}h^rP_r(H_1,\ldots,H_n),
 \qquad P_r\in\Q[H_1,\ldots,H_n].
\]
Each $P_r$ vanishes on a fixed translate of the dominant integral
lattice.  After the change of variables
$d_i=\mu_i-\mu_{i+1}$ and $d_n=\mu_n$, this set is an affine translate
of $\Z_{\ge0}^n$ and is therefore Zariski dense.  Hence every $P_r$ is
zero.
\end{proof}

\begin{lemma}[Harish--Chandra specialization]
\label{lem:HC-specialization-compatibility}
For every $z\in Z(U_{\Z})^\Gamma$,
\begin{equation}\label{eq:HC-specialization-compatibility}
 \hc_h\bigl(\operatorname{sp}_h(z)\bigr)
 =\operatorname{sp}_h^0\bigl(\hc(z)\bigr).
\end{equation}
In particular, if
$\sigma_{\lambda,h}=\operatorname{sp}_h(\sigma_\lambda)$, then
\begin{equation}\label{eq:HC-specialization-on-sigma}
 \hc_h(\sigma_{\lambda,h})
 =\operatorname{sp}_h^0(G_\lambda).
\end{equation}
\end{lemma}

\begin{proof}
Fix a partition $\mu$ with at most $n$ parts, and let $\chi_\mu(z)$ be
the scalar by which $z$ acts on $V(\mu)$.  By the defining
highest-weight characterization of the algebraic Harish--Chandra map,
$\hc(z)$ evaluated at the corresponding $\rho$-shifted weight equals
$\chi_\mu(z)$.  Specializing the action matrices by
\eqref{eq:integral-hadic-specialization} shows that
$\operatorname{sp}_h(z)$ acts on $V_h(\mu)$ by the specialized scalar
$\operatorname{sp}_h(\chi_\mu(z))$.  The defining characterization of
$\hc_h$ therefore gives the same value when
$\hc_h(\operatorname{sp}_h(z))$ is evaluated at the shifted highest
weight of $V_h(\mu)$.  The element
$\operatorname{sp}_h^0(\hc(z))$ has that value as well, since
$\operatorname{sp}_h^0$ specializes both the Cartan variables and the
multiplicative $\rho$-shift.

Thus the two sides of
\eqref{eq:HC-specialization-compatibility} have the same shifted
highest-weight evaluations for all $\mu$.  They are equal by
\cref{lem:Cartan-highest-weight-separation}.  The last assertion follows
from $\hc(\sigma_\lambda)=G_\lambda$.
\end{proof}

\begin{lemma}[$h$-adic saturation]
\label{lem:HC-hadic-saturation}
For every $z\in Z_h$ and every $m\ge0$,
\begin{equation}\label{eq:HC-saturation}
 \hc_h(z)\in h^mU_h^0
 \quad\Longrightarrow\quad
 z\in h^mZ_h.
\end{equation}
In particular, $\hc_h$ is injective.
\end{lemma}

\begin{proof}
The PBW theorem makes $U_h(\mathfrak{gl}_n)$ topologically free, hence
$h$-torsion-free, over $\Q[[h]]$.  Reduction modulo $h$ identifies it
with $U(\mathfrak{gl}_n)$.  Reducing the triangular decomposition, its
weight-zero PBW projection, and the $\rho$-shift shows that the reduction
of $\hc_h$ is the classical shifted Harish--Chandra map.  The classical
Harish--Chandra isomorphism for $\mathfrak{gl}_n$ follows from the
semisimple case recalled in \cite[Chapter~6]{Jantzen96} and the splitting
$\mathfrak{gl}_n=\mathfrak{sl}_n\oplus\Q I_n$.

If $\hc_h(z)\in hU_h^0$, the reduction $\bar z$ has zero classical
Harish--Chandra image, hence $\bar z=0$.  Thus $z=hz_1$ with $z_1$ central.
Iteration proves \eqref{eq:HC-saturation}; the case $\hc_h(z)=0$ gives
injectivity.
\end{proof}

\begin{corollary}
\label{cor:HC-highest-weight-separation}
Finite-dimensional type-one highest-weight modules separate $Z_h$.
\end{corollary}

\begin{proof}
If $z\in Z_h$ acts by zero on every $V_h(\mu)$, then all shifted
highest-weight evaluations of $\hc_h(z)$ vanish.  By
\cref{lem:Cartan-highest-weight-separation}, $\hc_h(z)=0$, and
\cref{lem:HC-hadic-saturation} gives $z=0$.
\end{proof}

\begin{proposition}[Faithful $h$-adic realization]
\label{prop:faithful-hadic-comparison}
Let $U_h(\mathfrak{gl}_n)$ be the $h$-adically complete
Drinfeld--Jimbo algebra over $\Q[[h]]$, in the convention
\[
 q=e^{h/2},\qquad K_j=e^{hH_j/2}.
\]
The restriction of $\operatorname{sp}_h$ to
$\cZ_{n,\mathrm{fin}}$ extends uniquely to a continuous injective
algebra homomorphism
\begin{equation}\label{eq:central-hadic-injection}
 \iota_h^{\cZ}:\widehat{\cZ}^{\mathrm{BG}}_n
 \lhook\joinrel\longrightarrow Z_h.
\end{equation}
It intertwines the preceding central automorphisms with
\[
 \vartheta_h(H_j)=-H_{n+1-j},\qquad
 \delta_{c,h}(H_j)=H_j+c,
\]
and their values on $E_i,F_i$ from
\eqref{eq:vartheta-definition}--\eqref{eq:delta-definition}:
\begin{equation}\label{eq:intrinsic-hadic-equivariance}
 \iota_h^{\cZ}\circ\vartheta
 =\vartheta_h\circ\iota_h^{\cZ},
 \qquad
 \iota_h^{\cZ}\circ\delta_c
 =\delta_{c,h}\circ\iota_h^{\cZ}.
\end{equation}
\end{proposition}

\begin{proof}
Under the specialization $Q=e^h$, the coefficient homomorphism
\begin{equation}\label{eq:Habiro-to-hadic-coefficients}
 \tau_h:\widehat R_{\mathrm{ev}}^{\mathrm{BG}}
 \cong\widehat{\Z[Q]}\longrightarrow\Q[[h]]
\end{equation}
is injective.  Indeed, Habiro's Taylor homomorphism
\[
 T_1:\widehat{\Z[Q]}\longrightarrow\Z[[Q-1]]
\]
is injective by \cite[Theorem~5.4]{Habiro04}.  Substitution
$Q-1\mapsto e^h-1$ preserves the order of every nonzero power series,
and its composite with $T_1$ is $\tau_h$.

Lemma~10.27 of \cite{BG24}, after the parameter translation, writes
the Harish--Chandra polynomial $G_\lambda$ as a finite sum of products
\begin{equation}\label{eq:interpolation-total-factorial-degree}
 \prod_{\nu=1}^{d}
 f_{j_\nu}(Q^{s_\nu}y_{i_\nu}),
 \qquad
 \sum_{\nu=1}^{d}j_\nu=|\lambda|,
\end{equation}
with integral Laurent coefficients.  Under
\eqref{eq:shifted-Cartan-specialization}, a factorial of index $j$ becomes
\[
 \prod_{b=0}^{j-1}
 \left(1-
 \exp\!\left(h\left(H_i+\frac{n-1}{2}+s+b\right)\right)\right),
\]
which belongs to $h^jU_h^0$.  Equations
\eqref{eq:HC-specialization-on-sigma} and
\eqref{eq:interpolation-total-factorial-degree} therefore give
\[
 \hc_h(\sigma_{\lambda,h})\in h^{|\lambda|}U_h^0.
\]
By \cref{lem:HC-hadic-saturation},
\begin{equation}\label{eq:sigma-hadic-order}
 \sigma_{\lambda,h}\in h^{|\lambda|}Z_h
 \subset h^{|\lambda|}U_h(\mathfrak{gl}_n).
\end{equation}
Also $\tau_h(I_m)\subset h^m\Q[[h]]$.  It follows from
\eqref{eq:sigma-hadic-order} that the algebraic central inclusion sends
$\cZ_{n,\mathrm{fin}}^{(m)}$ into
$h^mU_h(\mathfrak{gl}_n)$.  Explicitly, for
$z=\sum_\lambda c_\lambda\sigma_\lambda$ set
\[
 \iota_h^{\cZ}(z)
 =\sum_\lambda\tau_h(c_\lambda)\sigma_{\lambda,h}.
\]
For each $N$, only the finitely many $\lambda$ with $|\lambda|<N$
can contribute modulo $h^N$, by
\eqref{eq:sigma-hadic-order}; hence the sum is well defined and
continuous.  On $\cZ_{n,\mathrm{fin}}$ it agrees with
$\operatorname{sp}_h$.  Since multiplication is continuous in the
source and in the $h$-adic target, density of
$\cZ_{n,\mathrm{fin}}$ shows that the extension is an algebra
homomorphism.

To prove equivariance, let
\begin{equation}\label{eq:determinant-central-element}
 \mathbf d_n=\hc^{-1}(\eta_n)\in\cZ_{n,\mathrm{fin}}.
\end{equation}
Let $K_{\mathrm{det}}=K_1\cdots K_n$.  Since $K_{\mathrm{det}}$ is
central and the total $\rho$-shift is zero, the shifted Harish--Chandra
map sends $K_{\mathrm{det}}^2$ to $x_1\cdots x_n$.  Therefore
\begin{equation}\label{eq:determinant-central-monomial}
 \mathbf d_n=q^{n(n-1)}K_{\mathrm{det}}^2,
 \qquad
 \mathbf d_n^{\pm1}\in Z(U_{\Z})^\Gamma.
\end{equation}
Every element of $\Lambda_n^{\mathrm{Laur}}$ becomes a symmetric
polynomial after multiplication by a sufficiently large power of
$\eta_n$.  Hence the Harish--Chandra map identifies
\begin{equation}\label{eq:Laurent-central-subalgebra}
 \cZ_{n,\mathrm{Laur}}
 :=\cZ_{n,\mathrm{fin}}[\mathbf d_n^{-1}]
 \quad\text{with}\quad
 \Lambda_n^{\mathrm{Laur}}.
\end{equation}
The determinant-inverse construction in
\cref{cor:interpolation-completion-structure} places
$\mathbf d_n^{-1}$ in $\widehat\cZ_n^{\mathrm{BG}}$.  Since
$\cZ_{n,\mathrm{Laur}}$ contains the dense subalgebra
$\cZ_{n,\mathrm{fin}}$, it is dense in the completion.  Since
$\iota_h^{\cZ}$ is an algebra homomorphism and agrees with
$\operatorname{sp}_h$ on $\cZ_{n,\mathrm{fin}}$,
\begin{equation}\label{eq:determinant-inverse-specialization}
 \iota_h^{\cZ}(\mathbf d_n^{-1})
 =\iota_h^{\cZ}(\mathbf d_n)^{-1}
 =\operatorname{sp}_h(\mathbf d_n)^{-1}
 =\operatorname{sp}_h(\mathbf d_n^{-1}).
\end{equation}
Thus $\iota_h^{\cZ}$ and $\operatorname{sp}_h$ agree on
$\cZ_{n,\mathrm{Laur}}$.  The Harish--Chandra formulas in the preceding
paragraph show that $\vartheta$ and $\delta_c$ preserve this Laurent
subalgebra.  The formulas on the Drinfeld--Jimbo generators give, on all of
$U_{\Z}$ and hence on $\cZ_{n,\mathrm{Laur}}$,
\[
 \operatorname{sp}_h\circ\vartheta
 =\vartheta_h\circ\operatorname{sp}_h,
 \qquad
 \operatorname{sp}_h\circ\delta_c
 =\delta_{c,h}\circ\operatorname{sp}_h.
\]
The subalgebra is dense, and the two sides of each identity define
continuous maps from $\widehat\cZ_n^{\mathrm{BG}}$ to $Z_h$.  The
identities therefore extend to the completed center and prove
\eqref{eq:intrinsic-hadic-equivariance}.

For injectivity, consider the central characters on finite-dimensional
modules.  For a partition $\mu$ with at most $n$ parts, let
\[
 \chi_\mu:\widehat{\cZ}^{\mathrm{BG}}_n
 \longrightarrow\widehat R_{\mathrm{ev}}^{\mathrm{BG}}
\]
be scalar action on $V(\mu)$.  It is a finite sum because
$\sigma_\lambda$ annihilates $V(\mu)$ unless
$\lambda\subseteq\mu$.  Compatibility with the $h$-adic module gives
\begin{equation}\label{eq:central-scalar-hadic-compatibility}
 \rho_{\mu,h}\bigl(\iota_h^{\cZ}(z)\bigr)
 =\tau_h\bigl(\chi_\mu(z)\bigr)\id_{V_h(\mu)}.
\end{equation}
If $\iota_h^{\cZ}(z)=0$, injectivity of $\tau_h$ gives
$\chi_\mu(z)=0$ for every $\mu$.  Write
$z=\sum_\lambda c_\lambda\sigma_\lambda$ and induct on $|\mu|$.
After all proper subpartitions have been treated, evaluation on
$V(\mu)$ gives
\[
 0=c_\mu\,
   \sigma_\mu|_{V(\mu)}.
\]
By \cite[Theorem~8.2(c)]{BG24}, the diagonal scalar
$\sigma_\mu|_{V(\mu)}$ is a Laurent unit times
$\prod_{\square\in\mu}(1-Q^{h(\square)})$, hence is nonzero.
The coefficient ring is a domain because its Taylor map $T_1$ is
injective into the domain $\Z[[Q-1]]$.  Thus $c_\mu=0$, completing the
induction and proving injectivity.
\end{proof}

\begin{corollary}
\label{cor:knot-hadic-identification}
For every zero-framed knot $K$,
\begin{equation}\label{eq:knot-hadic-identification}
 \iota_h^{\cZ}\!
 \left(J_K^{\mathrm{int}}(\mathfrak{gl}_n;q^2)\right)
 =J_{K,h}(\mathfrak{gl}_n;q^2),
\end{equation}
where the right side is the usual universal $h$-adic knot element in the
convention \eqref{eq:Hopf}--\eqref{eq:HL-R-convention}.
\end{corollary}

\begin{proof}
The series in \eqref{eq:BG-knot-element-central-sector} converges
by definition of the interpolation topology.  The finite identity
\textup{(BG3)} says that on every $V(\mu)$ its value is the scalar of the
universal knot element; only finitely many subpartitions of $\mu$
contribute.  Therefore its image under $\iota_h^{\cZ}$ and
$J_{K,h}(\mathfrak{gl}_n;q^2)$ act by the same scalar on every
finite-dimensional highest-weight module.  Both elements are central; the
centrality of the universal knot element follows from the one-component
adjoint-invariance statement recalled before
\eqref{eq:reduced-normalization}.  The separation assertion of
\cref{cor:HC-highest-weight-separation} proves
\eqref{eq:knot-hadic-identification}.
\end{proof}

In view of \cref{cor:knot-hadic-identification}, we henceforth write
$J_K(\mathfrak{gl}_n;q^2)$ for the interpolation-sector element
$J_K^{\mathrm{int}}(\mathfrak{gl}_n;q^2)$; under
$\iota_h^{\cZ}$ it is the universal element $J_{K,h}$.

\begin{lemma}[Ribbon normalization of the diagram involution]
\label{lem:vartheta-ribbon}
In the convention
$q=e^{h/2}$,
$K_j=e^{hH_j/2}$,
$\mathcal R=D\Psi$ with
$\Psi=\Theta_{\mathrm{HL}}^{-1}$ as in
\eqref{eq:HL-R-convention}, and
$g=K_{-2\rho}=u\mathbf r^{-1}$, the continuous Hopf
automorphism
\[
 \vartheta_h(H_j)=-H_{n+1-j},\qquad
 \vartheta_h(E_i)=E_{n-i},\qquad
 \vartheta_h(F_i)=F_{n-i}
\]
is a ribbon Hopf automorphism:
\[
 (\vartheta_h\otimes\vartheta_h)(\mathcal R)=\mathcal R,
 \qquad
 \vartheta_h(g)=g,
 \qquad
 \vartheta_h(\mathbf r)=\mathbf r.
\]
\end{lemma}

\begin{proof}
The induced isometry
$\epsilon_j\mapsto-\epsilon_{n+1-j}$ permutes the positive roots and
fixes $\rho$.  For the full $\mathfrak{gl}_n$ Cartan factor
\[
 D=\exp\!\left(\frac h2\sum_{j=1}^nH_j\otimes H_j\right)
\]
one has
$(\vartheta_h\otimes\vartheta_h)(D)=D$.
The explicit normalized quasi-$R$ matrix is given in
\cite[Sections~3G1--3G2, equations~(64)--(70)]{HL16}.
Its image under $\vartheta_h\otimes\vartheta_h$ satisfies the same
intertwining equation and has the same constant term; uniqueness of the
normalized solution therefore gives
$(\vartheta_h\otimes\vartheta_h)(\Psi)=\Psi$, and hence
$\vartheta_h$ fixes $\mathcal R$.

A Hopf automorphism commutes with the antipode, so it fixes
$u=m(S\otimes\id)(\mathcal R_{21})$.  Since it also fixes
$g=K_{-2\rho}$, the identity $g=u\mathbf r^{-1}$ implies that it fixes
the ribbon element $\mathbf r$.
\end{proof}

\begin{proposition}[Invariance of the universal knot element]
\label{prop:knot-element-automorphism-invariance}
For every zero-framed knot $K$, every integer $n\ge2$, and every $c\in\Z$,
\begin{equation}\label{eq:J-vartheta-delta-invariance}
 \vartheta\bigl(J_K(\mathfrak{gl}_n;q^2)\bigr)
 =J_K(\mathfrak{gl}_n;q^2),
 \qquad
 \delta_c\bigl(J_K(\mathfrak{gl}_n;q^2)\bigr)
 =J_K(\mathfrak{gl}_n;q^2).
\end{equation}
\end{proposition}

\begin{proof}
Write $J=J_K(\mathfrak{gl}_n;q^2)$.  By
\cref{lem:vartheta-ribbon}, $\vartheta_h$ fixes the universal
$R$-matrix and the pivotal and ribbon data and commutes with all Hopf
structure maps.  In the explicit braided functor of
\cite[Section~8.2 and Proposition~8.1]{Habiro06}, every local morphism
used to evaluate a bottom tangle is built from precisely these data.
Applying $\vartheta_h$ to all labels therefore leaves the state sum
unchanged and gives
\[
 \vartheta_h\bigl(\iota_h^{\cZ}(J)\bigr)
 =\iota_h^{\cZ}(J).
\]
By \eqref{eq:intrinsic-hadic-equivariance},
$\iota_h^{\cZ}(\vartheta(J)-J)=0$.  Injectivity in
\eqref{eq:central-hadic-injection} gives
$\vartheta(J)=J$ in the interpolation sector.

In $U_h(\mathfrak{gl}_n)$, $\delta_c$ extends continuously by
$H_j\mapsto H_j+c$.  It fixes $E_i$, $F_i$, and every difference
$H_i-H_{i+1}$, hence fixes the embedded
$U_h(\mathfrak{sl}_n)$ subalgebra pointwise.  Under the standard embedding,
\eqref{eq:BG-gl-sl-exact} and
\eqref{eq:knot-hadic-identification} identify
$\iota_h^{\cZ}(J)$ with
$J_K(\mathfrak{sl}_n;q^2)$.  By
\cref{prop:algebraic-automorphisms}, $\delta_{c,h}$ is the
identity on this universal knot element.  Therefore
\[
 \iota_h^{\cZ}\bigl(\delta_c(J)-J\bigr)
 =\delta_{c,h}\bigl(\iota_h^{\cZ}(J)\bigr)
  -\iota_h^{\cZ}(J)=0
\]
by \eqref{eq:intrinsic-hadic-equivariance}.  Equation
\eqref{eq:central-hadic-injection} gives $\delta_c(J)=J$, proving
both equalities in \eqref{eq:J-vartheta-delta-invariance}.
\end{proof}

\subsection{Completed Harish--Chandra equivariance}

Recall the completed Harish--Chandra isomorphism
\[
 \widehat\hc:\widehat\cZ_n^{\mathrm{BG}}
 \xrightarrow{\ \sim\ }\widehat\Lambda_n,
\]
where $\widehat\Lambda_n$ is the completion of
$\Lambda_n^{\mathrm{pol}}=R_{\mathrm{ev}}[y_1,\ldots,y_n]^{\mathfrak S_n}$
for the interpolation filtration $\mathcal S_m$, and
$y_i=q^{n-1}K_i^2$.  The symmetric Laurent subring is dense in
$\widehat\Lambda_n$, and the substitutions below extend continuously by
\cref{cor:interpolation-completion-structure}.

Define automorphisms of the algebraic torus in the shifted coordinates by
\begin{align}
 \mathsf D_c(y_1,\ldots,y_n)
 &=(q^{2c}y_1,\ldots,q^{2c}y_n),
 \label{eq:Dc-HC}\\
 \mathsf I(y_1,\ldots,y_n)
 &=(y_n^{-1},\ldots,y_1^{-1}).
 \label{eq:I-HC}
\end{align}
We use the same symbols for the induced continuous pullbacks on
$\widehat\Lambda_n$.

\begin{lemma}
\label{lem:HC-equivariance}
For every $a\in\Z$ and $\xi\in\widehat\cZ_n^{\mathrm{BG}}$,
\begin{align}
 \widehat\hc(\delta_a\xi)
 &=\mathsf D_a^*\widehat\hc(\xi),
 \label{eq:HC-delta-equivariance}\\
 \widehat\hc(\vartheta\xi)
 &=(\mathsf D_{n-1}\mathsf I)^*\widehat\hc(\xi).
 \label{eq:HC-vartheta-equivariance}
\end{align}
\end{lemma}

\begin{proof}
On the algebraic quantum group, both automorphisms preserve the triangular
decomposition and its weight-zero Cartan summand, so Harish--Chandra
projection is equivariant.  Their action on the shifted coordinates is
\[
 \delta_a(y_i)=q^{2a}y_i,
 \qquad
 \vartheta(y_i)=q^{2(n-1)}y_{n+1-i}^{-1}.
\]
Thus the asserted identities hold on the dense symmetric Laurent subring.
Their continuous extensions are precisely the automorphisms transported
through $\widehat\hc$ above, and density proves the identities on the
completed sector.
\end{proof}

\begin{proposition}[Completed Harish--Chandra reflection]
\label{prop:completed-HC-reflection}
For every zero-framed knot $K$ and every integer $n\ge2$, let
\begin{equation}\label{eq:fK-HC}
 f_K=\widehat\hc
 \bigl(J_K(\mathfrak{gl}_n;q^2)\bigr).
\end{equation}
Then
\begin{equation}\label{eq:fK-scaling-reflection}
 \mathsf D_c^*f_K=f_K\quad(c\in\Z),
 \qquad
 \mathsf I^*f_K=f_K.
\end{equation}
\end{proposition}

\begin{proof}
The scaling equality follows from the $\delta_c$ equality in
\cref{prop:knot-element-automorphism-invariance} and
\eqref{eq:HC-delta-equivariance}.  The $\vartheta$ equality and
\eqref{eq:HC-vartheta-equivariance} give
\begin{equation}\label{eq:fK-vartheta-intermediate}
 (\mathsf D_{n-1}\mathsf I)^*f_K=f_K.
\end{equation}
As point transformations,
\[
 \mathsf I
 =\mathsf D_{-(n-1)}
  \circ(\mathsf D_{n-1}\circ\mathsf I),
\]
so pullback reverses the order and gives
\begin{equation}\label{eq:I-star-factorization}
 \mathsf I^*
 =(\mathsf D_{n-1}\mathsf I)^*
  \circ\mathsf D_{-(n-1)}^*.
\end{equation}
Both factors on the right are continuous automorphisms of the completed
Harish--Chandra target by \cref{lem:HC-equivariance}.
Equations \eqref{eq:fK-scaling-reflection} and
\eqref{eq:fK-vartheta-intermediate}, applied successively in
\eqref{eq:I-star-factorization}, give $\mathsf I^*f_K=f_K$.
\end{proof}

\subsection{Reflection on the one-row curve}

A direct calculation from \eqref{eq:gamma-affine} gives
\begin{align}
 \mathsf I\gamma_n(z)
 &=(1,q^{-2},q^{-4},\ldots,q^{-2n+4},q^{-n+2}z^{-1}),
 \label{eq:I-gamma}\\
 \mathsf D_{n-2}\mathsf I\gamma_n(z)
 &=(q^{2n-4},q^{2n-6},\ldots,1,q^{n-2}z^{-1}).
 \label{eq:D-I-gamma}
\end{align}
Let $\mathsf P$ be the cyclic coordinate permutation that moves the last
coordinate to the first, and put
\[
 \mathsf T=\mathsf P\circ\mathsf D_{n-2}\circ\mathsf I.
\]
Equations \eqref{eq:I-gamma}--\eqref{eq:D-I-gamma} give the exact
identity of Laurent-polynomial curves
\begin{equation}\label{eq:T-on-gamma}
 \mathsf T\gamma_n(z)=\gamma_n(z^{-1}).
\end{equation}
Since the Harish--Chandra target consists of symmetric functions,
$\mathsf P^*$ acts trivially on it.  Therefore
\cref{prop:completed-HC-reflection} implies
\begin{equation}\label{eq:T-fK-invariance}
 \mathsf T^*f_K=f_K.
\end{equation}
For $s\in\Z$, recall $z_s=q^{2s+n}$.  Then
\begin{equation}\label{eq:inverse-node-index}
 z_s^{-1}=z_{-s-n}.
\end{equation}
Inversion gives a continuous disk isomorphism because, with
$z=z_s+u_s$,
\begin{equation}\label{eq:formal-inversion-disk}
z^{-1}-z_s^{-1}
 =-z_s^{-2}u_s+z_s^{-3}u_s^2-z_s^{-4}u_s^3+\cdots.
\end{equation}
The series converges in the $(\pi_d,u_s)$-adic topology, and its linear
coefficient $-z_s^{-2}$ is a unit.  Since inversion is its own inverse
after interchanging the two centers, it defines a continuous disk
isomorphism
\[
 \operatorname{inv}_s^*:
 \mathscr A_{d,-s-n}\xrightarrow{\ \sim\ }\mathscr A_{d,s}.
\]
Via the completed Harish--Chandra isomorphism, write
\[
 \operatorname{res}^{\mathrm{HC}}_{d,s}(f)
 =\operatorname{res}_{d,s}\bigl(\widehat\hc^{-1}(f)\bigr).
\]
On the dense Laurent subring, \eqref{eq:T-on-gamma} gives
\begin{equation}\label{eq:restriction-intertwining}
 \operatorname{res}^{\mathrm{HC}}_{d,s}(\mathsf T^*f)
 =
 \operatorname{inv}_s^*
 \operatorname{res}^{\mathrm{HC}}_{d,-s-n}(f).
\end{equation}
Continuity and separatedness extend this identity uniquely to
$\widehat\Lambda_n$.
Applying it to $f=f_K$ and using
\eqref{eq:T-fK-invariance} gives
\begin{equation}\label{eq:curve-reflection}
 \operatorname{res}_{d,s}(J_K)
 =
 \operatorname{inv}_s^*
 \operatorname{res}_{d,-s-n}(J_K).
\end{equation}

\begin{corollary}
\label{cor:local-germ-reflection}
For every zero-framed knot $K$, every $n\ge2$, every $d\ge1$, and every
$s\in\Z$,
\begin{equation}\label{eq:local-germ-reflection}
 \mathscr J_{K,s}(z)
 =\mathscr J_{K,-s-n}(z^{-1})
\end{equation}
as an equality in the disk at $z_s$, where the right side is pulled back
by the formal inversion series \eqref{eq:formal-inversion-disk}.
\end{corollary}

\begin{proof}
This follows immediately from \eqref{eq:curve-reflection},
\eqref{eq:restriction-of-knot}, and the definition of
$\operatorname{inv}_s^*$.
\end{proof}

\subsection{Integral descent through \texorpdfstring{$X=z+z^{-1}$}{X=z+1/z}}

The unramified case uses the following one-variable formal inverse theorem.

\begin{lemma}
\label{lem:formal-inverse-function}
Let $B$ be a commutative ring and
$G(T)=b_1T+b_2T^2+\cdots\in TB[[T]]$ with $b_1\in B^\times$.  There is a
unique $H(S)\in SB[[S]]$ such that
$G(H(S))=S$ and $H(G(T))=T$.  Consequently, substitution
$S\mapsto G(T)$ is an isomorphism between the $S$-adically complete
$B$-algebra $B[[S]]$ and the $T$-adically complete $B$-algebra
$B[[T]]$.
\end{lemma}

\begin{proof}
Write $H(S)=\sum_{m\ge1}h_mS^m$.  Comparing the coefficient of $S$ in
$G(H(S))=S$ gives $h_1=b_1^{-1}$.  For $m\ge2$, the coefficient of $S^m$
is $b_1h_m$ plus an expression involving only
$h_1,\ldots,h_{m-1}$; this determines $h_m$ uniquely.  The resulting
series satisfies $G(H(S))=S$.  Applying the same construction to obtain a
left inverse and using uniqueness shows $H(G(T))=T$.  All compositions
are well defined because every substituted series has zero constant
term; the coefficient of any fixed degree therefore depends on only
finitely many input coefficients.
\end{proof}

\begin{lemma}[Reflection-invariant formal descent]
\label{lem:formal-reflection-descent}
Let $\cO$ be a complete DVR with $2\in\cO^\times$, let
$a\in\cO^\times$, and set $x_a=a+a^{-1}$.  Let $F_a$ and
$F_{a^{-1}}$ be integral germs on the residue disks of $a$ and $a^{-1}$,
compatible under $z\mapsto z^{-1}$.  If the two disks coincide, assume
that the germs also agree after recentering.

Then the two germs descend uniquely through $X=z+z^{-1}$.  If
$\bar a^2\ne1$, the descended germ lies in $\cO[[X-x_a]]$.  If
$\bar a=\varepsilon\in\{1,-1\}$, it lies in
$\cO[[X-2\varepsilon]]$.
\end{lemma}

\begin{proof}
Suppose first that $\bar a^2\ne1$.  Writing $z=a+u$, one has
\[
 X-x_a=(1-a^{-2})u+O(u^2),
\]
and the linear coefficient is a unit.  Thus $X-x_a$ is a formal coordinate
on the disk of $a$.  Transporting $F_a$ gives the required quotient germ,
and inversion compatibility identifies its pullback on the disk of
$a^{-1}$ with $F_{a^{-1}}$.

Now let $\bar a=\varepsilon\in\{1,-1\}$.  After recentering at
$\varepsilon$, put
\[
 t=\frac{z-\varepsilon}{z+\varepsilon}.
\]
Inversion sends $t$ to $-t$, so an invariant germ belongs to
$\cO[[t^2]]$.  Moreover,
\[
 X-2\varepsilon=\frac{4\varepsilon t^2}{1-t^2},
\]
whose linear coefficient as a series in $t^2$ is a unit.  Hence
$\cO[[t^2]]=\cO[[X-2\varepsilon]]$.  Uniqueness follows from the
injectivity of pullback.
\end{proof}

\section{The global cyclotomic expansion}\label{sec:main-proof}

Fix a zero-framed knot $K$ and an integer $n\ge2$ throughout this section.
The local formal information from
\cref{sec:BG-restriction,sec:reflection} is combined with the algebraic
interpolation criterion of \cref{sec:newton-local}: first locally at each
cyclotomic prime, and then globally by denominator removal.

\subsection{Fibers of the quotient coordinate}

For $i\ge0$, retain $z_i$ from \eqref{eq:z-s} and set
\begin{equation}\label{eq:proof-zxy}
 x_i=X_i^{(n)}=z_i+z_i^{-1},
 \qquad
 y_i=J_i^{SU(n)}(K;q).
\end{equation}
All $z_i$ are units in every $\widehat\cO_d$.

\begin{lemma}
\label{lem:residue-fibers}
For every $d\ge1$ and all $i,j$,
\begin{equation}\label{eq:X-z-factorization}
 x_i-x_j=(z_i-z_j)\bigl(1-(z_iz_j)^{-1}\bigr).
\end{equation}
Consequently, in the residue field of $\widehat\cO_d$,
\begin{equation}\label{eq:x-cluster-z-description}
 \bar x_i=\bar x_j
 \quad\Longleftrightarrow\quad
 \bar z_i=\bar z_j
 \quad\text{or}\quad
 \bar z_i=\bar z_j^{-1}.
\end{equation}
Thus the full inverse image of an $x$-residue class consists of either
one inversion-stable $z$-disk or two $z$-disks exchanged by inversion; a
finite node cluster may meet one or both disks.
\end{lemma}

\begin{proof}
Expanding the right side of \eqref{eq:X-z-factorization} gives
\[
 z_i-z_j-z_j^{-1}+z_i^{-1}=x_i-x_j.
\]
In the residue field, multiply the equality $\bar x_i=\bar x_j$ by the
nonzero element $\bar z_i\bar z_j$.  One obtains
\[
 (\bar z_i-\bar z_j)(\bar z_i\bar z_j-1)=0.
\]
Since the residue field is an integral domain, one of the two factors
vanishes.  The
converse follows directly from \eqref{eq:X-z-factorization}.
\end{proof}

\subsection{Integral interpolation at one cyclotomic prime}

\begin{proposition}[Local interpolation polynomial]
\label{prop:local-interpolation-polynomial}
Fix $d\ge1$ and $k\ge0$.  The unique polynomial
$P_{k,d}(T)\in\operatorname{Frac}(\widehat\cO_d)[T]$ of degree at most
$k$ satisfying
\begin{equation}\label{eq:local-interpolation-values}
 P_{k,d}(x_i)=y_i\qquad(0\le i\le k)
\end{equation}
belongs to $\widehat\cO_d[T]$.
\end{proposition}

\begin{proof}
The nodes are pairwise distinct in $\Bbbk$ by
\cref{lem:node-factorization}, and $\Bbbk$ embeds in
$\operatorname{Frac}(\widehat\cO_d)$ by
\cref{lem:cyclotomic-DVR}; hence the interpolation polynomial exists and
is unique.

Partition the nodes into residue clusters for the $x$-coordinate.  For a
cluster $C$, choose $i_0\in C$ and set $a=z_{i_0}$.  By
\cref{lem:residue-fibers}, every $z$-node above $C$ lies in the residue
disk of $a$ or of $a^{-1}$.  The local knot germs on these disks are
compatible under inversion by \cref{cor:local-germ-reflection}.  Hence
\cref{lem:formal-reflection-descent} gives an integral germ in the quotient
coordinate $X$.

In the unramified case, write this germ as
$G_C(W)\in\widehat\cO_d[[W]]$ with $W=X-x_a$.  In the ramified case,
let $\bar a=\varepsilon\in\{1,-1\}$ and write the descended germ as
$\widetilde G(S)=\sum_{m\ge0}g_mS^m$ with $S=X-2\varepsilon$.  Recenter it
at $x_a$ by setting
\[
 W=X-x_a,\qquad c=x_a-2\varepsilon\in\pi_d\widehat\cO_d,
\]
and
\begin{equation}\label{eq:ramified-recentering}
 G_C(W)=\widetilde G(W+c)
 =\sum_{j\ge0}\left(
   \sum_{m\ge j}g_m\binom{m}{j}c^{m-j}
 \right)W^j.
\end{equation}
The inner sums converge $\pi_d$-adically, so
$G_C(W)\in\widehat\cO_d[[W]]$ in both cases.  Its pullback agrees with the
knot germ on every disk above $C$, and therefore
\begin{equation}\label{eq:cluster-X-values}
 y_i=G_C(x_i-x_a)\qquad(i\in C),
 \qquad x_i-x_a\in\pi_d\widehat\cO_d.
\end{equation}
The hypotheses of \cref{lem:cluster-interpolation} are satisfied for every
cluster, so the interpolation polynomial belongs to
$\widehat\cO_d[T]$.
\end{proof}

\subsection{Global denominator removal}

\begin{theorem}[Fixed-rank Laurent integrality]
\label{thm:main}
For every zero-framed knot $K$, every fixed $n\ge2$, and every $k\ge0$,
the formal Newton coefficient in \eqref{eq:formal-SU} satisfies
\begin{equation}\label{eq:main-integrality}
 H_k^{(n)}(K;q)\in\Z[q^{\pm1}].
\end{equation}
Equivalently, for every $r\ge0$,
\begin{equation}\label{eq:main-expansion}
 J_r^{SU(n)}(K;q)=\sum_{k=0}^{r}
 \left(\prod_{i=0}^{k-1}\{r-i\}\{r+n+i\}\right)
 H_k^{(n)}(K;q),
\end{equation}
with unique Laurent-polynomial coefficients.
\end{theorem}

\begin{proof}
Fix $k$.  By \cref{prop:one-row-restriction}, all values
$J_i^{SU(n)}(K;q)$ lie in $R$.  Newton interpolation over
$\Bbbk=\Q(q)$
therefore gives the rational function
\begin{equation}\label{eq:H-determinant-global}
 H_k^{(n)}(K;q)
 =\frac{\Delta_k^{(n)}(J(K))}{V_k^{(n)}},
 \qquad
 V_k^{(n)}=\prod_{0\le i<j\le k}
 \{j-i\}\{n+i+j\}.
\end{equation}
Here $\Delta_k^{(n)}(J(K))\in R$, because every matrix entry in its
defining determinant belongs to $R$, and $V_k^{(n)}\ne0$ by the
distinctness of the nodes.  Thus, up to a unit of $R$, the reduced denominator of
$H_k^{(n)}$ divides the displayed Vandermonde product.

Let $P_k^{\mathrm{glob}}(T)\in\Bbbk[T]$ be the unique degree-$\le k$
polynomial interpolating $(x_i,y_i)$ for $0\le i\le k$.  Fix a
cyclotomic polynomial $\Phi_d(q)$.  Under the embedding
$\Bbbk\hookrightarrow\operatorname{Frac}(\widehat\cO_d)$ from
\cref{lem:cyclotomic-DVR}, the polynomial
$P_k^{\mathrm{glob}}$ satisfies the same interpolation conditions as
$P_{k,d}$.  Uniqueness over the fraction field therefore gives
$P_k^{\mathrm{glob}}=P_{k,d}$.  By
\cref{prop:local-interpolation-polynomial}, this polynomial belongs to
$\widehat\cO_d[T]$.  Its leading coefficient is the $k$th divided
difference, hence is exactly the fraction in
\eqref{eq:H-determinant-global}.  Therefore
\begin{equation}\label{eq:all-cyclotomic-valuations}
 \val_{\Phi_d}\bigl(H_k^{(n)}(K;q)\bigr)\ge0
 \qquad(d\ge1).
\end{equation}
By \cref{lem:cyclotomic-DVR}, this is the ordinary
$\Phi_d$-adic valuation on $\Bbbk$.

It remains to show that no noncyclotomic denominator can occur.  Define
the multiset
\[
 \mathcal M_k=
 \{\,j-i,\ n+i+j:0\le i<j\le k\,\},
 \qquad S_k=\sum_{m\in\mathcal M_k}m,
\]
where entries are repeated with their pair multiplicities.  Every entry
is positive and, by \eqref{eq:node-factorization},
\begin{equation}\label{eq:primitive-Vandermonde-polynomial}
 q^{S_k}V_k^{(n)}
 =\prod_{m\in\mathcal M_k}(q^{2m}-1)\in\Z[q].
\end{equation}
Each factor on the right is monic and primitive, so Gauss's lemma shows
that their product is primitive.
Moreover, \eqref{eq:brace-cyclotomic-factorization} shows that every
nonunit irreducible factor of $V_k^{(n)}$ is cyclotomic.  Apply
\cref{lem:UFD-denominator} with
$N=\Delta_k^{(n)}(J(K))$ and $V=V_k^{(n)}$.  The local inequalities
\eqref{eq:all-cyclotomic-valuations} say, prime by prime and with the full
multiplicity, that the numerator has at least the valuation demanded by
the denominator.  They therefore eliminate every possible denominator
and yield \eqref{eq:main-integrality}.

The expansion and uniqueness follow from
\cref{prop:newton-inversion} and \eqref{eq:SU-newton-kernel}.  The color
$r=0$ is the trivial representation, so
$H_0^{(n)}(K;q)=J_0^{SU(n)}(K;q)=1$.
\end{proof}

By \cref{prop:CLZ-normalization}, the invariant and the kernel in
\cref{thm:main} are those of
\cite[Conjecture~1.3]{CLZ15}.  Thus the theorem proves that conjecture for
every zero-framed knot and every fixed $n\ge2$.

\begin{corollary}
\label{cor:determinant-divisibility}
For every zero-framed knot $K$, every $n\ge2$, and every $k\ge0$,
\begin{equation}\label{eq:main-determinant-divisibility}
 \Delta_k^{(n)}(J(K))\in
 V_k^{(n)}\Z[q^{\pm1}].
\end{equation}
\end{corollary}

\begin{proof}
Equation \eqref{eq:H-determinant-global} gives
$\Delta_k^{(n)}=V_k^{(n)}H_k^{(n)}$; use
\cref{thm:main}.
\end{proof}

\begin{corollary}[Full Newton congruence filtration]
\label{cor:higher-congruences}
For every zero-framed knot $K$, every $n\ge2$, and $r>s\ge0$,
\begin{equation}\label{eq:pairwise-SU}
 J_r^{SU(n)}(K;q)-J_s^{SU(n)}(K;q)
 \in\{r-s\}\{r+s+n\}\Z[q^{\pm1}].
\end{equation}
More generally, for $0\le m<r$,
\begin{equation}\label{eq:higher-remainder-SU}
 J_r^{SU(n)}(K;q)\equiv
 \sum_{k=0}^{m}C_{r+1,k}^{(n)}H_k^{(n)}(K;q)
 \pmod{C_{r+1,m+1}^{(n)}\Z[q^{\pm1}]}.
\end{equation}
\end{corollary}

\begin{proof}
Apply \cref{cor:newton-filtration} to the integral Newton expansion and
use
$X_r^{(n)}-X_s^{(n)}=\{r-s\}\{r+s+n\}$.
\end{proof}

\subsection{Cyclotomic valuations of the kernel and Vandermonde factor}

Fix $d\ge1$ and put $\ell_d=d^\sharp$.  For $k\ge0$ and
$a\in\mathbb Z/\ell_d\mathbb Z$, set
\begin{equation}\label{eq:residue-counts}
 c_a=c_a(k,\ell_d)=\#\{0\le i\le k:i\equiv a\pmod{\ell_d}\}.
\end{equation}
For representatives $0\le a<\ell_d$,
\[
 c_a=
 \begin{cases}
  1+\left\lfloor\dfrac{k-a}{\ell_d}\right\rfloor,&a\le k,\\[6pt]
  0,&a>k.
 \end{cases}
\]

\begin{proposition}[Cyclotomic valuations]
\label{prop:exact-cyclotomic-multiplicities}
Let $M=\lfloor k/\ell_d\rfloor$.  The diagonal Newton factor
\begin{equation}\label{eq:diagonal-kernel}
 C_{k+1,k}^{(n)}
 =\prod_{i=0}^{k-1}\{k-i\}\{k+n+i\}
 =\{k\}!\prod_{m=n+k}^{n+2k-1}\{m\}
\end{equation}
satisfies
\begin{equation}\label{eq:diagonal-valuation}
 \val_{\Phi_d}\!\left(C_{k+1,k}^{(n)}\right)
 =\left\lfloor\frac{k}{\ell_d}\right\rfloor
 +\left\lfloor\frac{n+2k-1}{\ell_d}\right\rfloor
 -\left\lfloor\frac{n+k-1}{\ell_d}\right\rfloor.
\end{equation}
The full Vandermonde factor has valuation
\begin{align}
 \val_{\Phi_d}(V_k^{(n)})
 &=M(k+1)-\ell_d\frac{M(M+1)}2
 \notag\\
 &\quad+
 \frac12\left(
  \sum_{a\in\Z/\ell_d\Z}c_ac_{-n-a}
  -\sum_{\substack{a\in\Z/\ell_d\Z\\
                       2a\equiv-n\pmod{\ell_d}}}c_a
 \right).
 \label{eq:full-Vandermonde-valuation}
\end{align}
\end{proposition}

\begin{proof}
The factor $\Phi_d(q)$ occurs in $\{m\}$ exactly when $\ell_d\mid m$, and
then with multiplicity one.  The first term on the right of
\eqref{eq:diagonal-valuation} counts the multiples of $\ell_d$ in $[1,k]$;
the difference of floors counts those in $[n+k,n+2k-1]$.

For the first line of \eqref{eq:full-Vandermonde-valuation}, write
$j-i=t\ell_d$.  For
$1\le t\le M=\lfloor k/\ell_d\rfloor$, there are $k+1-t\ell_d$ pairs with this
difference.  Hence their number is
\[
 \sum_{t=1}^{M}(k+1-t\ell_d)
 =M(k+1)-\ell_d\frac{M(M+1)}2.
\]
For the second line,
$\sum_ac_ac_{-n-a}$ counts ordered pairs $(i,j)$ satisfying
$\ell_d\mid n+i+j$.  Its diagonal terms are exactly the $i=j$ solutions,
namely the indices in residue classes satisfying
$2a\equiv-n\pmod{\ell_d}$.  Removing these diagonal solutions leaves twice the
number of unordered pairs, so division by two counts the pairs with
$i<j$.  Together with the preceding count for the factors $\{j-i\}$, this
accounts for all factors in $V_k^{(n)}$.
\end{proof}

Together with \eqref{eq:main-determinant-divisibility},
\eqref{eq:full-Vandermonde-valuation} gives
\[
 \val_{\Phi_d}\!\left(\Delta_k^{(n)}(J(K))\right)
 \ge \val_{\Phi_d}\!\left(V_k^{(n)}\right).
\]
When $H_k^{(n)}(K;q)\ne0$, the difference of these two valuations is
$\val_{\Phi_d}(H_k^{(n)}(K;q))$.

\begin{remark}[The case $n=2$]\label{rem:n2}
Let $n=2$ and put $N=r+1$, the representation dimension in the usual
colored-Jones convention.  Then
\begin{align}
 C_{r+1,k}^{(2)}
 &=\left(\prod_{j=N-k}^{N-1}\{j\}\right)
   \left(\prod_{j=N+1}^{N+k}\{j\}\right)
 =\frac{\prod_{j=N-k}^{N+k}\{j\}}{\{N\}}.
 \label{eq:n2-Habiro-kernel}
\end{align}
Thus the color shift and the omitted central factor are explicit.  In the
reduced zero-framed convention fixed above, the right side of
\eqref{eq:n2-Habiro-kernel} is Habiro's central-factor-omitted
cyclotomic kernel \cite[Theorem~3.1]{Habiro02}.  Hence uniqueness identifies the present $H_k^{(2)}$ with the
coefficients in that convention.
\end{remark}

\section{Consequences of the fixed-rank expansion}
\label{sec:fixed-rank-consequences}

The fixed-rank expansion has two complementary consequences.  Its first
nontrivial Newton coefficient is determined by the fundamental HOMFLY--PT
polynomial, while the full coefficient sequence admits both a relative
colored lift and diagonal lifts to the ordinary Habiro ring.

\subsection{The first cyclotomic coefficient}

For the fundamental color, write
\[
 P_K(A,q)=\HH_1(K;A,q)\in\Z[A^{\pm1},q^{\pm1}].
\]
The positive-rank specialization is
\[
 \HH_1(K;q^n,q)=J_1^{SU(n)}(K;q)\qquad(n\ge2).
\]

Recall the differential-factor notation
\begin{equation}\label{eq:HOMFLY-differential-factor}
 \{m;A\}=Aq^m-A^{-1}q^{-m},
 \qquad
 \{m;A\}_k=\prod_{\nu=0}^{k-1}\{m-\nu;A\},
\end{equation}
with the empty product equal to $1$.

\begin{lemma}[Fundamental differential factorization]
\label{lem:fundamental-factorization}
For every zero-framed knot $K$, there exists a unique Laurent polynomial
$F_K(A,q)\in\Z[A^{\pm1},q^{\pm1}]$ such that
\begin{equation}\label{eq:fundamental-factorization}
 \HH_1(K;A,q)-1=\{1;A\}\{-1;A\}F_K(A,q).
\end{equation}
\end{lemma}

\begin{proof}
We first regard the specialized skein invariants as taking values in
$\Q(q)$.  Let $c(L)$ denote the number of components of an oriented link.
At $A=q$, the constant function $L\mapsto1$ satisfies the specialized
skein relation and the unknot normalization.  At $A=q^{-1}$, the function
$L\mapsto(-1)^{c(L)-1}$ satisfies the same conditions, because oriented
smoothing changes the number of components by one.  Hence
\[
 P_K(q,q)=P_K(q^{-1},q)=1
\]
for every knot $K$.

For an oriented link $L$, put
$Q_L(A,q)=(-1)^{c(L)-1}P_L(-A,q)$.  The change $A\mapsto-A$ changes the
sign of the skein relation, while oriented smoothing changes the parity of
$c(L)-1$.  Thus $Q_L$ satisfies the original skein relation and $Q_U=1$.
Uniqueness gives
\begin{equation}\label{eq:HOMFLY-A-parity}
 P_L(-A,q)=(-1)^{c(L)-1}P_L(A,q).
\end{equation}
In particular, $P_K$ is even in $A$, and
\begin{equation}\label{eq:four-fundamental-specializations}
 \HH_1(K;\pm q,q)=\HH_1(K;\pm q^{-1},q)=1.
\end{equation}

Choose $M\ge0$ such that
\[
 g(A)=A^M\bigl(\HH_1(K;A,q)-1\bigr)\in R[A].
\]
Division by the monic polynomial
$(A^2-q^2)(A^2-q^{-2})$ in $R[A]$ gives
\[
 g(A)=(A^2-q^2)(A^2-q^{-2})U(A)+R_0(A),
\]
where $U,R_0\in R[A]$ and $\deg_A R_0<4$.
The remainder vanishes at the four distinct points
$q,-q,q^{-1},-q^{-1}$ in $\Q(q)$, hence $R_0=0$.  Finally,
\[
 \{1;A\}\{-1;A\}=A^{-2}(A^2-q^2)(A^2-q^{-2}).
\]
It follows that \eqref{eq:fundamental-factorization} holds with
$F_K=A^{2-M}U(A)\in\Z[A^{\pm1},q^{\pm1}]$.  Uniqueness follows because the
Laurent polynomial ring is a domain.
\end{proof}

\begin{proposition}
\label{prop:first-coefficient}
For every zero-framed knot $K$ and every $n\ge2$,
\begin{equation}\label{eq:first-coefficient}
 H_1^{(n)}(K;q)=[n-1]_qF_K(q^n,q).
\end{equation}
\end{proposition}

\begin{proof}
Specializing \eqref{eq:fundamental-factorization} at $A=q^n$ gives
\[
 J_1^{SU(n)}(K;q)-1
 =\{n+1\}\{n-1\}F_K(q^n,q)
 =\{1\}\{n+1\}[n-1]_qF_K(q^n,q).
\]
On the other hand, the $r=1$ row of the Newton expansion is
\[
 J_1^{SU(n)}(K;q)
 =H_0^{(n)}(K;q)+C_{2,1}^{(n)}H_1^{(n)}(K;q),
\]
where $H_0^{(n)}(K;q)=1$ and
$C_{2,1}^{(n)}=\{1\}\{n+1\}$.  Cancelling the nonzero factor
$\{1\}\{n+1\}$ proves \eqref{eq:first-coefficient}.
\end{proof}

Consequently,
\[
 H_1^{(n)}(K;q)\in[n-1]_q\Z[q^{\pm1}],
 \qquad
 J_1^{SU(n)}(K;q)-1
 \in\{n-1\}\{n+1\}\Z[q^{\pm1}].
\]

\subsection{A relative colored Newton completion}

Let $\mathscr R_n=\Z[q^{\pm1},x^{\pm1}]$ and define
\begin{equation}\label{eq:relative-kernel}
 P_k^{(n)}(x,q)=\prod_{i=0}^{k-1}
 (xq^{-i}-x^{-1}q^i)(xq^{n+i}-x^{-1}q^{-n-i}).
\end{equation}
Since $P_k^{(n)}$ divides $P_{k+1}^{(n)}$, the ideals
$(P_k^{(n)})$ form a decreasing filtration.  Put
\begin{equation}\label{eq:relative-completion}
 \widehat{\mathscr R}_n=
 \varprojlim_{k\ge1}\mathscr R_n/(P_k^{(n)}).
\end{equation}

\begin{proposition}[Colored Newton lift]
\label{prop:colored-Newton-lift}
For every zero-framed knot $K$ and every $n\ge2$, the series
\begin{equation}\label{eq:colored-lift}
 \mathscr J_K^{(n)}(x,q)=\sum_{k=0}^{\infty}
 P_k^{(n)}(x,q)H_k^{(n)}(K;q)
\end{equation}
defines an element of $\widehat{\mathscr R}_n$.  For every $r\ge0$,
\begin{equation}\label{eq:colored-specialization}
 \mathscr J_K^{(n)}(q^r,q)=J_r^{SU(n)}(K;q).
\end{equation}
\end{proposition}

\begin{proof}
By \cref{thm:main}, $H_k^{(n)}(K;q)\in\Z[q^{\pm1}]$, so every summand
belongs to $\mathscr R_n$.  For $k\ge m$, one has
$P_k^{(n)}\in(P_m^{(n)})$; hence modulo $(P_m^{(n)})$ the series is the
finite sum over $0\le k<m$.  These residue classes are compatible and
define an inverse-limit element.

At $x=q^r$,
\[
 P_k^{(n)}(q^r,q)
 =\prod_{i=0}^{k-1}\{r-i\}\{r+n+i\}
 =C_{r+1,k}^{(n)}.
\]
For $k>r$, the factor indexed by $i=r$ is $\{0\}=0$, so the evaluation is
finite.  The identity \eqref{eq:colored-specialization} now follows from
\cref{thm:main}.
\end{proof}

\subsection{Diagonal lifts to the ordinary Habiro ring}

Recall Habiro's cyclotomic completion
\begin{equation}\label{eq:Habiro-ring}
 \Hab=\varprojlim_{m\ge1}\Z[q]/((q;q)_m),
 \qquad (q;q)_m=\prod_{a=1}^{m}(1-q^a).
\end{equation}

\begin{lemma}
\label{lem:Habiro-elementary}
The element $q$ is a unit in $\Hab$, the natural map
$\Z[q^{\pm1}]\to\Hab$ is injective, and every root of unity $\zeta$
defines a ring homomorphism
\[
 \operatorname{ev}_\zeta:\Hab\longrightarrow\Z[\zeta].
\]
If $a_k\in(q;q)_k\Z[q^{\pm1}]$, then $\sum_{k\ge0}a_k$ converges in
$\Hab$ and, when $\zeta$ has order $L$,
\[
 \operatorname{ev}_\zeta\!\left(\sum_{k\ge0}a_k\right)
 =\sum_{k=0}^{L-1}a_k(\zeta).
\]
\end{lemma}

\begin{proof}
Put $I_m^{\mathrm H}=((q;q)_m)$.  Since $(q;q)_m$ has constant term $1$,
\[
 q\,\frac{1-(q;q)_m}{q}\equiv1\pmod{I_m^{\mathrm H}}.
\]
These inverses are compatible, so $q$ is a unit in $\Hab$ and the same
inverse limit may be taken over $\Z[q^{\pm1}]$.  If a Laurent polynomial
$f$ maps to zero, then after multiplying by a power of $q$ one obtains a
nonzero $g\in\Z[q]$ divisible by $(q;q)_m$ for every $m$, contradicting
$\deg(q;q)_m=m(m+1)/2>\deg g$ for large $m$.

If $\zeta$ has order $L$, the factor $1-q^L$ makes evaluation at $\zeta$
well defined on every quotient of level $m\ge L$, hence on $\Hab$.  The
condition $a_k\in I_k^{\mathrm H}\Z[q^{\pm1}]$ gives convergence, and all
terms with $k\ge L$ vanish after evaluation at $\zeta$.
\end{proof}

\begin{lemma}
\label{lem:consecutive-quantum}
For every $s\ge1$ and $k\ge0$,
\begin{equation}\label{eq:consecutive-divisibility}
 \prod_{i=0}^{k-1}\{s+i\}\in(q;q)_k\Z[q^{\pm1}].
\end{equation}
\end{lemma}

\begin{proof}
For every cyclotomic polynomial $\Phi_d(q)$,
$\val_{\Phi_d}((q;q)_k)=\lfloor k/d\rfloor$.  Up to a Laurent unit,
$\{s+i\}=1-q^{2(s+i)}$.  If $d$ is odd, the factor $\Phi_d$ occurs when
$d\mid s+i$, and among $k$ consecutive integers there are at least
$\lfloor k/d\rfloor$ such indices.  If $d$ is even, it occurs when
$d/2\mid s+i$, giving at least
$\lfloor2k/d\rfloor\ge\lfloor k/d\rfloor$ indices.  Comparing the
$\Phi_d$-adic valuations for all $d$ and applying unique factorization in
$\Z[q]$ proves \eqref{eq:consecutive-divisibility}.
\end{proof}

\begin{theorem}[Habiro diagonal lift]
\label{thm:Habiro-diagonal}
Let $K$ be a zero-framed knot, let $n\ge2$, and let
$1\le s\le n-1$.  Then
\begin{equation}\label{eq:ordinary-Habiro-series}
 \mathscr J_K^{(n,s)}(q)=\sum_{k=0}^{\infty}
 \left(\prod_{i=0}^{k-1}\{-s-i\}\{n-s+i\}\right)
 H_k^{(n)}(K;q)
\end{equation}
defines an element of $\Hab$.  If $\zeta$ has order $L>s$, then
\begin{equation}\label{eq:root-evaluation}
 \mathscr J_K^{(n,s)}(\zeta)=J_{L-s}^{SU(n)}(K;\zeta).
\end{equation}
Moreover, \eqref{eq:ordinary-Habiro-series} is the image of
\eqref{eq:colored-lift} under the continuous specialization $x=q^{-s}$.
\end{theorem}

\begin{proof}
Since $\{-s-i\}=-\{s+i\}$ and both $s$ and $n-s$ are positive, two
applications of \cref{lem:consecutive-quantum} give
\[
 \prod_{i=0}^{k-1}\{-s-i\}\{n-s+i\}
 \in(q;q)_k^2\Z[q^{\pm1}].
\]
By \cref{thm:main}, $H_k^{(n)}(K;q)\in\Z[q^{\pm1}]$; hence
\cref{lem:Habiro-elementary} proves convergence in $\Hab$.

Let $r=L-s$.  Since $\zeta^L=1$, for every $i$,
\[
 \left.\{r-i\}\right|_{q=\zeta}
 =\left.\{-s-i\}\right|_{q=\zeta},
 \qquad
 \left.\{r+n+i\}\right|_{q=\zeta}
 =\left.\{n-s+i\}\right|_{q=\zeta}.
\]
For $k>r$, the factor indexed by $i=r$ is $\{-L\}$ and vanishes at
$q=\zeta$.  Therefore the Habiro series truncates to
\[
 \sum_{k=0}^{r}C_{r+1,k}^{(n)}(\zeta)H_k^{(n)}(K;\zeta)
 =J_r^{SU(n)}(K;\zeta),
\]
which proves \eqref{eq:root-evaluation}.

Since
$P_k^{(n)}(q^{-s},q)\in(q;q)_k^2\Z[q^{\pm1}]$, evaluation at $x=q^{-s}$
is continuous from the relative Newton filtration to the Habiro
filtration.  By \cref{lem:filtered-extension}, it extends to
$\widehat{\mathscr R}_n\to\Hab$, where it sends
\eqref{eq:colored-lift} to \eqref{eq:ordinary-Habiro-series}.
\end{proof}

\begin{corollary}
\label{cor:Habiro-diagonal-unification}
Fix a zero-framed knot $K$, an integer $n\ge2$, and
$1\le s\le n-1$.  The values
\[
 J_{L-s}^{SU(n)}(K;\zeta),
 \qquad \operatorname{ord}(\zeta)=L>s,
\]
are uniquely determined by the Taylor series of
$\mathscr J_K^{(n,s)}$ at $q=1$.
\end{corollary}

\begin{proof}
Combine \cref{thm:Habiro-diagonal} with the injectivity of Habiro's Taylor
homomorphism $T_1:\Hab\to\Z[[q-1]]$
\cite[Theorem~5.4]{Habiro04}.
\end{proof}

\section{Rank-uniform integer-valued completion}
\label{sec:rank-uniform-completion}

We first construct Laurent differential coefficients
$G_k(K;A,q)$.  The fixed-rank theorem then shows that the associated
Newton coefficients are integral at every node $A=q^n$, yielding the
ring $\mathscr I_q^{\ge2}$ and the two-variable Newton completion.

\subsection{Structural symmetries and finite negative-rank reflection}

\begin{proposition}
\label{prop:colored-HOMFLY-integrality}
For every knot $K$ and partition $\lambda$,
\[
 \HH_\lambda(K;A,q)\in\Z[A^{\pm1},q^{\pm1}].
\]
\end{proposition}

\begin{proof}
This is Morton's integrality theorem for the $(1,1)$-tangle invariant
associated with an annular meridian eigenvector
\cite[Theorem~1]{Morton07}, applied to the eigenvector indexed by
$\lambda$.  Its quotient normalization agrees with the reduced
unknot normalization used here; see also \cite{Zhu23}.
\end{proof}

\begin{lemma}[Reduced sign and transpose symmetries]
\label{lem:rank-uniform-symmetries}
For every zero-framed knot $K$ and every partition $\lambda$,
\begin{align}
 \HH_\lambda(K;-A,q)&=\HH_\lambda(K;A,q),
 \label{eq:reduced-A-even}\\
 \HH_\lambda(K;A,q^{-1})
 &=\HH_{\lambda^t}(K;A,q).
 \label{eq:reduced-transpose}
\end{align}
\end{lemma}

\begin{proof}
The unreduced colored invariants satisfy
\cite[Theorems~3.3 and~3.6, equations~(15) and~(18)]{CLPZ23}
\[
 W_\lambda(K;q,-A)=(-1)^{|\lambda|}W_\lambda(K;q,A),
 \qquad
 W_\lambda(K;q^{-1},A)
 =(-1)^{|\lambda|}W_{\lambda^t}(K;q,A).
\]
The same identities hold for the colored unknot.  Since the reduced
invariant is the quotient by the unknot value, both signs cancel.
\end{proof}

\begin{lemma}[Finite negative-rank reflection]
\label{lem:finite-negative-rank-reflection}
For every zero-framed knot $K$, every $m\ge1$, and every $0\le r\le m$,
\begin{equation}\label{eq:finite-negative-rank-reflection}
 \HH_r(K;q^{-m},q)=\HH_{m-r}(K;q^{-m},q).
\end{equation}
Moreover, for every zero-framed knot $K$ and every $r\ge0$,
\begin{equation}\label{eq:rank-one-trivial}
 \HH_r(K;q,q)=1.
\end{equation}
\end{lemma}

Here $A=q^{-m}$ is an algebraic specialization of the reduced
HOMFLY--PT polynomial, not the definition of a negative-rank
quantum-group invariant.

\begin{proof}
If $m=1$, then $r\in\{0,1\}$.  The color-zero invariant is $1$, and
\cref{lem:fundamental-factorization} at $A=q^{-1}$ gives
$\HH_1(K;q^{-1},q)=1$.  Thus
\eqref{eq:finite-negative-rank-reflection} holds for $m=1$.  Assume
$m\ge2$ and set $\bar q=q^{-1}$.

By \eqref{eq:reduced-transpose},
\begin{equation}\label{eq:negative-rank-to-column}
 \HH_{[r]}(K;q^{-m},q)
 =\HH_{[1^r]}(K;\bar q^m,\bar q).
\end{equation}
By \cref{prop:all-partition-normalization}, applied with quantum parameter
$\bar q$, followed by \cref{lem:BG-gl-sl-interface}, the right-hand side
is the reduced $U_{\bar q^2}(\mathfrak{sl}_m)$ invariant colored by
$V_r=\bigwedge^rV$.

The endpoint representations $V_0$ and $V_m$ restrict to the trivial
$U_{\bar q^2}(\mathfrak{sl}_m)$ module, so the cases $r=0,m$ are equal.
Now assume $1\le r\le m-1$.  Let $\vartheta_{m,\bar q}$ be the diagram
automorphism of \cref{lem:vartheta-ribbon} with $n=m$ and quantum
parameter $\bar q$.  The pullback module ${}^{\vartheta}V_r$ has highest
weight $\omega_{m-r}$ and is therefore isomorphic to $V_{m-r}$.  Choose an
intertwining isomorphism
\[
 \phi:{}^{\vartheta}V_r\xrightarrow{\ \sim\ }V_{m-r}.
\]
Functoriality of the bottom-tangle construction in \textup{(DJ4)} and
\cref{lem:vartheta-ribbon} give
$\vartheta_{m,\bar q}(J_{K,h})=J_{K,h}$.  Writing $\rho_s$ for the action
on $V_s$ and using
$\rho_{{}^{\vartheta}V_r}(x)=\rho_r(\vartheta_{m,\bar q}(x))$, we obtain
\[
 \begin{aligned}
 \rho_r(J_{K,h})
 &=\rho_r\!\left(\vartheta_{m,\bar q}(J_{K,h})\right)\\
 &=\rho_{{}^{\vartheta}V_r}(J_{K,h})\\
 &=\phi^{-1}\rho_{m-r}(J_{K,h})\phi.
 \end{aligned}
\]
The universal knot element is central, so both endomorphisms are scalar.
Their scalar eigenvalues, and therefore their reduced quantum-trace
quotients, are equal.

Both reduced invariants belong to $\Z[\bar q^{\pm1}]$ by
\cref{prop:colored-HOMFLY-integrality}.  Applying the injective coefficient
embedding \eqref{eq:q-hadic-coefficient-embedding} with $h$ replaced by
$-h$ shows that equality in the $h$-adic realization is already equality
of generic Laurent polynomials.  Thus
\[
 \HH_{[1^r]}(K;\bar q^m,\bar q)
 =\HH_{[1^{m-r}]}(K;\bar q^m,\bar q).
\]
Applying \eqref{eq:reduced-transpose} again proves
\eqref{eq:finite-negative-rank-reflection}.

Finally, at $A=q$ the $r$th symmetric color is the one-dimensional
$U_{q^2}(\mathfrak{gl}_1)$ module of weight $r$.  The root part of the
universal $R$-matrix is trivial in rank one, while the abelian Cartan
contribution to a one-component framed link depends only on its framing.
For a zero-framed knot it is $1$, and division by the corresponding unknot
scalar gives \eqref{eq:rank-one-trivial}.
\end{proof}

\subsection{Laurent differential coefficients for arbitrary knots}

For $0\le k\le r$, define
\begin{equation}\label{eq:rank-uniform-Z}
 Z_{r,k}(A,q)=
 \qbinom{r}{k}\{r+k-1;A\}_k\{-1;A\}.
\end{equation}

\begin{theorem}[General Laurent differential expansion]
\label{thm:general-rank-uniform-DE}
For every zero-framed knot $K$, there is a unique sequence
\[
 G_k(K;A,q)\in\Z[A^{\pm1},q^{\pm1}],\qquad k\ge1,
\]
such that
\begin{equation}\label{eq:general-rank-uniform-DE}
 \HH_r(K;A,q)
 =1+\sum_{k=1}^{r}Z_{r,k}(A,q)G_k(K;A,q)
 \qquad(r\ge0).
\end{equation}
It satisfies
\begin{equation}\label{eq:general-G-parity}
 G_k(K;-A,q)=(-1)^{k+1}G_k(K;A,q).
\end{equation}
\end{theorem}

\begin{proof}
Put $S=\Z[A^{\pm1},q^{\pm1}]$.  By
\cref{prop:colored-HOMFLY-integrality}, every $\HH_r$ belongs to $S$.
We proceed by induction on the color.  Suppose that
$G_1,\ldots,G_{k-1}$ have been constructed and satisfy
\eqref{eq:general-G-parity}, and put
\begin{equation}\label{eq:rank-uniform-E}
 E_k(A,q)=\HH_k(K;A,q)-1-
 \sum_{j=1}^{k-1}Z_{k,j}(A,q)G_j(K;A,q)\in S.
\end{equation}
The diagonal kernel is
\begin{equation}\label{eq:rank-uniform-diagonal}
 Z_{k,k}(A,q)
 =\{-1;A\}\prod_{m=k}^{2k-1}\{m;A\}.
\end{equation}

Each $Z_{r,j}$ contains $j+1$ factors that are odd under
$A\mapsto-A$.  Consequently,
\[
 Z_{r,j}(-A,q)=(-1)^{j+1}Z_{r,j}(A,q).
\]
The induction hypothesis and \eqref{eq:reduced-A-even} imply that
\begin{equation}\label{eq:rank-uniform-E-even}
 E_k(-A,q)=E_k(A,q).
\end{equation}

Fix $m\in\{k,k+1,\ldots,2k-1\}$ and set $s=m-k$, so that
$0\le s<k$.  At $A=q^{-m}$, the finite reflection gives
$\HH_k=\HH_s$.  If $j\le s$, then
\begin{align}
 \frac{Z_{k,j}(q^{-m},q)}{\{-1;q^{-m}\}}
 &=\qbinom{k}{j}\prod_{i=0}^{j-1}\{i-s\}\notag\\
 &=(-1)^j
 \frac{\{k\}!\{s\}!}
 {\{j\}!\{k-j\}!\{s-j\}!}
 =\frac{Z_{s,j}(q^{-m},q)}{\{-1;q^{-m}\}}.
 \label{eq:rank-uniform-kernel-collision}
\end{align}
If $s<j<k$, the product defining $Z_{k,j}(q^{-m},q)$ contains
$\{0\}$ and is zero.  Substitution in
\eqref{eq:rank-uniform-E}, followed by the already constructed
color-$s$ expansion, yields
\[
 E_k(q^{-m},q)=0.
\]
Equation \eqref{eq:rank-uniform-E-even} gives the second zero
$E_k(-q^{-m},q)=0$.  After clearing a Laurent power of $A$, the
factor theorem therefore gives
\begin{equation}\label{eq:rank-uniform-m-factor}
 \{m;A\}\mid E_k(A,q)
 \qquad(k\le m\le2k-1).
\end{equation}

At $A=q$, every lower-order term in
\eqref{eq:rank-uniform-E} contains $\{-1;q\}=0$, while
\eqref{eq:rank-one-trivial} gives $\HH_k(K;q,q)=1$.  Thus
$E_k(q,q)=0$, and evenness also gives $E_k(-q,q)=0$.  Hence
\begin{equation}\label{eq:rank-uniform-minus-one-factor}
 \{-1;A\}\mid E_k(A,q).
\end{equation}

The ring $S$ is a UFD, and
\[
 \{m;A\}=A^{-1}q^m(A-q^{-m})(A+q^{-m}),
 \qquad
 \{-1;A\}=A^{-1}q^{-1}(A-q)(A+q).
\]
The relevant linear factors are $A\pm q^{-m}$ for
$k\le m\le2k-1$, together with $A\pm q$.  Each is prime in $S$: the
corresponding quotient identifies with the domain $\Z[q^{\pm1}]$ under
the appropriate evaluation of $A$.  These prime factors are pairwise
nonassociate.  It follows successively that their product divides $E_k$;
using \eqref{eq:rank-uniform-diagonal}, we obtain
$Z_{k,k}\mid E_k$ in $S$.  Define
\[
 G_k=E_k/Z_{k,k}\in S.
\]
The parities of the numerator and denominator give
\eqref{eq:general-G-parity}, closing the induction.  Finally, the
color-$k$ equation is the first equation in which $G_k$ occurs and
$Z_{k,k}\ne0$; triangularity proves uniqueness.
\end{proof}

For $k=1$, the uniqueness in \cref{thm:general-rank-uniform-DE} and
\eqref{eq:fundamental-factorization} give
\begin{equation}\label{eq:F-equals-G1}
 G_1(K;A,q)=F_K(A,q).
\end{equation}

\subsection{The integer-valued coefficient ring}

The symmetric Gaussian identity
\[
 \qbinom{r}{k}
 =\frac{\prod_{i=0}^{k-1}\{r-i\}}{\{k\}!}
\]
and the specialization $\{r+i;q^n\}=\{n+r+i\}$ compare
\eqref{eq:general-rank-uniform-DE} with
\eqref{eq:main-expansion}.  Uniqueness of the fixed-rank Newton
coefficients gives
\begin{equation}\label{eq:rank-uniform-H-comparison}
 \{k\}! H_k^{(n)}(K;q)
 =\{n-1\}G_k(K;q^n,q)
 \qquad(k\ge1).
\end{equation}

Recall the geometric integer-valued Laurent ring
\begin{equation}\label{eq:integer-valued-ring}
 \mathscr I_q^{\ge2}
 =\left\{f(A)\in\Bbbk[A^{\pm1}]:
 f(q^n)\in R\ \text{for every }n\ge2\right\}.
\end{equation}

\begin{corollary}
\label{cor:rank-uniform-Newton-coefficients}
For every zero-framed knot $K$, set
\[
 \mathsf H_0(K;A,q)=1,\qquad
 \mathsf H_k(K;A,q)
 =\frac{\{-1;A\}G_k(K;A,q)}{\{k\}!}\quad(k\ge1).
\]
Then $\mathsf H_k(K;A,q)\in\mathscr I_q^{\ge2}$ and
\[
 \mathsf H_k(K;q^n,q)=H_k^{(n)}(K;q)
 \qquad(n\ge2).
\]
For each $k\ge0$, the element $\mathsf H_k(K;A,q)$ is the unique
element of $\Bbbk[A^{\pm1}]$ with these specializations, and
\begin{equation}\label{eq:rank-uniform-Newton-parity}
 \mathsf H_k(K;-A,q)=(-1)^k\mathsf H_k(K;A,q).
\end{equation}
\end{corollary}

\begin{proof}
Equation \eqref{eq:rank-uniform-H-comparison} and \cref{thm:main} give the
specialization statement and the integer-valued property.  If two Laurent
polynomials in $A$ over $\Bbbk$ have the same values at all $A=q^n$, their
difference, after multiplication by a power of $A$, is a polynomial over
$\Bbbk$ with infinitely many distinct zeros.  It is therefore zero.
Equation
\eqref{eq:rank-uniform-Newton-parity} follows from
\eqref{eq:general-G-parity}.  Together with
\cref{thm:general-rank-uniform-DE}, this proves
\cref{thm:rank-uniform-intro}.
\end{proof}

The basis below is a geometric-node analogue of the $q$-binomial bases in
quantum integer-valued polynomial rings \cite{HH17} and of regular bases
arising from $P$-orderings \cite{Bhargava97}.  Since the present coefficient
ring and node set are different from those settings, we give a direct
argument.  The polynomial part of \eqref{eq:integer-valued-ring} has the
following explicit $R$-basis.  Define
\begin{equation}\label{eq:integer-valued-beta}
 \beta_j(A)=
 \frac{\prod_{i=0}^{j-1}(1-Aq^{-2-i})}{(q;q)_j},
 \qquad \beta_0=1.
\end{equation}

\begin{theorem}[Integer-valued Newton basis]
\label{thm:integer-valued-Newton-basis}
One has
\begin{align}
 \{f\in\Bbbk[A]:f(q^n)\in R\ (n\ge2)\}
 &=\bigoplus_{j\ge0}R\,\beta_j(A),
 \label{eq:integer-valued-polynomial-basis}\\
 \mathscr I_q^{\ge2}
 &=\bigcup_{M\ge0}A^{-M}
 \left(\bigoplus_{j\ge0}R\,\beta_j(A)\right).
 \label{eq:integer-valued-Laurent-basis}
\end{align}
\end{theorem}

\begin{proof}
At the node $A=q^{n+2}$,
\[
 \beta_j(q^{n+2})=
 \begin{cases}
 \displaystyle\frac{(1-q^n)(1-q^{n-1})\cdots
 (1-q^{n-j+1})}{(1-q)(1-q^2)\cdots(1-q^j)},&j\le n,\\[6pt]
 0,&j>n.
 \end{cases}
\]
This is the ordinary Gaussian coefficient
$\genfrac{[}{]}{0pt}{}{n}{j}_q$, and the diagonal value at $j=n$ is
$1$.  Since $\deg_A\beta_j=j$, the polynomials $\beta_0,\ldots,\beta_d$
form a $\Bbbk$-basis of the degree-$\le d$ polynomials.  Evaluation at
$A=q^2,q^3,\ldots,q^{d+2}$ gives a lower triangular matrix with diagonal
entries $1$.  Coefficients in $R$ therefore give node values in $R$, and
successive substitution in the opposite direction recovers every
coefficient in $R$ from integral node values.  This proves
\eqref{eq:integer-valued-polynomial-basis}.  Multiplying a Laurent
polynomial by a sufficiently large power of $A$ gives
\eqref{eq:integer-valued-Laurent-basis}, because every $q^n$ is a unit of
$R$.
\end{proof}

\subsection{The two-variable rank-uniform Newton element}

Define the universal color nodes
\begin{equation}\label{eq:two-variable-nodes}
 X_r(A,q)=Aq^{2r}+A^{-1}q^{-2r},\qquad r\ge0.
\end{equation}
They satisfy
\begin{equation}\label{eq:two-variable-node-difference}
 X_r(A,q)-X_i(A,q)=\{r-i\}\{r+i;A\}.
\end{equation}
Put
\begin{equation}\label{eq:two-variable-Newton-kernel}
 B_k(X;A,q)=\prod_{i=0}^{k-1}(X-X_i(A,q)),
 \qquad B_0=1,
\end{equation}
and define
\begin{equation}\label{eq:two-variable-completion}
 \widehat{\mathscr C}_{A,q}
 =\varprojlim_{m\ge1}
 \mathscr I_q^{\ge2}[X]/(B_m).
\end{equation}
The topology here is the Newton topology defined by the descending
ideals $(B_m)$.  Root-of-unity evaluation is performed only after fixing a
positive rank and a color: first set
\[
 A=q^n,\qquad X=X_r(q^n,q).
\]
The Newton series then truncates to a Laurent polynomial in $q$, which may
be evaluated at $q=\zeta$.

\begin{theorem}[Rank-uniform Newton element]
\label{thm:two-variable-unified-element}
For every zero-framed knot $K$, the series
\begin{equation}\label{eq:two-variable-unified-element}
 \mathscr H_K(X,A,q)=
 \sum_{k=0}^{\infty}
 B_k(X;A,q)\mathsf H_k(K;A,q)
\end{equation}
converges in $\widehat{\mathscr C}_{A,q}$ and has the following
properties.
\begin{enumerate}[label=\textup{(\roman*)},leftmargin=2.8em]
\item For every $r\ge0$,
\begin{equation}\label{eq:two-variable-color-evaluation}
 \mathscr H_K(X_r(A,q),A,q)=\HH_r(K;A,q).
\end{equation}
\item For every $n\ge2$, the specialization $A=q^n$ gives the
fixed-rank Newton series with coefficients $H_k^{(n)}$.  After the
substitution $X=q^nx^2+q^{-n}x^{-2}$, it is the colored lift
$\mathscr J_K^{(n)}(x,q)$ of
\eqref{eq:colored-lift}.
\item For every $n\ge2$, $r\ge0$, and complex root of unity $\zeta$,
there is a canonical finite evaluation
\begin{equation}\label{eq:two-variable-root-evaluation}
 \operatorname{ev}_{n,r,\zeta}(\mathscr H_K)
 =J_r^{SU(n)}(K;\zeta),
\end{equation}
where the right side means evaluation of the generic Laurent polynomial,
and $\operatorname{ev}_{n,r,\zeta}$ means first specializing $A=q^n$,
then $X=X_r(q^n,q)$, and finally $q=\zeta$.
\item The Newton coefficients in
\eqref{eq:two-variable-unified-element} are unique, and the full
family of color evaluations determines the element.
\end{enumerate}
\end{theorem}

\begin{proof}
For $k\ge m$, one has $B_m\mid B_k$, so modulo $(B_m)$ the series has
only its first $m$ terms.  These residues are compatible and prove
convergence.

At $X=X_r$, every term with $k>r$ vanishes.  For $k\le r$,
\eqref{eq:two-variable-node-difference} gives
\[
 B_k(X_r;A,q)
 =\prod_{i=0}^{k-1}\{r-i\}\{r+i;A\}.
\]
Substituting
$\mathsf H_k=\{-1;A\}G_k/\{k\}!$ turns the resulting finite sum into
\eqref{eq:general-rank-uniform-DE}, proving
\eqref{eq:two-variable-color-evaluation}.

By definition of $\mathscr I_q^{\ge2}$, evaluation at $A=q^n$ is a
ring homomorphism to $R$.  It sends the coefficient
$\mathsf H_k$ to $H_k^{(n)}$ and the universal Newton kernel to the
fixed-rank one.  If $X=Ax^2+A^{-1}x^{-2}$, direct factorization gives
\begin{equation}\label{eq:two-variable-x-factorization}
 B_k(X;A,q)=\prod_{i=0}^{k-1}
 (xq^{-i}-x^{-1}q^i)
 (Axq^i-A^{-1}x^{-1}q^{-i}).
\end{equation}
At $A=q^n$, this is $P_k^{(n)}(x,q)$, proving the second assertion.
Setting $x=q^r$ makes the series finite; its value is a Laurent
polynomial in $q$, which can then be evaluated at $q=\zeta$.  This
proves \eqref{eq:two-variable-root-evaluation} without dividing by a
possibly vanishing quantum dimension.

Finally, every $B_k$ is monic of degree $k$ in $X$, so
$B_0,\ldots,B_{m-1}$ form a free
$\mathscr I_q^{\ge2}$-basis of
$\mathscr I_q^{\ge2}[X]/(B_m)$.  Hence every inverse-limit element has
a unique Newton expansion.  If all color evaluations of an expansion
vanish, evaluate successively at $r=0,1,\ldots$.  At step $r$, all
earlier coefficients are already zero and the only new term is
$B_r(X_r)$ times the $r$th coefficient.  The coefficient ring is a
domain and $B_r(X_r)\ne0$, so induction makes every coefficient zero.
\end{proof}

\section{Examples}\label{sec:examples}

For the unknot $U$, one has $J_r^{SU(n)}(U;q)=1$ for every $r\ge0$.
By uniqueness of the Newton expansion,
$H_0^{(n)}(U;q)=1$ and $H_k^{(n)}(U;q)=0$ for $k\ge1$.

\subsection{Sharpness of the rank-uniform coefficient ring}

For the figure-eight knot, the double-twist formula
\cite[Theorem~1.2 and equation~(3.19)]{CLZ21}, compared with
\eqref{eq:general-rank-uniform-DE}, gives
\begin{equation}\label{eq:figure-eight-rank-uniform-coefficient}
 \mathsf H_k(4_1;A,q)=\frac{\{k-2;A\}_k}{\{k\}!}.
\end{equation}
For $k=1$ this is
\[
 \frac{Aq^{-1}-A^{-1}q}{q-q^{-1}},
\]
which does not belong to $\Z[A^{\pm1},q^{\pm1}]$: otherwise its numerator
would be divisible by $q-q^{-1}$, contrary to evaluation at $q=1$.

Put $Y=A^2$ and define $\gamma_0^{4_1}=1$ and,
for $k\ge1$,
\begin{equation}\label{eq:figure-eight-gamma}
 A^kq^{-2k}\mathsf H_k(4_1;A,q)
 =\gamma_k^{4_1}(Y)
 :=\frac{\prod_{m=-1}^{k-2}(1-YQ^m)}{(Q;Q)_k}.
\end{equation}
Then
\begin{equation}\label{eq:figure-eight-gamma-values}
 \gamma_k^{4_1}(Q^n)=\genfrac{[}{]}{0pt}{}{n+k-2}{k}_{Q}
 \qquad(n\ge2).
\end{equation}

\begin{proposition}[Sharp denominator growth]
\label{prop:figure-eight-sharp-denominators}
For every $d\ge0$, the polynomials
$\gamma_0^{4_1},\ldots,\gamma_d^{4_1}$ form an $R$-basis of the $R$-module of
polynomials in $\Bbbk[Y]$ of degree at most $d$ that take values in $R$
at every node
$Y=Q^n$, $n\ge2$.  The leading coefficient of $\gamma_k^{4_1}(Y)$ is
\begin{equation}\label{eq:figure-eight-leading-coefficient}
 \frac{(-1)^kQ^{k(k-3)/2}}{(Q;Q)_k}.
\end{equation}
In particular, it has a pole of order $k$ along $Q=1$, so the
quantum-factorial denominator length is unbounded.
\end{proposition}

\begin{proof}
Evaluate at $Y_n=Q^{n+2}$ for $0\le n\le d$.  By
\eqref{eq:figure-eight-gamma-values}, the $(n,k)$ entry of the evaluation
matrix is $\genfrac{[}{]}{0pt}{}{n+k}{k}_{Q}$.  Its determinant is the
Vandermonde determinant of the nodes multiplied by the leading
coefficients of the $\gamma_k^{4_1}$:
\[
 \prod_{0\le i<j\le d}(Q^{j+2}-Q^{i+2})
 \prod_{k=0}^{d}\frac{(-1)^kQ^{k(k-3)/2}}{(Q;Q)_k}.
\]
Writing $C=\binom{d+1}{2}$ and
$E=\sum_{0\le i<j\le d}(i+2)$, one has
\[
 \prod_{0\le i<j\le d}(Q^{j+2}-Q^{i+2})
 =(-1)^CQ^E\prod_{h=1}^{d}(1-Q^h)^{d+1-h},
\]
whereas
\[
 \prod_{k=0}^{d}(Q;Q)_k
 =\prod_{h=1}^{d}(1-Q^h)^{d+1-h}.
\]
The products cancel, leaving a signed power of $Q$, which is a unit in
$R$.  Thus the evaluation matrix is invertible over $R$ and the stated
basis property follows.  Reading the coefficient of $Y^k$ in
\eqref{eq:figure-eight-gamma} gives
\eqref{eq:figure-eight-leading-coefficient}.  Since each factor
$1-Q^a$ has a simple zero at $Q=1$, its denominator has order $k$ there.
\end{proof}

\subsection{Explicit fixed-rank coefficients for small twist knots}

We extract the Newton coefficients from the HOMFLY--PT formulas in
\cite[Theorem~1.2 and Section~3]{CLZ21}, using the notation
\eqref{eq:HOMFLY-differential-factor}.  Related rigorous formulas for the
trefoil, the figure-eight knot, and general twist knots are given in
\cite{Kawagoe25}.
The common $k$th kernel in those formulas is
\[
 \qbinom{r}{k}\{r+k-1;A\}_k\{k-2;A\}_k.
\]
For $n\ge2$ and $0\le k\le r$, specializing $A=q^n$ gives
\begin{align*}
 \qbinom{r}{k}
 &=\frac{\prod_{i=0}^{k-1}\{r-i\}}{\{k\}!},\\
 \{r+k-1;q^n\}_k
 &=\prod_{i=0}^{k-1}\{r+n+i\},\\
 \{k-2;q^n\}_k
 &=\prod_{j=0}^{k-1}\{n-1+j\}.
\end{align*}
Consequently,
\begin{align}
 \left.
 \qbinom{r}{k}\{r+k-1;A\}_k\{k-2;A\}_k
 \right|_{A=q^n}
 &=C_{r+1,k}^{(n)}
   \frac{\prod_{j=0}^{k-1}\{n-1+j\}}{\{k\}!}\notag\\
 &=C_{r+1,k}^{(n)}\qbinom{n+k-2}{k}.
 \label{eq:CLZ-kernel-extraction}
\end{align}
The last equality is the factorial definition of the symmetric Gaussian
coefficient after cancelling the common powers of $\{1\}$.  Thus the common
kernel contributes the color-independent factor
$\qbinom{n+k-2}{k}$.

The knot-dependent factors in
\cite[equations~(3.18)--(3.20) and~(3.22)]{CLZ21} are
\begin{align*}
 E_k^{4_1}(A,q)&=1,\\
 E_k^{3_1}(A,q)&=(-1)^kA^{2k}q^{k(k-1)},\\
 E_k^{5_2}(A,q)&=(-1)^kA^{4k}q^{3k(k-1)}
  \sum_{\ell=0}^k A^{-2\ell}q^{-3k\ell+\ell(\ell+2)}
  \qbinom{k}{\ell},\\
 E_k^{6_1}(A,q)&=A^{-2k}q^{-2k(k-1)}
  \sum_{\ell=0}^k A^{2\ell}q^{3k\ell-\ell(\ell+2)}
  \qbinom{k}{\ell}.
\end{align*}
Since the colored formula is the sum of $E_k^K(A,q)$ times the common
kernel, uniqueness of the Newton expansion gives
\[
 H_k^{(n)}(K;q)=E_k^K(q^n,q)\qbinom{n+k-2}{k}.
\]
In the normalization of Chen--Liu--Zhu \cite{CLZ15,CLZ21}, where $3_1$
denotes the left-handed trefoil, the corresponding Newton coefficients are
\begin{align}
 H_k^{(n)}(4_1;q)
 &=\qbinom{n+k-2}{k},
 \label{eq:example-41}\\
 H_k^{(n)}(3_1;q)
 &=(-1)^kq^{k(2n+k-1)}\qbinom{n+k-2}{k}.
 \label{eq:example-31}
\end{align}
For the twist knots $5_2$ and $6_1$, the double-twist formulas in
\cite{CLZ21} give
\begin{align}
 H_k^{(n)}(5_2;q)
 &=(-1)^kq^{4nk+3k(k-1)}\qbinom{n+k-2}{k}
 \sum_{\ell=0}^{k}
 q^{-2n\ell-3k\ell+\ell(\ell+2)}\qbinom{k}{\ell},
 \label{eq:example-52}\\
 H_k^{(n)}(6_1;q)
 &=q^{-2nk-2k(k-1)}\qbinom{n+k-2}{k}
 \sum_{\ell=0}^{k}
 q^{2n\ell+3k\ell-\ell(\ell+2)}\qbinom{k}{\ell}.
 \label{eq:example-61}
\end{align}
Since every Gaussian coefficient appearing above belongs to
$\Z[q^{\pm1}]$ and the inner sums are finite, all four expressions lie in
$\Z[q^{\pm1}]$.

\section{Conclusion and further questions}
\label{sec:outlook}

For each fixed $n\ge2$, completed Harish--Chandra reflection and integral
local descent convert the one-sided interpolation expansion into the
two-sided Chen--Liu--Zhu expansion, proving the conjecture for every
zero-framed knot.  Across all positive ranks, the Laurent differential
coefficients give integer-valued Newton coefficients in
$\mathscr I_q^{\ge2}$ and a single two-variable Newton inverse-limit
element.  Fixing rank and color makes the expansion finite before any
root-of-unity evaluation.  The figure-eight example shows both that the
coefficient ring cannot generally be replaced by
$\Z[A^{\pm1},q^{\pm1}]$ and that the quantum-factorial denominator length
is unbounded.

Two structural questions remain.  First, does
$\widehat{\mathscr C}_{A,q}$ admit a canonical realization as a one-row
stable quotient of the Beliakova--Gorsky interpolation center, with the
basis \eqref{eq:integer-valued-beta} induced by a stable central basis?
Such a result would compare the two constructions of the unified object.
Second, rectangular colors lead to a multivariable Harish--Chandra
quotient with a stratified ramification locus; compare the explicit
conjectural double-twist formulas in \cite{KNTZ20}.  An integral
formal-descent theorem for that quotient would extend the one-row method
developed here.

\appendix
\section{Comparison of the colored normalization conventions}
\label{app:CLZ-normalization}

This appendix proves the convention comparisons used in
\cref{prop:CLZ-normalization,prop:all-partition-normalization}.

\begin{proof}[Proofs of \cref{prop:CLZ-normalization,prop:all-partition-normalization}]
Write $A$ for the variable denoted by $t$ in \cite{CLZ15} and by $a$ in
\cite{CLZ21}.  For a partition $\lambda$, the standard reduced colored
HOMFLY--PT invariant is
\[
 \HH_\lambda(K;A,q)
 =\frac{W_\lambda(K;q,A)}{S_\lambda(q,A)}
 =\frac{W_\lambda(K;q,A)}{W_\lambda(U;q,A)}.
\]
For the one-row partition, Chen--Liu--Zhu write
\[
 P_r^{\mathrm{CLZ}}(K;q,A)
 =\frac{W_{(r)}(K;q,A)}{S_{(r)}(q,A)}
 =\frac{W_{(r)}(K;q,A)}{W_{(r)}(U;q,A)},
 \qquad
 J_{r,\mathrm{CLZ}}^{SU(n)}(K;q)
 =P_r^{\mathrm{CLZ}}(K;q,q^n);
\]
see \cite[(2.2) and (2.4)]{CLZ15} and
\cite[Introduction and (1.1)]{CLZ21}.
For the empty partition, all invariants under comparison equal $1$, so
assume $|\lambda|\ge1$; the one-row conclusion then includes every
$r\ge1$.

\medskip
\noindent\emph{Skein and crossing conventions.}
The comparison between the Hecke and quantum-group constructions is
developed in \cite[Sections~3--4, especially Theorem~4.3 and
Definition~4.4]{LZ10}.  In the variables used there, the substitution
relevant here is
\[
 t_{\mathrm{LZ}}^{1/2}=q^{-1},
 \qquad
 \nu_{\mathrm{LZ}}^{1/2}=A^{-1}.
\]
Lin--Zheng use the skein relation
\[
 \nu_{\mathrm{LZ}}^{-1/2}P_{L_+}
 -\nu_{\mathrm{LZ}}^{1/2}P_{L_-}
 =(t_{\mathrm{LZ}}^{-1/2}-t_{\mathrm{LZ}}^{1/2})P_{L_0}.
\]
Under the displayed substitution it becomes
\[
 A P_{L_+}-A^{-1}P_{L_-}=(q-q^{-1})P_{L_0},
\]
which is precisely the crossing convention in
\cite[Section~2.1]{CLZ15}.  In particular, $L_+$ and $L_-$ retain their
labels, so the crossing and quantum-group parameter conventions agree.
The ordinary polynomial $P$
in \cite{LZ10} is normalized to be $1$ on the unknot, whereas the
framed polynomial $\mathcal H$ in \cite{CLZ15} has unknot value
$(A-A^{-1})/(q-q^{-1})$.  Thus the normalized object to compare with
Chen--Liu--Zhu is Lin--Zheng's colored invariant $W$ from
\cite[Definition~4.4]{LZ10}, not their ordinary polynomial $P$.

\medskip
\noindent\emph{Color, twist, and framing.}
The braid generator for a positive crossing in
\cite[Section~4]{LZ10} is the standard positive
$U_q(\mathfrak{sl}_n)$ braiding.  If $m=|\lambda|$ and
$\ell(\lambda)\le n$, a primitive Hecke idempotent $p_\lambda\in H_m(q)$
selects the simple polynomial module whose $\mathfrak{gl}_n$ highest
weight is $\lambda$.
Let $\mathcal D$ be a diagram of $K$, and let $w(\mathcal D)$ be its writhe.
Lin--Zheng remove blackboard framing by the factor $\theta_V^{-w(\mathcal D)}$ in
\cite[equation~(2.12)]{LZ10}.  Their
\cite[Definition~4.4]{LZ10} gives
\[
 W_{L;\lambda^1,\ldots,\lambda^\ell}
 \big|_{t_{\mathrm{LZ}}^{1/2}=q^{-1},\,
          \nu_{\mathrm{LZ}}^{1/2}=q^{-n}}
 =q^{\frac{2}{n}\sum_{i<j}|\lambda^i||\lambda^j|
                   \operatorname{lk}(L_i,L_j)}
   I_{L;V_{\lambda^1},\ldots,V_{\lambda^\ell}}.
\]
Here $I_{L;V_{\lambda^1},\ldots,V_{\lambda^\ell}}$ denotes the
Reshetikhin--Turaev invariant in the convention of \cite{LZ10}.
For a knot the linking factor is $1$.  Moreover,
\cite[Corollary~4.5]{LZ10} gives the writhe factor
\[
 t_{\mathrm{LZ}}^{\kappa_\lambda w(\mathcal D)/2}
 \nu_{\mathrm{LZ}}^{|\lambda|w(\mathcal D)/2}
 =q^{-\kappa_\lambda w(\mathcal D)}A^{-|\lambda|w(\mathcal D)},
 \qquad
 \kappa_\lambda=\sum_i\lambda_i(\lambda_i-2i+1).
\]
For $\lambda=(r)$ this is exactly the correction in
\cite[Section~2.1]{CLZ15}.  The $U_q(\mathfrak{sl}_n)$ twist in
\cite[Theorem~4.1]{LZ10} contains the scalar-summand correction; the
computation following \cite[equation~(4.21)]{LZ10}
combines this term with the cabling contribution and leaves precisely the
displayed intercomponent-linking factor, which equals one for a knot.

In our square-parameter convention, the module just obtained is the
restriction of the $U_{q^2}(\mathfrak{gl}_n)$-module $V(\lambda)$; see
\eqref{eq:parameter-translation}.  A direct calculation from the ribbon
data in
\cite[Section~3G2, equations~(70)--(72)]{HL16},
with the full $\mathfrak{gl}_n$ Cartan factor, gives the positive twist
\[
 \theta_\lambda
 =v^{n|\lambda|}\mfrakq^{c(\lambda)}
 =q^{n|\lambda|+\kappa_\lambda},
 \qquad c(\lambda)=\frac12\kappa_\lambda.
\]
After $A=q^n$, the displayed HOMFLY--PT correction is precisely
$\theta_\lambda^{-w(\mathcal D)}$.  This confirms the braiding, twist, and
zero-framing conventions for every partition; for
$\lambda=(r)$ one has $\kappa_{(r)}=r(r-1)$.

\medskip
\noindent\emph{Closure traces and normalization.}
We next compare the closure traces.  Let $V$ be the defining module, with
ordered basis $v_1,\ldots,v_n$, and set $Jv_i=v_{n+1-i}$.  In the Hopf
convention
\eqref{eq:Hopf}, use the defining representation of
\cite[Section~3.1]{BG24} and the Hopf-compatible normalized positive
braiding fixed above; equivalently, this is
$\mathcal R=D\Theta_{\mathrm{HL}}^{-1}$ in
the notation of \cite[Section~3G2, equation~(70)]{HL16}.  Its Hecke
operator is
\[
 g_{\mathrm{HL}}(v_i\otimes v_j)=
 \begin{cases}
  qv_i\otimes v_i,&i=j,\\
  v_j\otimes v_i+(q-q^{-1})v_i\otimes v_j,&i<j,\\
  v_j\otimes v_i,&i>j.
 \end{cases}
\]
After the factor $q^{1/n}$ in
\cite[equations~(4.6)--(4.8)]{LZ10} converts the
$\mathfrak{sl}_n$ braiding to the $\mathfrak{gl}_n$ Hecke operator, the
Lin--Zheng operator is
\[
 g_{\mathrm{LZ}}(v_i\otimes v_j)=
 \begin{cases}
  qv_i\otimes v_i,&i=j,\\
  v_j\otimes v_i,&i<j,\\
  v_j\otimes v_i+(q-q^{-1})v_i\otimes v_j,&i>j.
 \end{cases}
\]
Thus both operators satisfy $(g-q)(g+q^{-1})=0$, but their matrices are
related by basis reversal rather than equality:
\begin{equation}\label{eq:normalization-Hecke-conjugacy}
 g_{\mathrm{LZ}}
 =(J\otimes J)g_{\mathrm{HL}}(J\otimes J)^{-1}.
\end{equation}
If $\rho_{\mathrm{HL}},\rho_{\mathrm{LZ}}:H_m(q)\to
\operatorname{End}(V^{\otimes m})$ are the resulting representations of
the abstract type-$A$ Hecke algebra, then
\begin{equation}\label{eq:normalization-Hecke-representation-conjugacy}
 \rho_{\mathrm{LZ}}(x)
 =J^{\otimes m}\rho_{\mathrm{HL}}(x)(J^{-1})^{\otimes m}
 \qquad(x\in H_m(q)).
\end{equation}

The definition of $K_{\pm2\rho}$ gives
\begin{equation}\label{eq:normalization-pivotal-conjugacy}
 K_{2\rho}=JK_{-2\rho}J^{-1}
 \quad\text{on }V.
\end{equation}
For $m\ge1$, the Lin--Zheng and Habiro--L\^e blackboard-framed closure
functionals are therefore, respectively,
\[
 \tau_m^+(x)
 =\operatorname{Tr}_{V^{\otimes m}}
   \bigl(K_{2\rho}^{\otimes m}\rho_{\mathrm{LZ}}(x)\bigr),
 \qquad
 \tau_m^-(x)
 =\operatorname{Tr}_{V^{\otimes m}}
   \bigl(K_{-2\rho}^{\otimes m}\rho_{\mathrm{HL}}(x)\bigr).
\]
Equations \eqref{eq:normalization-Hecke-representation-conjugacy} and
\eqref{eq:normalization-pivotal-conjugacy}, followed by invariance of the
ordinary trace under conjugacy, give
\begin{equation}\label{eq:normalization-closure-trace-equality}
 \tau_m^+(x)=\tau_m^-(x)
 \qquad(x\in H_m(q)).
\end{equation}
In particular, their common fundamental unknot value is
\begin{equation}\label{eq:two-fundamental-quantum-dimensions}
 \operatorname{Tr}_V(K_{2\rho})
 =\operatorname{Tr}_V(K_{-2\rho})
 =\frac{q^n-q^{-n}}{q-q^{-1}}.
\end{equation}

For every partition $\lambda\vdash m$, both constructions use the same
abstract Hecke idempotent $p_\lambda\in H_m(q)$; see
\cite[Lemma~3.3 and equation~(4.11)]{LZ10} and
\cite{AM98,Lukac05,ML03}.  Equation
\eqref{eq:normalization-Hecke-representation-conjugacy} conjugates the two
Hecke actions, while \eqref{eq:normalization-closure-trace-equality}
identifies their closure traces.  Hence the unreduced $\lambda$-colored
invariants and the corresponding unknot values coincide.  Since both
conventions use the framing correction $\theta_\lambda^{-w(\mathcal D)}$, their
reduced invariants coincide as well.

The normalizing denominator agrees with our quantum dimension.  Indeed,
the Frobenius formula in
\cite[Section~2.1]{CLZ15}, together with
\[
 \frac{q^{nm}-q^{-nm}}{q^m-q^{-m}}
 =\sum_{i=1}^{n}q^{m(n+1-2i)},
\]
gives the principal specialization
\[
 S_\lambda(q,q^n)
 =s_\lambda(q^{n-1},q^{n-3},\ldots,q^{1-n})
 =\dim_qV(\lambda).
\]
Here $K_{-2\rho}$ has eigenvalues
$q^{1-n},q^{3-n},\ldots,q^{n-1}$ on the defining module; symmetry of the
Schur polynomial permits the reversed order displayed above.  Dividing
the quantum trace by the unknot value therefore gives exactly the
reduced trace convention in \eqref{eq:reduced-normalization}.

By centrality and \eqref{eq:reduced-normalization}, the scalar colored
invariant in our convention is the quantum trace quotient normalized to
$1$ on the unknot.  The preceding Hecke calculation therefore proves
\[
 \HH_\lambda(K;q^n,q)=J_K(V(\lambda);q^2)
\]
for every $\ell(\lambda)\le n$, which is
\eqref{eq:all-partition-normalization}.  When $\lambda=(r)$,
$\HH_{(r)}=P_r^{\mathrm{CLZ}}$ and the right side is
\eqref{eq:SU-normalization}; this also proves
\cref{prop:CLZ-normalization}.
\end{proof}

\section{The Newton transform and \texorpdfstring{$q$}{q}-holonomicity}\label{app:qholonomic}

\begin{definition}
\label{def:q-holonomic}
For a sequence $f=(f_r)_{r\ge0}$ with values in $\Bbbk=\Q(q)$, let
\[
 (Lf)_r=f_{r+1},\qquad (Mf)_r=q^rf_r;
\]
then $LM=qML$.  In one discrete variable, the standard
$q$-Weyl-algebra definition of $q$-holonomicity is equivalent to the
existence of a nonzero recurrence
\begin{equation}\label{eq:q-holonomic-recurrence}
 \sum_{a=0}^{d}p_a(q,q^r)f_{r+a}=0,
 \qquad p_a(q,M)\in\Q(q,M),
\end{equation}
after clearing denominators if necessary.  This equivalence and the
multivariate $q$-Weyl-algebra definition used below are recalled in
\cite[Section~3]{GLsurvey}.
\end{definition}

\begin{proposition}[Newton transforms preserve $q$-holonomicity]
\label{prop:qholonomic-appendix}
Fix $n\ge2$.  If $f_r\in\Bbbk$ is $q$-holonomic in $r$, then its
Newton coefficient sequence $h_k$ at the nodes $X_r^{(n)}$ is
$q$-holonomic in $k$.
\end{proposition}

\begin{proof}
By \eqref{eq:newton-barycentric} and \eqref{eq:Dkj},
\[
 h_k=\sum_{j=0}^{k}\frac{f_j}{D_{k,j}^{(n)}}.
\]
Define the kernel on $\mathbb N^2$ by
\[
 K(k,j)=
 \begin{cases}
  D_{k,j}^{(n)\,-1},&0\le j\le k,\\
  0,&j>k.
 \end{cases}
\]
On its triangular support, \eqref{eq:Dkj} gives
\begin{equation}\label{eq:qholonomic-kernel-factorial}
 K(k,j)=(-1)^{k-j}
 \frac{\{n+j-1\}!\{n+2j\}}
 {\{j\}!\{k-j\}!\{n+j+k\}!}.
\end{equation}
The identity
\[
 \{m\}!=(-1)^m q^{-m(m+1)/2}(q^2;q^2)_m
\]
expresses every factorial in \eqref{eq:qholonomic-kernel-factorial} in
terms of a $q$-Pochhammer symbol and a quadratic monomial.  To make the
triangular support precise, define a function on $\Z$ by
\[
 b(m)=
 \begin{cases}
  (q^2;q^2)_m^{-1},&m\ge0,\\
  0,&m<0.
 \end{cases}
\]
Examples~2.2(e), Lemma~2.5(b), and Theorem~5.2(f) of
\cite{GLsurvey} show that $b$ is $q$-holonomic on $\Z$.
Theorem~5.2(c) there then shows that $(k,j)\mapsto b(k-j)$ is
$q$-holonomic on $\Z^2$, and Proposition~5.4(a) permits restriction to
$\mathbb N^2$.  This factor vanishes exactly when $j>k$ and therefore
supplies the zero extension in the definition of $K(k,j)$.

The remaining factorial factors in
\eqref{eq:qholonomic-kernel-factorial} satisfy first-order
$q$-hypergeometric recurrences on their natural supports and are therefore
$q$-holonomic.  Signs, quadratic monomials, affine pullback, and products
preserve $q$-holonomicity
\cite[Examples~2.2(a)--(b),(d)--(e), Lemma~2.5,
Theorem~5.2(a)--(c),(f), and Proposition~5.4(a)]{GLsurvey}.
Hence $K(k,j)$ is $q$-holonomic on $\mathbb N^2$, including the boundary
$j=k$ and the stipulated zero extension.

The pullback $(k,j)\mapsto f_j$ is $q$-holonomic.  Since
$K(k,j)=0$ for $j>k$, the summand has finite support in $j$ for every
fixed $k$; product closure followed by finite-sum closure therefore gives
\[
 h_k=\sum_{j\ge0}f_jK(k,j)
     =\sum_{j=0}^{k}\frac{f_j}{D_{k,j}^{(n)}}
 \]
as a $q$-holonomic function of $k$
\cite[Theorems~5.2--5.3 and Corollary~5.5]{GLsurvey}.
\end{proof}

\begin{corollary}
\label{cor:qholonomic-coefficients}
For every zero-framed knot $K$ and every fixed $n\ge2$, the sequence
$k\mapsto H_k^{(n)}(K;q)$ is $q$-holonomic.
\end{corollary}

\begin{proof}
For fixed $n$, the unreduced sequence
$r\mapsto W_K(h_r;q^n,q)$ is $q$-holonomic by
\cite[Corollary~1.2 and Remark~1.3]{GLL18}.  In the normalization fixed in
\cref{prop:CLZ-normalization},
\[
 J_r^{SU(n)}(K;q)
 =\frac{W_K(h_r;q^n,q)}{\qbinom{n+r-1}{r}}.
\]
The reciprocal quantum dimension is $q$-hypergeometric, since the ratio of
successive terms is $[r+1]_q/[n+r]_q$.  Product closure therefore makes
$r\mapsto J_r^{SU(n)}(K;q)$ $q$-holonomic, and
\cref{prop:qholonomic-appendix} gives the result.
\end{proof}

\begingroup
\setlength{\emergencystretch}{2em}

\endgroup

\vspace{1.25em}
\begingroup
\small
\noindent\textsc{Honghuai Fang}\par
\noindent Institute for Theoretical Sciences, Westlake University,\par
\noindent No.~600 Dunyu Road, Xihu District, Hangzhou, Zhejiang 310030, China\par
\noindent\textit{Email address:} \texttt{fanghonghuai@westlake.edu.cn}\par

\medskip
\noindent\textsc{Tian Zhou}\par
\noindent Peking University,\par
\noindent No.~5 Yiheyuan Road, Haidian District, Beijing 100871, China\par
\noindent\textit{Email address:} \texttt{2201110034@pku.edu.cn}\par
\endgroup


\begin{thebibliography}{99}

\bibitem{AM98}
A.~K. Aiston and H.~R. Morton,
\emph{Idempotents of Hecke algebras of type $A$},
J. Knot Theory Ramifications \textbf{7} (1998), no.~4, 463--487,
\href{https://doi.org/10.1142/S0218216598000243}
{doi:10.1142/S0218216598000243}.

\bibitem{BG24}
A.~Beliakova and E.~Gorsky,
\emph{Cyclotomic expansions for $\mathfrak{gl}_N$ link invariants via interpolation Macdonald polynomials},
Selecta Math. (N.S.) \textbf{30} (2024), Paper No.~101,
\href{https://doi.org/10.1007/s00029-024-00990-y}{doi:10.1007/s00029-024-00990-y}.

\bibitem{Bhargava97}
M.~Bhargava,
\emph{$P$-orderings and polynomial functions on arbitrary subsets of
Dedekind rings},
J. Reine Angew. Math. \textbf{490} (1997), 101--128,
\href{https://doi.org/10.1515/crll.1997.490.101}
{doi:10.1515/crll.1997.490.101}.

\bibitem{CLZ15}
Q.~Chen, K.~Liu, and S.~Zhu,
\emph{Volume conjecture for $SU(n)$-invariants},
\href{https://arxiv.org/abs/1511.00658}{arXiv:1511.00658 [math.QA]}.

\bibitem{CLZ21}
Q.~Chen, K.~Liu, and S.~Zhu,
\emph{Cyclotomic expansions for the colored HOMFLY--PT invariants of double twist knots},
\href{https://arxiv.org/abs/2110.03616}{arXiv:2110.03616 [math.GT]}.

\bibitem{CLPZ23}
Q.~Chen, K.~Liu, P.~Peng, and S.~Zhu,
\emph{Congruence skein relations for colored HOMFLY--PT invariants},
Comm. Math. Phys. \textbf{400} (2023), 683--729,
\href{https://doi.org/10.1007/s00220-022-04604-6}
{doi:10.1007/s00220-022-04604-6}.

\bibitem{FYHLMO85}
P.~Freyd, D.~Yetter, J.~Hoste, W.~B.~R. Lickorish, K.~Millett, and A.~Ocneanu,
\emph{A new polynomial invariant of knots and links},
Bull. Amer. Math. Soc. (N.S.) \textbf{12} (1985), no.~2, 239--246,
\href{https://doi.org/10.1090/S0273-0979-1985-15361-3}{doi:10.1090/S0273-0979-1985-15361-3}.

\bibitem{GL05}
S.~Garoufalidis and T.~T.~Q. L\^e,
\emph{The colored Jones function is $q$-holonomic},
Geom. Topol. \textbf{9} (2005), 1253--1293,
\href{https://doi.org/10.2140/gt.2005.9.1253}{doi:10.2140/gt.2005.9.1253}.

\bibitem{GLsurvey}
S.~Garoufalidis and T.~T.~Q. L\^e,
\emph{A survey of $q$-holonomic functions},
Enseign. Math. \textbf{62} (2016), no.~3--4, 501--525,
\href{https://doi.org/10.4171/LEM/62-3/4-7}{doi:10.4171/LEM/62-3/4-7}.

\bibitem{GLL18}
S.~Garoufalidis, A.~D. Lauda, and T.~T.~Q. L\^e,
\emph{The colored HOMFLYPT function is $q$-holonomic},
Duke Math. J. \textbf{167} (2018), no.~3, 397--447,
\href{https://doi.org/10.1215/00127094-2017-0030}{doi:10.1215/00127094-2017-0030}.

\bibitem{Habiro02}
K.~Habiro,
\emph{On the quantum $\mathfrak{sl}_2$ invariants of knots and integral homology spheres},
in \emph{Invariants of Knots and $3$-Manifolds (Kyoto, 2001)},
Geom. Topol. Monogr. \textbf{4} (2002), 55--68,
\href{https://doi.org/10.2140/gtm.2002.4.55}{doi:10.2140/gtm.2002.4.55}.

\bibitem{Habiro04}
K.~Habiro,
\emph{Cyclotomic completions of polynomial rings},
Publ. Res. Inst. Math. Sci. \textbf{40} (2004), no.~4, 1127--1146,
\href{https://doi.org/10.2977/PRIMS/1145475444}{doi:10.2977/PRIMS/1145475444}.

\bibitem{Habiro06}
K.~Habiro,
\emph{Bottom tangles and universal invariants},
Algebraic \& Geometric Topology \textbf{6} (2006), 1113--1214,
\href{https://doi.org/10.2140/agt.2006.6.1113}{doi:10.2140/agt.2006.6.1113}.

\bibitem{HL16}
K.~Habiro and T.~T.~Q. L\^e,
\emph{Unified quantum invariants for integral homology spheres associated
with simple Lie algebras},
Geom. Topol. \textbf{20} (2016), no.~5, 2687--2835,
\href{https://doi.org/10.2140/gt.2016.20.2687}
{doi:10.2140/gt.2016.20.2687}.

\bibitem{HH17}
N.~Harman and S.~Hopkins,
\emph{Quantum integer-valued polynomials},
J. Algebraic Combin. \textbf{45} (2017), no.~2, 601--628,
\href{https://doi.org/10.1007/s10801-016-0717-3}
{doi:10.1007/s10801-016-0717-3}.

\bibitem{Jantzen96}
J.~C. Jantzen,
\emph{Lectures on Quantum Groups},
Graduate Studies in Mathematics, vol.~6, American Mathematical Society, 1996.

\bibitem{KNTZ20}
M.~Kameyama, S.~Nawata, R.~Tao, and H.~D.~Zhang,
\emph{Cyclotomic expansions of HOMFLY--PT colored by rectangular
Young diagrams},
Lett. Math. Phys. \textbf{110} (2020), 2573--2583,
\href{https://doi.org/10.1007/s11005-020-01318-5}
{doi:10.1007/s11005-020-01318-5}.

\bibitem{Kassel95}
C.~Kassel,
\emph{Quantum Groups},
Graduate Texts in Mathematics, vol.~155, Springer, 1995.

\bibitem{Kawagoe25}
K.~Kawagoe,
\emph{The colored HOMFLY--PT polynomials of the trefoil knot, the
figure-eight knot, and twist knots},
J. Geom. Phys. \textbf{213} (2025), Paper No.~105488,
\href{https://doi.org/10.1016/j.geomphys.2025.105488}
{doi:10.1016/j.geomphys.2025.105488}.

\bibitem{LZ10}
X.-S.~Lin and H.~Zheng,
\emph{On the Hecke algebras and the colored HOMFLY polynomial},
Trans. Amer. Math. Soc. \textbf{362} (2010), no.~1, 1--18,
\href{https://doi.org/10.1090/S0002-9947-09-04691-1}
{doi:10.1090/S0002-9947-09-04691-1}.

\bibitem{Lukac05}
S.~G. Lukac,
\emph{Idempotents of the Hecke algebra become Schur functions in the
skein of the annulus},
Math. Proc. Cambridge Philos. Soc. \textbf{138} (2005), no.~1, 79--96,
\href{https://doi.org/10.1017/S0305004104007984}
{doi:10.1017/S0305004104007984}.

\bibitem{ML03}
H.~R. Morton and S.~G. Lukac,
\emph{The Homfly polynomial of the decorated Hopf link},
J. Knot Theory Ramifications \textbf{12} (2003), no.~3, 395--416,
\href{https://doi.org/10.1142/S0218216503002536}
{doi:10.1142/S0218216503002536}.

\bibitem{Morton07}
H.~R. Morton,
\emph{Integrality of Homfly 1-tangle invariants},
Algebr. Geom. Topol. \textbf{7} (2007), 327--338,
\href{https://doi.org/10.2140/agt.2007.7.327}
{doi:10.2140/agt.2007.7.327}.

\bibitem{Okounkov98}
A.~Okounkov,
\emph{On Newton interpolation of symmetric functions: a characterization
of interpolation Macdonald polynomials},
Adv. in Appl. Math. \textbf{20} (1998), no.~4, 395--428,
\href{https://doi.org/10.1006/aama.1998.0590}
{doi:10.1006/aama.1998.0590}.

\bibitem{PT87}
J.~H. Przytycki and P.~Traczyk,
\emph{Invariants of links of Conway type},
Kobe J. Math. \textbf{4} (1987), no.~2, 115--139.

\bibitem{RT90}
N.~Yu. Reshetikhin and V.~G. Turaev,
\emph{Ribbon graphs and their invariants derived from quantum groups},
Comm. Math. Phys. \textbf{127} (1990), no.~1, 1--26,
\href{https://doi.org/10.1007/BF02096491}
{doi:10.1007/BF02096491}.

\bibitem{Sahi96}
S.~Sahi,
\emph{Interpolation, integrality, and a generalization of Macdonald's
polynomials},
Internat. Math. Res. Notices \textbf{1996} (1996), no.~10, 457--471,
\href{https://doi.org/10.1155/S107379289600030X}
{doi:10.1155/S107379289600030X}.

\bibitem{Zhu23}
S.~Zhu,
\emph{New structures for colored HOMFLY--PT invariants},
Sci. China Math. \textbf{66} (2023), no.~2, 341--392,
\href{https://doi.org/10.1007/s11425-021-1951-7}
{doi:10.1007/s11425-021-1951-7}.

\end{thebibliography}
\end{document}